\documentclass[11pt,a4paper]{amsart}
 \usepackage[]{amsmath, amsthm, amsfonts, amssymb, amscd,mathtools}
\usepackage[mathscr]{eucal}
\usepackage{hyperref}
\usepackage[all]{xy}
\usepackage[usenames,dvipsnames]{color}
\usepackage{todonotes}
\usepackage{rotating}
\usepackage[shortlabels]{enumitem}
\usepackage{tikz}

\setlist[enumerate]{label=\it{(\roman*)},ref=\it{(\roman*)}}

\usepackage{dsfont}

\usepackage{blindtext}

\usepackage{stackengine}
\stackMath

\makeatletter
\@namedef{subjclassname@2020}{
  \textup{2020} Mathematics Subject Classification}
\makeatother

\newcommand{\Betti}{\text{\rm B}}
\newcommand{\dR}{\text{\rm dR}}

\newcommand{\A}{\mathds{A}}

\newcommand{\C}{\mathds{C}}

\newcommand{\G}{\mathds{G}}

\newcommand{\N}{\mathds{N}}
\renewcommand{\P}{\mathds{P}}
\newcommand{\Q}{\mathds{Q}}
\newcommand{\R}{\mathds{R}}

\newcommand{\Z}{\mathds{Z}}

\newcommand{\caE}{\mathcal{E}}

\newcommand{\caH}{\mathcal{H}}

\newcommand{\caL}{\mathcal{L}}

\newcommand{\caO}{\mathcal{O}}
\newcommand{\caP}{\mathcal{P}}

\newcommand{\caX}{\mathcal{X}}

\newcommand{\fD}{\mathfrak{D}}

\newcommand{\sE}{\mathscr{E}}

\newcommand{\bfone}{{\mathds{1}}}

\newcommand{\bfdelta}{{\boldsymbol{\delta}}}

   \DeclareMathOperator {\Gr} {Gr}
   \DeclareMathOperator {\Alt} {Alt}

 \DeclareMathOperator {\Div}{div}
 \DeclareMathOperator{\supp}{supp}
 \DeclareMathOperator{\codim}{codim}
 \DeclareMathOperator{\im}{Im}
 \DeclareMathOperator{\re}{Re}
 \DeclareMathOperator{\si}{si}
 \DeclareMathOperator{\rd}{rd}
 \DeclareMathOperator {\cl}{cl}
 \DeclareMathOperator {\CH}{CH} 
 \DeclareMathOperator {\Tr}{Tr} 
  
\DeclareMathOperator{\Spec}{Spec}

\DeclareMathOperator{\Ext}{Ext}

\DeclareMathOperator{\Id}{Id}

\DeclareMathOperator{\GL}{GL}
\DeclareMathOperator{\gl}{\mathfrak{gl}}

\DeclareMathOperator{\Ad}{Ad}

\DeclareMathOperator{\Ht}{ht}

\DeclareMathOperator{\Li}{Li}
\DeclareMathOperator{\rH}{H}

\DeclareMathOperator{\diag}{diag}

\newcommand{\RMHS}{\R-\mathbf{MHS}}

\newcommand{\rel}{\text{\rm rel}}
\newcommand{\exc}{\text{\rm exc}}

\numberwithin{equation}{section}

\theoremstyle{plain}
\newtheorem{prop}{Proposition}[section]

\newtheorem{cor}[prop]{Corollary}
\newtheorem{lem}[prop]{Lemma}
\newtheorem{thm}[prop]{Theorem}

\newtheorem*{mainthm}{Main Theorem}

\theoremstyle{definition}
\newtheorem*{dfn}{Definition}
\newtheorem{df}[prop]{Definition}

\newtheorem{rmk}[prop]{Remark}
\newtheorem{ex}[prop]{Example}

\begin{document}

\renewcommand\stackalignment{l}

\title{Height pairing on higher cycles and mixed Hodge structures II}

\author{J.~I. Burgos~Gil}
\thanks{J.\,I.~Burgos Gil was partially supported by research projects
  PID2019-108936GB-C21, 
 CEX2019-000904-S and CEX2023-001347-S financed
by  MICIU/AEI/10.13039/501100011033 and project  PID2022-142024NB-I00 financed
by  MICIU/AEI/10.13039/501100011033  and FEDER Una manera de hacer Europa.}
\address{Instituto de Ciencias Matem\'aticas (CSIC-UAM-UCM-UCM3).
  Calle Nicol\'as Ca\-bre\-ra~15, Campus UAM, Cantoblanco, 28049 Madrid,
  Spain} 
\email{burgos@icmat.es}

\author{S.~Goswami}
\thanks{S.~Goswami was supported by the Universitat de Barcelona Mar\'ia Zambrano postdoctoral fellowship, and partially supported by Spanish MICINN research project PID2023-147642NB-I00}

\address{Department of Mathematics and Informatics, University of Barcelona, Gran Via Corts Catalanes, 585, L'Eixample, 08007 Barcelona, Spain, and 3 Beni Master Lane, 700061 Kolkata, WB, India}
\email{gossouvik@gmail.com}

\author{G.~Pearlstein}

\address{Department of Mathematics, University of Pisa, Largo Bruno Pontecorvo, 5, 56127 Pisa PI}
\email{greg.pearlstein@unipi.it}
\thanks{G.~Pearlstein was partially supported by Progetti di Rilevante Interesse Nazionale (PRIN), Geometry of Algebraic Structures: Moduli,
 Invariants, Deformations and Gruppo Nazionale per le Strutture Algebriche, Geometriche e le loro Applicazioni (GNSAGA)}
\date{\today}
\date{\today}

\newif\ifprivate
\privatetrue
\thanks{}

\begin{abstract}
  To a pair of Bloch higher cycles that intersect properly and have
  complementary codimensions, we attach a mixed Hodge structure with some
  extra data (a \emph{framed mixed Hodge structure}). Using 
  this mixed Hodge structure, we define two different archimedean
  local height pairings. Both constructions generalize the
  biextension archimedean height attached to a pair of classical algebraic cycles
  homologous to zero. When applied to the polylogarithm variation of mixed
  Hodge structures we recover both the single valued polylogarithm of  
  Bloch, Wigner et al. and the one defined by Brown. 
  We also prove several salient properties of
  these heights including various vanishing results.

\end{abstract}

\maketitle

\tableofcontents{}

\section*{Introduction}
\label{sec:introduction}
\subsection*{Motivation}
Let $X$ be a smooth projective variety of dimension $d$ defined over a
number field $F$. Assuming several conjectures, Beilinson defined a height
pairing between Chow groups of algebraic cycles homologous to zero
that are in complementary codimensions. This pairing has two
components, one defined for each non-archimedean prime of the number
ring associated with $F$, and the other is defined for each archimedean
embedding of $F$ inside $\C$. The archimedean component of the height pairing is of particular interest in complex geometry, as it can be defined independent of the arithmetic setting. In particular,
R.~Hain in \cite{Hain:Height} attached a mixed Hodge structure to a
pair of algebraic cycles satisfying the conditions for Beilinson's
height pairing, such that the archimedean component of the height
pairing can be understood as a class measuring how far this mixed
Hodge structure is from being split as an $\R$-mixed Hodge
structure. More concretely, the mixed Hodge structure appearing in
Hain's  
work is a biextension $B$ associated with a pure Hodge structure $H$
of weight $-1$, whose weight-graded pieces are of the form
\begin{equation}\label{eq:3}
  \Gr_{0}^{W}(B)=\Q(0),\quad \Gr_{-1}^{W}(B)=H,\quad
  \Gr_{-2}^{W}(B)=\Q(1).
\end{equation}
A real mixed Hodge structure with underlying vector space $V$ and a weight filtration of the form 
 $0\subsetneq W_k\subsetneq W_{k+1}=V$ is always split over $\R$.  Therefore, the property 
   of a mixed Hodge structure of the form \eqref{eq:3} being split over $\R$ depends only on the extension 
   data from $\Gr^W_0(B)$ to $\Gr^W_{-2}(B)$. 

Motivated by this Hodge theoretic approach,  the authors
extended the archimedean component of the height 
pairing  to Bloch higher cycles of 
degree 1  in the paper
\cite{BGGP:Height}. Namely, there is a bigraded complex of algebraic cycles 
$Z^{\ast}(X,\ast)_{00}$ computing
higher Chow  groups (see  Section \ref{subsec:2.2}).
With two cycles in 
$Z^{\ast}(X,1)_{00}$ that intersect properly, we associated a
mixed Hodge structure $B$ with graded pieces 
\begin{equation}\label{eq:1}
  \Gr_{0}^{W}(B)=\Q(0),\quad \Gr_{-2}^{W}(B)=H,\quad
  \Gr_{-4}^{W}(B)=\Q(2),
\end{equation}
where now $H$ has pure weight $-2$. It is shown in
\cite{BGGP:Height} 
that one can attach a height to such a mixed Hodge structure.

The main purpose of the present paper is to generalize the
archimedean component of the height pairing to properly intersecting
higher cycles 
in $Z^{p}(X,n)_{00}$ and $Z^{q}(X,n)_{00}$ with $n\geq 1$,
that are in complementary 
codimensions, which means that $p+q=d+n+1$.
To this end, we first attach a mixed Hodge structure to the cycles and then attach a
height to the mixed Hodge structure.  
It turns out that, for arbitrary mixed Hodge structures there are two
natural extensions of the height
pairing, which we denote  $\Ht_{1}$ and
$\Ht_{2}$, which carry distinct information. In fact,
variants of both heights 
have been considered before. For instance, we will see that each of
the heights gives rise to a type of single-valued polylogarithms.

Height $\Ht_{1}$ is easier to write down in terms of differential
forms, while height $\Ht_{2}$ is more satisfactory from the
theoretical point of view. We  
prove several results related to these heights, in
particular, a vanishing result for $\Ht_{1}$ which shows that, in
order to get a non-zero value of $\Ht_{1}$  we need that  $\dim(X)\ge
\min(p,q)$.

There are two significant departures from the case of usual algebraic cycles
that we should always keep in mind during our discourse. 
\begin{itemize}
\item
For $n\geq 1$, the cycle class in de Rham cohomology of any higher
cycle is zero. So we do not need to restrict
ourselves to cycles homologous 
to zero.

In particular, for $n=1$ we can attach a rational mixed Hodge
structure of the shape \eqref{eq:1} to any pair of cycles
of $Z^{\ast}(X,1)_{00}$ that intersect properly, are in
complementary codimensions, and satisfy a technical condition. But now
it is no longer  
true that two consecutive pieces of a mixed Hodge structure
of the shape \eqref{eq:1} always splits over $\R$.

\item
In case of usual cycles, the condition to define the (archimedean)
height pairing, namely complementary codimensions and proper
intersection, guarantees that the cycles have empty intersection. This
is not true anymore for higher cycles. For instance, for 
cycles in $Z^{\ast}(X,1)_{00}$, the condition of proper intersection only
means that they intersect in a finite set of points after pulling back
to $X\times
(\P^1)^{2}$.  
\end{itemize}

In the case $n=1$ the condition that a mixed Hodge structure of the shape
\eqref{eq:1} splits over $\R$ is governed by several maps.
\begin{align*}
  \delta _{0,2}&\colon \Q(0)\longrightarrow H,\\
  \delta _{2,4}&\colon H\longrightarrow \Q(2),\\
  \delta _{0,4}&\colon \Q(0)\longrightarrow \Q(2).
\end{align*}
The maps $\delta _{0,2}$ and $\delta _{2,4}$ are determined by the
regulator of each cycle together with the 
relation of each cycle with the intersection of the cycles. While the map $\delta
_{0,4}$ is responsible for the height.  The total
interplay between the weight zero and weight -4 pieces is reflected by
the map $\delta _{0,4}$ and by the composition $\delta _{2,4}\circ
\delta _{0,2}$. In this case it is easy to isolate the contribution of
interest because  $\delta _{0,4}$ is imaginary while the
composition  $\delta _{2,4}\circ
\delta _{0,2}$ is real. The details of this discussion can be found in \cite[\S 4]{BGGP:Height}.

An even more involved picture emerges for $n>1$.
We can still attach a mixed Hodge structure to a pair of
higher cycles in $Z^{\ast}(X, n)_{00}$ that intersect properly and
have complementary codimension $p+q=d+n+1$, but controlling the
interactions of the extension classes is trickier, which is why we
consider two different height pairings.

We point out that both heights defined on this paper should be thought
as the archimedean component of an as yet to be understood global
height; for instance through a comparison with the height pairing defined
in \cite{Burgoswami:hait} as was done in the case $n=1$ in
\cite{BGGP:Height}. Moreover, as already happens in the classical case $n=0$,
the  
archimedean component of the height pairing is not invariant under
rational equivalence and is only well defined for cycles that
intersect properly. 

\subsection*{Mixed Hodge structures and heights associated to higher cycles}
Now we elaborate on the motivational discussion above. In the work of R.~Hain (\cite{Hain:Height}), the notion of a
biextension of type \eqref{eq:3} arose through a pair of cycles
$Z\in Z^{p}(X)$ and $W\in Z^{q}(X)$, which are homologous to zero,
intersect properly, and are in complementary codimensions
($p+q=d+1$). In this case, the biextension MHS $B_{Z,W}$ is given by
the non-torsion part of a subquotient of
$\rH^{2p-1}(X\setminus |Z|,|W|; \Z(p))$. A crucial observation in
defining this mixed Hodge structure is the fact that proper
intersection and complementary codimensions imply
$|Z|\cap |W|=\emptyset$. Hence, the mixed Hodge structures
$\rH^{2p-1}(X\setminus |Z|,|W|; \Z(p))$ and
$\rH^{2q-1}(X\setminus |W|,|Z|; \Z(q))$ are in duality.

The main question which forms the basis of our project is
the following: Can one generalize this study of heights and mixed Hodge structures to Bloch's
higher Chow cycles? The natural first step is to define a mixed Hodge
structure associated to a pair of higher cycles. To do that, we fix a
complex $(Z^{p}(X,\ast)_{00},\delta)$ of refined pre-cycles (see
\S \ref{subsec:2.2})
whose homology computes the higher Chow groups $\CH^{p}(X,\ast)$. We
use this complex since the cohomology class of higher cycles in it
lives naturally in the appropriate relative cohomology groups. Using this complex, in
\cite{BGGP:Height}, we extended the notion of biextension to that of a
\textit{generalized biextension} of the shape \eqref{eq:1} attached to a
pair of higher cycles $Z\in Z^{p}(X,1)_{00}$ and $W\in Z^{q}(X,1)_{00}$
that intersect properly, and are in complementary codimensions (which
in this case means $p+q=d+2$). We had to make a key assumption (see
\cite[Assumption 3.27]{BGGP:Height}), which although satisfied
generically for $n=1$, is not expected to hold for $n>1$.

In this article, we can avoid this assumption and attach a
mixed Hodge structure to any pair of
higher cycles $Z\in Z^{p}(X,n)_{00}$ and $W\in Z^{q}(X,n)_{00}$ that are in complementary codimensions and intersect properly. 

\begin{mainthm}\label{mainthm}
Let $Z\in Z^{p}(X,n)_{00}$ and $W\in Z^{q}(X,n)_{00}$ be two  refined
cycles that intersect properly (Definition \ref{def:10}), and are in
complementary codimensions ($p+q=d+n+1$). Then there exists a mixed
Hodge structure $H_{Z,W}$ equipped with morphisms 
\begin{displaymath}
\phi_{Z}\colon \Q(0)\rightarrow \Gr^{W}_{0}H_{Z,W},
\end{displaymath}
and
\begin{displaymath}
\psi_{W^{\vee}}\colon \Gr^{W}_{-2n-2}H_{Z,W}\rightarrow \Q(n+1)
\end{displaymath}
given by the cycle class of $Z$ and the dual of the cycle class of
$W$. 
\end{mainthm}
A mixed Hodge structure provided with two maps as in the main theorem
is called a \emph{framed mixed Hodge structure}.

We now explain the ideas behind the construction of  this mixed Hodge
structure. First, we need to introduce some normal 
crossing divisors: We define $A\subset (\P^1)^{n}$ to be the simple
normal crossing divisor consisting of points with one affine
coordinate equal to $1$, and $B\subset (\P^1)^{n}$ be the simple
normal crossing divisor of points with at least one coordinate in the
set $\{0,\infty\}$. Then $A\cup B$ is also a simple normal crossing
divisor. Let $A_{X}$ and $B_{X}$ be their pullbacks to
$X\times (\P^1)^{n}$. One can check that there are canonical
isomorphisms (\cite[eqs. (3.4) and (3.5)]{BGGP:Height})
\begin{align}
  \rH^{r}(X;\Q(s))&=\rH^{r+n}(X\times (\P^1)^n\setminus A_X,B_X;\Q(s))\label{eq:51}\\
  \rH^{r}(X;\Q(s))&=\rH^{r+n}(X\times (\P^1)^n\setminus B_X,A_X;\Q(s+n))\label{eq:52}
\end{align}
for all $r,s\in \Z$.
Since the cycle $Z$ belongs to $Z^p(X,n)_{00}$, it defines a class
\begin{displaymath}
  \cl(Z)\in \rH^{2p}_{Z}(X\times (\P^1)^n\setminus A_X,B_X;\Q(p)).  
\end{displaymath}
The image of this class in
$\rH^{2p}(X\times (\P^1)^n\setminus A_X,B_X;\Q(p))=\rH^{2p-n}(X;\Q(p))$
is zero as we have already noted. This implies that $\cl(Z)$ lifts
uniquely to a class in 
\begin{displaymath}
\cl(Z)\in \Gr^{W}_{0}\rH^{2p-1}(X\times (\P^1)^{n}\setminus A_{X}\cup |Z|, B_{X}; \Q(p)).
\end{displaymath}
Similarly we obtain a class 
\begin{displaymath}
\cl(W)\in \Gr^{W}_{0}\rH^{2q-1}(X\times (\P^1)^{n}\setminus A_{X}\cup
|W|, B_{X}; \Q(q)).
\end{displaymath}
We want to put these two classes in the same mixed Hodge structure in
order to define a framed mixed Hodge structure. The idea is to consider the cartesian diagram  
\begin{displaymath}
  \xymatrix{&X\times (\P^{1})^{2n}\ar[dl]_{p_{1}} \ar[dr]^{p_{2}}&\\
    X\times (\P^{1})^{n} \ar[dr]&& X\times (\P^{1})^{n}\ar[dl]\\
    &X&
  }
\end{displaymath}
We denote $A_i$ (respectively $B_i$) the pullback of $A_X$
(resp. $B_X$) by $p_i$ and by abuse of notation we still denote
$Z=p_1^\ast Z$  and $W=p_2^\ast W$.  By a variant of the isomorphism
\eqref{eq:52}, we obtain a class
\begin{displaymath}
  \cl(Z)\in \Gr^{W}_0\rH^{2p+n-1}(X\times (\P^1)^{2n}\setminus A_1\cup B_2\cup
  Z,B_1\cup A_2;\Q(p+n)).
\end{displaymath}
Using the numerical condition $p+q=d+n+1$, it is easy to show that this class vanishes when restricted to $W$ and
can be lifted to a class
\begin{displaymath}
  \cl(Z)\in \Gr^{W}_0\rH^{2p+n-1}(X\times (\P^1)^{2n}\setminus A_1\cup B_2\cup
  Z,B_1\cup A_2\cup W;\Q(p+n)).
\end{displaymath}
Similarly we obtain a class
\begin{displaymath}
    \cl(W)\in \Gr^{W}_0\rH^{2q+n-1}(X\times (\P^1)^{2n}\setminus B_1\cup A_2\cup
  W,A_1\cup B_2\cup Z;\Q(q+n)).
\end{displaymath}
The main difficulty now is that the pair of subspaces
$C_1\coloneqq A_1\cup B_2\cup
Z$ and $C_2\coloneqq B_1\cup A_2\cup W$ are not in local product situation, and
hence the cohomology groups where $ \cl(Z)$ and $\cl(W)$ live
are not in duality. So we cannot yet define a framed MHS. Therefore we
blow-up $X\times 
(\P^1)^{2n}$ 
along centers contained in $Z\cup W$, to obtain a space $\caX_{Z,W}$
where the total transform of
$C_1\cup C_2$ is a normal crossing divisor $D$, so we are in local
product situation. The key point in this construction is that it should be
possible to decompose $D=D_1\cup D_2$ into two simple normal
crossing
divisors without common components in such a way that the classes of
$Z$ and $W$ can be lifted to classes
\begin{align*}
  \cl(Z)&\in \Gr_0^W\rH^{2p+n-1}(\caX_{Z,W}\setminus D_1,D_2;\Q(p+n)),\\
  \cl(W)&\in \Gr_0^W\rH^{2q+n-1}(\caX_{Z,W}\setminus D_2,D_1;\Q(q+n)).
\end{align*}
To this end, we use the precise formulation of Hironaka's resolution
of singularities recalled in the Appendix \ref{sec:resol-sing} and we
analyze for each blow-up on a smooth center, how to lift the
classes.
Finally, once again, the condition $p+q=d+n+1$ will now imply the isomorphism
\begin{multline*}
  \Gr_0^W\rH^{2q+n-1}(\caX_{Z,W}\setminus D_2,D_1;\Q(q+n))\cong \\
  \Gr_{-2n-2}^W\rH^{2p+n-1}(\caX_{Z,W}\setminus
  D_1,D_2;\Q(p+n))^{\vee}. 
\end{multline*}
We define $H_{Z,W}=\rH^{2p+n-1}(\caX_{Z,W}\setminus
  D_1,D_2;\Q(p+n))$, as the mixed Hodge structure associated to the pair $(Z,W)$. The classes of $Z$ and $W$ determine the morphisms
\begin{displaymath}
\phi_{Z}\colon \Q(0)\rightarrow \Gr^{W}_{0}H_{Z,W},
\end{displaymath}
and
\begin{displaymath}
\psi_{W^{\vee}}\colon \Gr^{W}_{-2n-2}H_{Z,W}\rightarrow \Q(n+1).
\end{displaymath}
We note that the mixed Hodge structure $H_{Z,W}$ and the framing
$(\phi_{Z},\psi_{W^{\vee}})$ is not unique, but depends on the choice
of the resolution of singularities $\caX_{Z,W}$  and on the choice of
the decomposition $D=D_{1}\cup D_{2}$. Nevertheless if $H'_{Z,W},
(\phi'_{Z},\psi'_{W^{\vee}})$ is a different choice; there is always a
third one $H''_{Z,W},
(\phi''_{Z},\psi''_{W^{\vee}})$ with maps $f\colon H_{Z,W}\to
H''_{Z,W}$ and $g\colon H_{Z,W}\to H''_{Z,W}$ such that
\begin{displaymath}
  \phi''_{Z}=\Gr_{0}^{W}f\circ \phi_{Z}=\Gr_{0}^{W}g\circ \phi'_{Z} ,
\end{displaymath}
  \begin{displaymath}
  \psi_{W^{\vee}}=\psi''_{W^{\vee}}\circ
  \Gr_{-2n-2}^{W}f, \text{ and }
\psi'_{W^{\vee}}
  = \psi''_{W^{\vee}}\circ \Gr_{-2n-2}^{W}g. 
\end{displaymath}

Once the mixed Hodge structure $H_{Z,W}$ is defined, the next task is
to define a suitable notion of height pairing. As mentioned earlier, we
give two definitions of heights.

Using the Deligne bigrading $H_{Z,W}=\bigoplus I^{\ast,\ast}$ (see \S
\ref{sec:deligne-splitting}) we 
lift the class $\cl(Z)\in I^{0,0}$ to a class $\cl(Z)\in H_{Z,W}$ and the
class $\cl(W)\in (I^{-n-1,-n-1})^{\vee}$ to a class $\cl(W)\in H_{Z,W}^{\vee}$.
Let $\langle~,~\rangle$ denote the duality
\begin{displaymath}
\rH^{2p+n-1}(\caX_{Z,W}\setminus D_1,D_2;\Q(p+n))\cong
\rH^{2q+n-1}(\caX_{Z,W}\setminus D_2,D_1;\Q(q+n))^{\vee},
\end{displaymath}
and $\im(z)$ denotes the imaginary component of any complex number
$z$. We define the first height as 
\begin{displaymath}
\Ht_{1}(Z,W)\coloneqq \im\langle \cl(W), \overline{\cl(Z)}\rangle.
\end{displaymath}
The second definition of height uses the Deligne map
$\delta_{H_{Z,W}}$ associated with the mixed Hodge structure $H_{Z,W}$,
and is defined as 
\begin{displaymath}
\Ht_{2}(Z,W)\coloneqq \langle \cl(W), \delta_{H_{Z,W}}(\cl(Z))\rangle.
\end{displaymath}
We prove that these height pairings do not depend on the choice of a
resolution of singularities leading to the space $\caX_{Z,W}$ nor on
the choice of $D=D_{1}\cup D_{2}$
  (Remark \ref{rmk-5-10}), and
satisfy the symmetry property $\Ht_{1,2}(Z,W)=(-1)
  ^n\Ht_{1,2}(W,Z)$   (Proposition \ref{5-5-dual}). An interesting
  observation about these heights is that when $X$, $Z$, $W$ are
  defined over $\R$, then $\Ht_{1,2}(Z,W)=0$ for odd $n$ (see
  Proposition \ref{conj-ht-cycle}). Finally, Proposition
  \ref{prop:5-14} shows that one of $p$ and $q$ should be $\leq
  \dim(X)$, if one wants to find a nonzero $\Ht_{1}(Z,W)$.

\subsection*{Main ingredients of the construction}
Now we discuss the two key ingredients that contributed to the main
result. The first one is the notion of abstract height associated
to a framed mixed Hodge structure, and the second one is a complex of differential forms rapidly decreasing along one divisor,
while slowly increasing along another. 
\subsubsection*{Abstract heights}
The notion of heights defined above has a more abstract origin, and
can be defined for mixed Hodge structures with a \textit{framing}. The notion of framed mixed Hodge structures is not new, and has appeared
in \cite{BGSV} and \cite{AJ03} to name a few. In \S
\ref{sec:deligne-splitting-height}, we define it in a way that is more
suited towards our eventual goal of attaching heights to higher
cycles. 
\begin{dfn}\label{intro-def-1}
Let $\Q(n),~n\in \Z$ be the Tate mixed Hodge structure of pure type,
with Betti generator $\bfone(n)_{\Betti}\coloneqq 1\in \Q$ (see Example
\ref{exm:3} for details). For $(a,b)\in \Z^{2}$,  an
$(a,b)$-\textit{framing} on a mixed Hodge structure $H$ is a pair 
of morphisms
\begin{displaymath}
  \phi\colon
  \Q(-a)\rightarrow \Gr^{W}_{2a}(H)\text{ and }
  \psi\colon \Gr^{W}_{2b}(H)\rightarrow \Q(-b)
\end{displaymath}
of pure Hodge structures.
\end{dfn}
To emphasize a framing on $H$, we sometimes denote it by
$H_{a,b}$. Let $e_{H}$ be the image inside $I^{a,a}_{H}$
($(a,a)$-component of the Deligne bigrading of $H$) of the Betti
generator $\bfone(-a)_{\Betti}\in \Q(-a)$. It is easy to show that if $H$ has a
framing given as above, then its dual $H^{\vee}$ also has a framing
given by morphisms
\begin{displaymath}
  \phi^{\vee}\colon
  \Q(b)\rightarrow \Gr^{W}_{-2b}(H^{\vee})\text{ and }
  \psi^{\vee}\colon \Gr^{W}_{-2a}(H^{\vee})\rightarrow \Q(a).
\end{displaymath}
Let $e_{H^{\vee}}$ be the element equivalent to $e_{H}$, inside
$I^{-b,-b}_{H^{\vee}}$. With this setup, we define two notions of
height (Definition \ref{height-framed} and Definition \ref{def-ht-2}): 
\begin{displaymath}
\Ht_{1}(H)\coloneqq\im\langle e_{H^{\vee}}, \overline{e_{H}}\rangle
\end{displaymath}
and
\begin{displaymath}
\Ht_{2}(H)\coloneqq \langle e_{H^{\vee}}, \delta_{H}(e_{H})\rangle.
\end{displaymath}
Here $\langle~,~\rangle$ denotes the duality between $H$ and
$H^{\vee}$ and $\delta_{H}$ is the Deligne splitting of $H$ (see
\eqref{delta-def}). Notice here that $\Ht_{2}(H)$ is already real
(Lemma 
\ref{lem-2-17}), while for $\Ht_{1}(H)$ we take the imaginary
component. These heights in general are not the same (see Example
\ref{ex-ht}). However, they do agree on the (generalized)
biextensions appearing in \eqref{eq:3} and \eqref{eq:1}. For $z\in
\P^{1}_{\C}\setminus \{0,1,\infty\}$, the mixed Hodge 
structure $H(z)$ of Hodge-Tate type associated with polylogarithms is a
natural testing ground for these heights, and in Examples \ref{exm:2}
and \ref{exm:14} 
we show that the heights $\Ht_{1}(H(z))$ and $\Ht_{2}(H(z))$ are given
by various real-valued avatars of polylogarithms appearing in
\cite{brown04:_singl} and 
\cite{beilinsondeligne:_inter_zagier}.  
\subsubsection*{Complex of differential forms}\label{int-subsec-complex}
Let $A$, $B$ be subvarieties inside $X$. The relative cohomology
groups $\rH^{\ast}(X\setminus A, B)$ of the pair of spaces
$(X\setminus A, B\setminus A\cap B)$ appear as a central theme in
this article. For example, as we have mentioned earlier, the
cycle class $\cl(Z)$ for a higher cycle
$Z\in Z^{p}(X,n)_{00}$ lives naturally in the cohomology group
\begin{displaymath}
\Gr^{W}_{0}\rH^{2p-1}(X\times (\P^1)^{n}\setminus A_{X}\cup |Z|, B_{X};\Q(p)).
\end{displaymath}
To represent elements in $\rH^{\ast}(X\setminus A, B)$, we need a
suitable complex which computes this
cohomology. There are a few such complexes, but none that have
all the properties of cohomology, so one has to resort to trade-offs.
We focus on the case when $A$ and $B$ are simple normal crossing
divisors without common components such that $A\cup B$ is also a
simple normal crossing divisor. In fact, we can always reduce to this case by
resolution of singularities. The minimal requirements we have are the
following.
\begin{enumerate}
\item We need a complex that computes $\rH^{\ast}(X\setminus A, B,\C)$
  with its Hodge filtration.
\item The complex should be compatible with the real structure
  $\rH^{\ast}(X\setminus A, B,\R)$ so we can represent easily the complex
  conjugation. 
\item The complex has to be compatible with the operations in
  cohomology listed in Section \ref{sec:cohom-results}. 
\end{enumerate}
The first candidate for such a complex is the complex of differential
forms with logarithmic singularities along $A$ and vanishing along
$B$:
\begin{displaymath}
  \Sigma _{B}E^{\ast}_{X}(\log A)=
  \{\omega \in E^{\ast}_{X}(\log A)\mid
  \iota ^{\ast} \omega =0\}
  \subset E^{\ast}_{X}(\log A),
\end{displaymath}
where $\iota\colon \widetilde{B}\rightarrow B$ is the normalization of
$B$. This is the complex that we have used in \cite{BGGP:Height}. It
satisfies almost all our requirements except one. Given three normal
crossing divisors $A$, $B$ and $C$ without common components such that
$A\cup B\cup C$ is still a normal crossings divisor; there is an
exceptional cup-product
\begin{displaymath}
  \rH^{r}(X\setminus A\cup C,B)\otimes
  \rH^{s}(X\setminus A,B\cup C)\longrightarrow
  \rH^{r+s}(X\setminus A,B\cup C).
\end{displaymath}
This cup product can not be easily represented using a complex of the
form $\Sigma _{B}E^{\ast}_{X}(\log A)$ because the product of a form
in  $E^{\ast}_{X}(\log C)$ by a form in $\Sigma _{C}E^{\ast}_{X}$ does
not belong in general to $\Sigma _{C}E^{\ast}_{X}$. Of course, one can
argue that, given particular cohomology classes, we can always find
nice representatives whose product lives in the right complex.
Nevertheless, to avoid unnecessary complications, it is more
convenient to use a complex where this cup-product is well defined.

Therefore, we will use the complex of differential forms, which are
slowly increasing along $A$ and rapidly decreasing along $B$ (see
Definition  \ref{def:3}) denoted $E_{X}^{\ast}(\si A,\rd B)$. This complex was studied by Harris and 
 Zucker (see for instance \cite{HarrisZucker:BcSvIII}) and has been
revisited recently in \cite{BCLR:_temper_delig_shimur}.
The main advantage of this complex is that it allows us to represent
the complex conjugation and all the operations in cohomology
easily. We caution the reader that unlike $\Sigma _{B}E^{\ast}_{X}(\log A)$, the complex $E_{X}^{\ast}(\si A, \rd B)\not\subset E^{\ast}_{X}(\log A)$, and this necessitates some additional restrictions while representing operations in cohomology (see Proposition \ref{prop:4-4'} and the following Remark \ref{rmk:4-5'}). For details about this complex, see Section \ref{sec:spac-diff-forms}.

\subsection*{Layout}
We now describe the organization of the paper. In \S
\ref{preliminaries}, we give an account of the basic ingredients
needed for the paper. We provide a short description of the rational mixed
Hodge structures with the de Rham and Betti conjugations, a weak
$\R$-Hodge complex, higher Chow groups with their associated mixed
Hodge extension class, and Goncharov's regulator current. In \S \ref{sec:deligne-splitting-height} we
introduce the first key concept of this article, namely the notion of
heights associated with framed mixed Hodge structures, provide examples,
and prove several properties associated with it. In \S
\ref{sec:cohom-results} and \S \ref{sec:spac-diff-forms}, we set up the
machinery for \S \ref{sec:main-construction}. In \S
\ref{sec:cohom-results} we prove several results associated with the
cohomology of the relative space $(X\setminus A, B\setminus A\cap B)$,
and its behavior with respect to blow-ups and cup products. In \S
\ref{sec:spac-diff-forms} we introduce and prove properties of the
space $E^{\ast}_{X}(\si A, \rd B)$ of forms slowly increasing along
$A$ and rapidly decreasing along $B$. In the cases of our interest, we
show that this complex computes the cohomology groups of \S
\ref{sec:cohom-results}, and that all the cohomology operations of \S
\ref{sec:cohom-results} can be 
represented by the corresponding ones of this complex. Finally, in \S
\ref{sec:main-construction} we develop the framed mixed Hodge
structures associated with a pair of higher cycles, the central theme of
this paper. The key result here is Theorem \ref{thm:5}, which helps us
to circumvent the issue arising from the lack of local product
situation. For the convenience 
of the reader, we state a precise version of Hironaka's
resolution of singularities at the very end of the paper in \S
\ref{sec:resol-sing}. We use it crucially in \S
\ref{sec:main-construction}.

\subsection*{Acknowledgements}\label{subsec-Ac} During the elaboration
of this paper we have benefited with stimulating conversation with many
mathematicians. Our thanks to S. Bloch, P. Brosnan,  V. Golyshev and
specially Matt Kerr who shared with us some unpublished notes with his
ideas about higher heights.

The authors would like to extend a warm thank you to the referee for
his or her many comments and suggestions that helped to make this
manuscript more readable. 

Part of the work on this project has been done during visits  of the three
authors to the ICMAT, the University of Barcelona, and the University
of Pisa. The second author wants to thanks the IHES for kindly hosting
him during the final stage of the project.  
The authors would like to
thank these institutions for their support and hospitality.

\section{Preliminaries}\label{preliminaries}
In this section we give a short survey of all the conventions, notations,
and preliminary results that we will need for the later
sections. Throughout, $X$ will denote a smooth and (for most cases)
projective variety of dimension $d$ over $\C$. As is the usual
practice, we will not distinguish between $X$ and its corresponding
complex manifold of complex points $X(\C)$. For example, the Betti
cohomology groups of $X(\C)$ will be denoted by $\rH^{\ast}(X)$
instead of $\rH^{\ast}(X(\C))$.
\subsection{Mixed Hodge structures}
We start by recalling the definition of mixed Hodge structure that we
use.
\begin{df}\label{def:8} Let $\Lambda \subset \R$ be a subring. A
  $\Lambda $-Mixed Hodge structure defined over 
  $\C$ (a $\Lambda $-MHS for short) is the data
  \begin{displaymath}
    H=((H_{\Betti},W),(H_{\dR},W,F),\alpha_{H}),
  \end{displaymath}
  where $H_{\Betti}$ is a $\Lambda $-module with an increasing
  weight filtration $W$, $H_{\dR}$ is a $\C$-vector space with a decreasing
  Hodge filtration $F$ and an increasing weight filtration $W$ and
  $\alpha_{H} \colon H_{\Betti}\otimes \C\to H_{\dR}$ is an isomorphism
  compatible with the weight filtration on both sides. Such that, for
  every $n\in \Z$ the triple
  \begin{displaymath}
    ((\Gr_{n}^{W}H_{\Betti}),(\Gr_{n}^{W}H_{\dR},\Gr_{n}^{W}F),(\Gr_{n}^{W}\alpha_{H}
    ))
  \end{displaymath}
  is a pure Hodge structure of weight $n$. The $\Q$-vector space
  $H_{\Betti}$ is called the Betti side of the MHS and the $\C$-vector
  space $H_{\dR}$ the de Rham side. The map $\alpha_{H} $ is called the
  comparison isomorphism. In an abstract mixed Hodge structure, the
  standard complex conjugation is the one corresponding to the Betti
  rational structure. When $\Lambda =\Q$ we omit the coefficients from the notation
  and call it a $MHS$.  
\end{df}

\begin{ex}\label{exm:3}
  For $a\in \Z$, the Tate mixed Hodge structure $\Q(a)$ is the mixed
  Hodge structure given by the following data
  \begin{gather*}
    \Q(a)_{\Betti}=\Q,\quad W_{-2a-1}\Q(a)_{\Betti}=0,\quad
    W_{-2a}\Q(a)_{\Betti}=\Q\\
    \Q(a)_{\dR}=\C,\quad F^{-a}\Q(a)_{\dR}=\C,\quad
    F^{-a+1}\Q(a)_{\dR}=0\\
    \alpha _{\Q(a)}(1)=(2\pi i)^{a}\in \C.
  \end{gather*}
The mixed Hodge structure $\Q(a)$ comes equipped with the
choice of two
generators.
\begin{align*}
  \bfone(a)_{\Betti} &\coloneqq 1\in \Q=\Q(a)_{\Betti}\\
  \bfone(a)_{\dR} &\coloneqq 1\in \C=\Q(a)_{\dR}.
\end{align*}
These generators are called the Betti and the de Rham generators. They
satisfy
\begin{displaymath}
  \overline {\bfone(a)_{\Betti}}=\bfone(a)_{\Betti},\qquad
  \overline {\bfone(a)_{\dR}}=(-1)^{a}\bfone(a)_{\dR},\qquad
  \alpha(\bfone(a)_{\Betti})=(2\pi i)^{a}\bfone(a)_{\dR}.
\end{displaymath}
\end{ex}
\begin{rmk}\label{rem:7} This definition of mixed Hodge structure using two
  different vector spaces, the Betti one $H_{\Betti}$ and the de Rham one
  $H_{\dR}$, and a comparison isomorphism, is very useful when working
  with periods, because we can choose the de Rham side to be defined
  over the same field as the variety $X$. Nevertheless in many
  occasions, for instance when the variety is defined over $\C$ or one
  is interested in complex variations of mixed Hodge structures,  
  it is convenient to use a less cumbersome notation. In this case we
  will consider a $\Q$-vector space $V$ with an increasing weight
  filtration $W$ and a decreasing Hodge filtration $F$ on
  $V_{\C}\coloneqq V\otimes \C$. So that $H_{\Betti}=V$, $H_{\dR}=V_\C$ and
  $\alpha =\Id$. If $V$ is fixed we will denote the mixed Hodge
  structure by the pair $(W,F)$.
\end{rmk}

\begin{rmk}\label{rem:2.2}
  Let $H$ be a $\Q$-mixed Hodge structure of rank one defined over $\C$. Then $H$ is
  necessarily pure of even weight, say $2a$. It follows that it is
  isomorphic to 
  $\Q(-a)$. The choice of an isomorphism $H\to \Q(-a)$ is equivalent to
  the choice of a generator $e$ of $H_{\Betti}$. Then the comparison
  isomorphisms will induce a unique compatible isomorphism
  $H_{\dR}\cong \Q(-a)_{\dR}$ defined over $\C$. 
  This is no longer true
  for $\Q$-mixed Hodge structures of rank one defined over a subfield
  $k$ of $\C$, because it may happen that the above isomorphism is
  not defined over $k$. This can be checked by looking at the periods
  of $H$. 
\end{rmk}

\begin{ex}\label{exm:6} Let $H=(H_{\Betti},H_{\dR},\alpha_{H} )$ be a
  MHS. Then the complex conjugate MHS, denoted as $\overline H$ is
  given by
  \begin{displaymath}
    (\overline H)_{\Betti}=H_{\Betti}, \ 
    (\overline H)_{\dR}=\overline{H_{\dR}}, \ 
    \alpha _{\overline H}(x) =
    \overline {\alpha _{H}(\overline x)}, \  W_{\overline {H}}=W_{H},\ 
    F_{\overline {H}} = \overline{F_{H}}.
  \end{displaymath}
  Recall that, for a complex vector space $E$, the complex conjugate
  $\overline E$ is the same abelian group with the complex conjugate
  action of $\C$. The complex conjugation $E\to \overline E$ is the
  identity of the underlying abelian groups.   

  In the notation  of Remark \ref{rem:7}, if $H=(V_{H},W_H,F_H)$ is a mixed Hodge
  structure. The complex conjugate MHS  $\overline {H}=(V_{\overline
    {H}},W_{\overline {H}},F_{\overline {H}})$ is given by
  \begin{displaymath}
    V_{\overline
      {H}}=V_{H},\quad W_{\overline {H}}=W_{H},\quad
    F_{\overline {H}} = \overline{F_{H}}.
  \end{displaymath}

  If $X$ is a variety defined over $\C$, then the complex conjugate
  $\overline X$ is
  defined by the cartesian diagram
  \begin{displaymath}
    \xymatrix{ \overline X \ar[r] \ar[d] & X\ar[d]\\
      \Spec \C \ar[r]^{c} &\Spec \C,
    }
  \end{displaymath}
  where $c$ is complex conjugation. At the level of complex manifolds, the map $x\mapsto \overline x$
  gives a homeomorphism $X(\C)\to \overline X(\C)$ that gives an
  isomorphism $\rH^{n}_{\Betti}(X;\Q)\simeq \rH^{n}_{\Betti}(\overline
  X;\Q)$. But this map sends $F^{p}\rH^{n}_{\dR}(\overline X)$ to
  $\overline {F^{p}}\rH^{n}_{\dR}(X)$. Thus we see that, as mixed
  Hodge structures,
  \begin{displaymath}
  \rH^{n}(\overline X)\simeq \overline{\rH^{n}(X)}. 
  \end{displaymath}
  This discussion can be extended to any \emph{diagram} of complex varieties.
  One has to be careful that, when $X$ is defined over $\R$, then we
  can identify $X$ with $\overline X$, but the above isomorphism is
  not the identity. 
\end{ex}

\begin{ex}\label{exm:8}
  There is an isomorphism of MHS between $\Q(a)$ and
  $\overline{\Q(a)}$ that sends $\bfone(a)_{\dR}$ to $\overline
  {\bfone(a)_{\dR}}$ and $\bfone(a)_{\Betti}$ to $(-1)^{a}\bfone(a)_{\Betti}$.  Indeed,
  \begin{displaymath}
    \alpha_ {\overline {\Q(a)}}(\bfone(a)_{\Betti})=\overline {(2\pi
      i)^a\bfone(a)_{\dR}} =(-1)^{a}(2\pi
      i)^a \overline {\bfone(a)_{\dR}}. 
    \end{displaymath}
    Thus, in order for the isomorphism between $\Q(a)$ and
  $\overline{\Q(a)}$ to be compatible with the comparison isomorphisms,
  we have to change either the sign of $\bfone(a)_{\Betti}$ or the sign of
  $\bfone(a)_{\dR}$. In this article we choose the former.   
\end{ex}

\begin{ex}\label{exm:11}  The map $\R(p-1)_{\dR}\to \R(p)_{\dR}$ given by 
  $x\mapsto x\otimes \bfone(1)_{\dR}$ induces an isomorphism
  of real vector spaces
  \begin{displaymath}
    \R(p-1)_{\Betti}\longrightarrow \R(p)_{\dR}/\R(p)_{\Betti}.
  \end{displaymath}
Therefore, if $H$ is an $\R$-MHS, the map
  $H(p-1)\to H(p)$ given by $x\mapsto x\otimes \bfone(1)_{\dR}$
  induces an isomorphism of real vector spaces
  \begin{displaymath}
    H(p-1)_{\Betti}\longrightarrow H(p)_{\dR}/H(p)_{\Betti}.
  \end{displaymath}
\end{ex}

Frequently, in the sequel we will use complexes that compute relative
cohomology of a complex projective variety, but they only have
information about the real structure and the Hodge filtration, and
not about the weight filtration. To work with these at ease we
introduce the following notation.  

\begin{df}\label{def:1} A \emph{weak $\R$-Hodge complex} is a triple
  \begin{displaymath}
    A=(A^{\ast}_{\Betti},(A^{\ast}_{\dR},F),\alpha _{A}),
  \end{displaymath}
  where $A^{\ast}_{\Betti}$ is a complex of real vector spaces,
  $(A^{\ast}_{\dR},F)$ is a filtered complex of $\C$-vector spaces and
  \begin{displaymath}
    \alpha _{A}\colon A^{\ast}_{\Betti}\otimes \C \longrightarrow A^{\ast}_{\dR}
  \end{displaymath}
  is an isomorphism. The isomorphism $\alpha _{A}$ allows us to
  identify $A^{\ast}_{\Betti}$ with the subcomplex $\alpha
  _{A}(A^{\ast}_{\Betti})\subset A^{\ast}_{\dR}$. Moreover it defines an antilinear
  involution $\omega \to \bar \omega$ of $A^{\ast}_{\dR}$ defined, for
  $r\in A^{\ast}_{\Betti}$ and $\lambda \in \C$, by
  \begin{displaymath}
    \overline{\alpha _{A}(r\otimes \lambda)}=\alpha _{A}(r\otimes \bar \lambda )
  \end{displaymath}

  Given a weak $\R$-Hodge complex $A=(A^{\ast}_{\Betti},(A^{\ast}_{\dR},F),\alpha _{A})$, the
  \emph{Tate
    twisted} weak $\R$-Hodge complex
  $A(a)=A\otimes \Q(a)$ is defined as
  \begin{displaymath}
    A(a)=(A^{\ast}_{\Betti}\otimes \Q(a)_{\Betti}, A^{\ast}_{\dR}\otimes
    \Q(a)_{\dR},\alpha _{A}\otimes \alpha _{\Q(a)})
  \end{displaymath}
\end{df}

\begin{rmk}\label{rem:2.4}
  Any \emph{Dolbeault complex} as in \cite[Definition
2.2]{Burgos:CDB} defines a weak $\R$-Hodge complex. We will notably
use this for complexes of differential forms and currents.
\end{rmk}

\begin{ex}\label{exm:9} As mentioned in Remark \ref{rem:2.4}, the
  basic examples of weak $\R$-Hodge 
  complexes are the complexes of differential forms and currents. Let
  $X$ be a complex manifold. Then the smooth differential forms form a
  weak $\R$-Hodge complex $E_{X}=(E^{\ast}_{X,\R},(E^{\ast}_{X},F),\Id)$ where the Betti
  part are real valued differential forms, the de Rham part are
  complex valued differential forms with the usual Hodge filtration
  and the comparison isomorphism is the identity.
  If $\omega \in E^{n}_{X}$, for shorthand, we will usually write 
  \begin{displaymath}
    \omega (p)\coloneqq \omega \otimes \bfone(p)_{\dR}\in E^{n}_{X}(p)_{\dR} 
  \end{displaymath}
  and similar notation for any other Dolbeault complex.

  We now shift to the complex of currents. Let
  $\prescript{\prime}{}{E}_{X}=(\prescript{\prime}{}{E}_{X,\R},
  (\prescript{\prime}{}{E}_{X,},F),\Id)$  be the weak $\R$-Hodge complex
  with $\prescript{\prime}{}{E}^{-n}_{X,\R}$ the topological dual of
  $E^{n}_{X,\R,c}$, the space of real valued compactly supported
  smooth differential forms, with the usual differential and Hodge
  filtration:
  \begin{displaymath}
    d T(\eta)=(-1)^{\deg T+1}T(d\eta), \quad
    F^{-p}\prescript{\prime}{}{E}^{\ast}_{X}=\{T\in
    \prescript{\prime}{}{E}^{\ast}_{X}\mid T(\omega 
    )=0,\ \forall \omega \in F^{p}\}
  \end{displaymath}
  If $X$ is equidimensional of dimension $d_{X}$, we write
  \begin{displaymath}
    D_{X}=\prescript{\prime}{}{E}_{X}[-2d_{X}](-d_{X}).
  \end{displaymath}
  One has to be careful that the complex of currents $D_{X}$ has an
  inherent twist that is usually omitted.
  For instance the current $\bfdelta _{X}$ characterized by
  \begin{displaymath}
    \bfdelta _{X}(\omega )=\int_{X}\omega 
  \end{displaymath}
  is an element of $\prescript{\prime}{}{E}^{-2d_{X}}_{X,\R}$. Then the
  element $\bfdelta _{X}\otimes \bfone(-d_{X})_{\Betti}$ belongs to
  $D^{0}_{\R}$. The element $\bfdelta _{X}\otimes
  \bfone(-d_{X})_{\Betti}$ is denoted simply by $\delta _{X}$ but one has
  to remember that it contains implicit the twist $\bfone(-d_{X})_{\Betti}$. In
  fact, every time one describes a current by its effect in
  differential forms, one is actually describing an element of
  $\prescript{\prime}{}{E}_{X}$ and one should keep in mind that there
  is an implicit factor $\bfone(-d_{X})_{\Betti}$.

  If $f\colon Y\to X$ is a morphism of complex manifolds then the map
  $f^{\ast}\colon E_{X}\to E_{Y}$ is a morphism of weak $\R$-Hodge
  complexes. If moreover $f$ is proper, then there is a morphism of
  weak $\R$-Hodge complexes $f_{\ast}\colon
  \prescript{\prime}{}{E}_{Y}\to \prescript{\prime}{}{E}_{X}$ defined
  by $f_{\ast}T(\eta)=T(f^{\ast}\eta)$. If $X$ and $Y$ are
  equidimensional of dimensions $d_{X}$ and $d_{Y}$ respectively so $f$
  has relative dimension $e=d_{Y}-d_{X}$ then we obtain a morphism of
  weak $\R$-Hodge complexes $f_{\ast}\colon D_{Y}\to
  D_{X}[-2e](-e)$. In particular, if $Y$ is a codimension $p$ subvariety of
  $X$ and $\iota\colon \widetilde Y\to X$ is a desingularization of
  $Y$, then the current $\iota _{\ast}\delta _{\widetilde Y}\otimes
  \bfone(-d_{Y})_{\Betti}\in D^{2p}_{X}(p)$ is denoted by $\delta
  _{Y}$. Note that here there is again an implicit twist.
  Consider the element of $\prescript{\prime}{}{E}^{-2d_{Y}}$ given by
  \begin{displaymath}
    \bfdelta _{Y}(\eta)=\int_{\widetilde Y}\iota ^{\ast}\eta.
  \end{displaymath}
   Then $\delta _{Y}=\bfdelta _{Y}\otimes
  \bfone(-d_{Y})_{\Betti}$.

  There are products $E_{X}\otimes D_{X}\to D_{X}$ and $D_{X}\otimes
  E_{X}\to D_{X}$ given by
  \begin{displaymath}
    (\omega \wedge T)(\eta)=(-1)^{\deg(\omega )\deg(T)}T(\omega
    \wedge \eta),\quad
    (T \wedge \omega) (\eta)=T(\omega
    \wedge \eta)
  \end{displaymath}
  and a morphism of weak $\R$-complexes $E_{X}\to D_{X}$ given by
  \begin{displaymath}
    \omega \mapsto [\omega ]=\delta _{X}\wedge \omega.
  \end{displaymath}
 The map $[\cdot
  ]$ can be extended to locally integrable differential forms.
 \end{ex}

 \begin{ex} \label{exm:12} If $X$ is a projective variety of dimension $d$, the
   trace map 
   \begin{displaymath}
     \Tr\colon \rH^{2d}(X;\R(d))\to \R(0)
   \end{displaymath}
   is a morphism of real mixed
   Hodge structures.  If
   $\omega $ is a top degree differential form then
   \begin{equation}\label{eq:35}
     \Tr(\omega(d))=\frac{1}{(2\pi i)^{d}}\int _{X}\omega .
   \end{equation}
   Hence
   \begin{displaymath}
     \Tr(\omega \otimes \bfone(d)_{\Betti})=\int _{X}\omega .
   \end{displaymath}
   If $Y$ is a codimension $p$ cycle, and $\omega $ a closed
   $(q,q)$-form, with $q=d-p$, then
   \begin{displaymath}
     \Tr(\delta _{Y}\wedge \omega(q))=\frac{1}{(2\pi i)^{q}}
     \int_{Y}\omega 
   \end{displaymath}
   
 \end{ex}

\begin{ex}\label{exm:5} Consider $X=\P^{1}$, the complex projective line with absolute
  coordinate $t$, so $\Div(t)=[0]-[\infty]$, and let
  $U=X\setminus \{0,\infty \}$ and $D=\{0,\infty \}$. Let $\alpha
  =dt/t$. Then 
  \begin{displaymath}
    \alpha(1) = \frac{dt}{t}\otimes \bfone(1)_{\dR}\in E^{1}_{X}(\log D; 1).
  \end{displaymath}
  Then $[\alpha(1)]\in D^{1}_{X}(1)$ and $d[\alpha(1) ]=\delta
  _{\Div(t)}$. In fact, writing down the implicit twists,
  \begin{displaymath}
    d[\alpha(1)]=d\big (\bfdelta_X\otimes
    \bfone(-1)_{\Betti}\wedge \frac{dt}{t}\otimes \bfone(1)_{\dR}\big)=(2\pi i)\bfdelta_{\Div(t)}
    (2\pi i)^{-1}=\delta _{\Div(t)}. 
  \end{displaymath}
  Here we are using that $\Div(t)$ has dimension zero and hence
  $\bfdelta_{\Div(t)}=\delta _{\Div(t)}$.
  This example can be easily generalized to $X$ a projective manifold,
  $f$ a rational function and 
  $\alpha =df/f$. Then $d[\alpha (1)]=\delta _{\Div(f)}$. 
\end{ex}

We recall the definition of the Deligne complex associated to a
Dolbeault complex from \cite[Definition 2.5]{Burgos:CDB}. In order to adapt to our convention, we will make a minor change of notation.
\begin{df} Let $A$ be the weak $\R$-Hodge complex associated to a
  Dolbeault complex. The Deligne complex associated to it
  (\cite[Definition 2.5, Theorem 2.6]{Burgos:CDB}) is given by:
\begin{displaymath}
  \fD^{n}(A,p)=
  \begin{cases}\displaystyle
    \frac{A^{n-1}(p-1)_{\dR}}{F^{1}+\overline F^{1}}\cap
    \frac{A^{n-1}(p-1)_{\Betti}}{(F^{1}+\overline F^{1})\cap A^{n-1}(p-1)_{\Betti}},&\text{
      if }n < 2p,\\[10mm]
    \displaystyle
    F^{0}A^{n}(p)_{\dR}\cap \overline{F}^{0}A^{n}(p)_{\dR}\cap A^{n}(p)_{\Betti},&\text{
      if }n \ge 2p,
  \end{cases}
\end{displaymath}
Using the decomposition of real vector spaces
\begin{displaymath}
  A^{n-1}(p-1)_{\dR}=A^{n-1}(p)_{\Betti}\oplus A^{n-1}(p-1)_{\Betti},
\end{displaymath}
for
$n<2p$ we can also write 
\begin{equation}
  \label{eq:33}
  \fD^{n}(A,p)=\frac{A^{n-1}(p)_{\dR}}{A^{n-1}(p)_{\Betti} + F^{0}A^{n-1}(p)_{\dR}}. 
\end{equation}
Thus, for $n<2p-1$ there is a projection
\begin{displaymath}
  \pi \colon A^{n-1}(p)_{\dR}\to \fD^{n}(A,p). 
\end{displaymath}
We next recall the differential of the Deligne complex. For clarity we include the
twists.  If
$n<2p-1$,
\begin{displaymath}
  d_{\fD}(x\otimes \bfone(p-1)_{\dR})=-\pi (dx)\otimes \bfone(p-1)_{\dR},
\end{displaymath}
for $n=2p-1$
\begin{displaymath}
  d_{\fD}(x\otimes \bfone(p-1)_{\dR})=-2\partial\bar \partial x\otimes \bfone(p)_{\dR},
\end{displaymath}
and for $n\ge 2p$ by
\begin{displaymath}
  d_{\fD}(x\otimes \bfone(p)_{\dR})=d x\otimes \bfone(p)_{\dR}.
\end{displaymath}
  \end{df}

  The relationship between this definition and the one appearing in
  \cite[Definition 2.5]{Burgos:CDB} is that, if $n<2p$ and
  $\omega(p-1)\in \fD^{n}(A,p)$, then 
  \begin{displaymath}
    \omega \in (2\pi i)^{p-1}A^{n-1}_{\R}\cap \bigoplus_{n-p\le p'<p}A^{p',n-1-p'}_{\C},
  \end{displaymath}
  while, if $n\ge 2p$ and $\omega(p)\in \fD^{n}(A,p)$, then
  \begin{displaymath}
    \omega \in (2\pi i)^{p}A^{n}_{\R}\cap \bigoplus_{p\le p' \le n-p}A^{p',n-p'}_{\C},    
  \end{displaymath}

  \begin{rmk}
    In the Betti picture, the differential in degree $2p-1$ is given
    as
    \begin{displaymath}
      d_{\fD}(x\otimes \bfone(p-1)_{\Betti})=\frac{i}{2\pi }\partial\bar
      \partial x\otimes \bfone(p)_{\Betti}.
    \end{displaymath}
    Note that $\frac{i}{2\pi }\partial\bar \partial$ is the $dd^{c}$
    operator appearing in pluripotential theory (see, for instance
    \cite{Demailly:cadg} ) and differs by a factor of 2 to the one
    used in \cite{GilletSoule:ait}.
  \end{rmk}

  If $A$ is a (graded commutative) Dolbeault algebra, with product $\wedge$, then the Deligne complex comes equipped with a
  graded commutative product $\fD^{n}(A,p)\otimes \fD^{m}(A,q)\to
  \fD^{n+m}(A,p+q)$ that is associative up to homotopy. See \cite[Thm.\,3.3]{Burgos:CDB} 
  for the explicit formulas.

  \begin{ex}\label{exm:10} 
    Let $X$ be a smooth complex variety, $D_{X}^{\ast}$ the weak Hodge
    complex of currents on $X$ and
    $\fD_{D}^{\ast}(X,p)=\fD^{\ast}(D^{\ast}_{X},p)$ the Deligne
    complex of currents on $X$. Let $f$ be a rational function on $X$
    and $D=\Div(f)$. Then $\delta _{D}\in \fD_{D}^{2}(X,1)$, and
    $[\log |f|]\in \fD_{D}^{1}(X,1)$ satisfying $d_{\fD}[\log |f|]=\delta
    _{D}$.   
  \end{ex}

\subsection{Higher Chow groups}\label{subsec:2.2}
In this part, we give a very short introduction to the cubical theory
of Bloch higher cycles. For 
more details the reader can consult \cite{Levine:BhCgr} and \cite[\S
1.7, 1.8]{BGGP:Height}.  
Fix a base field $k$ and let $\P^1$ be the projective line over
$k$. Let $ \square = \P^1\setminus \{1\}\,(\cong \A^1).$
The collection of all cartesian
products $(\P^1)^{n}$, $n\ge 0$ has a cocubical scheme structure,
denoted $(\P^1)^{\bullet}$.
For $i=1,\dots,n$, we denote by $t_{i}\in (k\cup\{\infty\})$  the
absolute coordinate of the $i$-th factor. Then the 
coface maps
are defined as
\begin{align*}
  \delta_0^i(t_1,\dots,t_n) &= (t_1,\dots,t_{i-1},0,t_{i},\dots,t_n), \\
\delta_1^i(t_1,\dots,t_n) &= (t_1,\dots,t_{i-1},\infty,t_{i},\dots,t_n).
\end{align*}
While the codegeneracy maps are given by
\begin{displaymath}
  \sigma ^{i}(t_1,\dots,t_n)=(t_1,\dots,t_{i-1},t_{i+1},\dots,t_n).
\end{displaymath}
Then, $\square^{\bullet}=\{\square^{n}\}_{n\ge 0}$ inherits a cocubical scheme structure from
that of $(\P^1)^{\bullet}$.
An $r$-dimensional \emph{face} $F$ of $\square^n$ is any subscheme of the form
\begin{displaymath}
  F=\delta^{i_1}_{j_1}\cdots
  \delta^{i_{n-r}}_{j_{n-r}}(\square^{r}).
\end{displaymath}
By convention, $\square^n$ is a
face of dimension $n$. The codimension of an $r$-dimensional face of
$\square ^{n}$ is $n-r$.  

Let $X$ be an equidimensional quasi-projective scheme of
dimension $d$ over the field $k$. Let $Z^{p}(X,n)$ be the free
abelian group generated by the codimension $p$ closed irreducible
subvarieties of $X\times \square^{n}$, which intersect properly
$X\times F$ for every face $F$
of $\square^n$. We call the elements of $Z^p(X,n)$
\textit{admissible pre-cycles}. 
The pullback by the coface and codegeneracy maps of $\square^{\bullet}$
endow $Z^{p}(X,\cdot)$ with a cubical abelian group structure, given by
\begin{displaymath}
\begin{gathered}
  \delta ^{j}_{i}=(\delta ^{i}_{j})^{\ast},\\
  \sigma _{i}=(\sigma ^{i})^{ \ast}.
\end{gathered}
\end{displaymath}
Note that the indexes have been raised or lowered to reflect the change from
cocubical to cubical structures. The total differential is defined as
\begin{displaymath}
    \delta= \sum_{i=1}^n \sum_{j=0,1}(-1)^{i+j}\delta_i^j.
\end{displaymath}
Let
$(Z^{p}(X,*),\delta)$ be the associated chain complex. Since we are
working with a cubical abelian group instead of a simplicial one, it
is known that the homology groups of the complex $(Z^{p}(X,*),\delta)$
do not agree with the higher Bloch Chow groups. Therefore we have to 
consider the \emph{normalized} and \emph{refined normalized} chain
complexes associated with $Z^p(X,*)$, 
 \begin{eqnarray*}
Z^{p}(X,n)_{0}\coloneqq \bigcap_{i=1}^{n} \ker
\delta^1_i,\\
Z^p(X,n)_{00} \coloneqq \bigcap_{i=1}^{n} \ker \delta^1_i\cap
\bigcap_{i=2}^{n} \ker \delta^0_i.
\end{eqnarray*}
The differential of these normalized complexes is the ones induced by
$\delta $  and are also denoted by $\delta$.
 The inclusion of cubical complexes
\begin{displaymath}
(Z^{p}(X,n)_{00}, \delta)\hookrightarrow (Z^{p}(X,n)_{0},\delta),
\end{displaymath}
is a quasi-isomorphism and the homology of either of the above chain
complexes agree with the Bloch higher Chow group:
\begin{displaymath}
\CH^{p}(X,n)= \rH_{n}(Z^{p}(X,\ast)_{00},\delta)= \rH_{n}(Z^{p}(X,\ast)_{0},\delta).
\end{displaymath}
Here the symbol $=$ means a canonical isomorphism. An element $Z$ of
the group $Z^{p}(X,n)_{0}$ will be called a \emph{normalized pre-cycle} while an
element $Z$ of $Z^{p}(X,n)_{00}$ will be called a \emph{refined
  pre-cycle}. In both cases, if $\delta Z=0$ it will be called a cycle
(normalized or refined respectively).  We note that a refined cycle
$Z\in Z^{p}(X,n)_{00}$ satisfies all the conditions
\begin{displaymath}
  \delta ^{j}_{i}(Z)=0,\quad j=0,1,  \ i=1,\dots,n.
\end{displaymath}
When $j=1$ or $j=0$ and $i>1$ the condition follows from the
definition of refined pre-cycle. While the condition for $j=0$ and
$i=1$ follows from being a cycle.

We also recall the definition of proper intersection of higher cycles. 
\begin{df}\label{def:10}
Let $X$ be a smooth quasi-projective scheme over $k$, and let
  $p,q,n,m\ge 0$ be non-negative
integers. If  $Z\in Z^{p}(X,n)$, 
$W\in Z^{q}(X,m)$, we say that $Z$ and $W$
\emph{intersect properly} if, for any
face $F$ of $\square ^{n+m}$,
\begin{displaymath}
  \codim_{X\times F}\left( \pi_{1}^{-1}|Z| \, \cap \pi_{2}^{-1}|W|\,
    \cap \, (X\times
  F)\right)\ge p+q,
\end{displaymath}
where
\begin{displaymath}
  \pi_1\colon X\times \square ^{n}\times \square^{m}\to X\times
  \square ^{n},\quad 
  \pi_2\colon X\times \square ^{n}\times \square^{m}\to X\times
  \square ^{m}
\end{displaymath}
are the projections.
\end{df}
\subsection{Wang's form and Goncharov's regulator}\label{Gon-reg}
Let $X$ be a smooth and complex variety. In his paper \cite{Gon95}, Goncharov defined a regulator
\begin{displaymath}
  \caP\colon \CH^{p}(X,n)\longrightarrow \rH^{2p-n}_{\fD}(X,\R(p))
\end{displaymath}
which is given by a morphism of complexes, also denoted by 
\begin{displaymath}
  \caP\colon Z^{p}(X,\ast)_{0}\to
  \fD^{2p-\ast}_{D}(X,p),
\end{displaymath}
where $\fD^{\ast}_{D}(X,p)$ denotes the Deligne complex of currents on
$X$ as in Example \ref{exm:10}. It can be shown (\cite[\S
3.1.1]{Kerr:03} and \cite{KLM:06}, see also \cite[Theorem
7.8]{BFT}) that Goncharov's and Beilinson's regulators agree. We will
briefly describe Goncharov's construction. For a quasi-projective
smooth complex variety $X$ we write $\fD_{\log}^{\ast}(X,p)$ for the
Deligne complex of differential forms with logarithmic singularities
at infinity. 

The building block is
provided by the differential form
\begin{displaymath}
  \lambda=-\frac{1}{2}\log (t\bar
  t)=-\frac{1}{2}\log (t\bar
  t)\otimes \bfone(0)_{\dR}\in\fD^{1}_{\log}\Big(\C^{\ast},1\Big).
\end{displaymath}
Let
$\pi_{i}\colon (\P^1)^{n}\to \P^1$ be the projection to the $i$-th
component. Then the Wang's forms are defined as 
\begin{displaymath}
W_{n}(n-1)\coloneqq \Alt (\pi^{\ast}_{1}\lambda\bullet(\cdots
(\pi^{\ast}_{n-1}\lambda\bullet \pi^{\ast}_{n}\lambda))\in
\fD^{n}_{\log}((\C^\ast)^{n}, n),
\end{displaymath}
where $\Alt$ means the alternate sum for all possible permutations of
$n$ elements divided by $n!$. This is needed because the product in
the Deligne complex is not associative.

By the definition of the Deligne complex $W_{n}\in
E^{n-1}_{(\P^{1})^{n},\R}(\log B)$, where $B=(\P^{1})^{n}\setminus
(\C^{\ast})^{n}$ is the normal crossing divisor of points of
$(\P^{1})^{n}$ with at least one coordinate equal to zero or infinity.

Consider the differential form
\begin{displaymath} \alpha
\coloneqq(-1)^{n}\left(\frac{dt_{1}}{t_{1}}\wedge\cdots\wedge
\frac{dt_{n}}{t_{n}}\right)\in F^{n}E^{n}_{(\P^1)^{n}}(\log
B).
\end{displaymath}

The main properties of the Wang forms are that they satisfy the
differential equations
\begin{displaymath}
  d_{\fD}W_{n}=0
\end{displaymath}
and
\begin{equation}\label{eq:34}
  d_{\fD}[W_{n}(n-1)] =\sum_{j=0}^{1} \sum_{i=1}^{n}(-1)^{i+j}(\delta _{j}^{i})_{\ast}[W_{n-1}(n-2)]. 
\end{equation}
Moreover
\begin{displaymath}
  dW_n=\frac{1}{2}(\alpha +(-1)^{n-1}\bar \alpha )
\end{displaymath}
and 
\begin{equation}\label{eq:53}
  d[W_{n}(n-1)] =\frac{1}{2}([\alpha(n-1)] +\overline{[ \alpha(n-1)]} )+ \sum_{j=0}^{1} \sum_{i=1}^{n}(-1)^{i+j}(\delta _{j}^{i})_{\ast}[W_{n-1}(n-2)]. 
\end{equation}

 By abuse of notation we will also denote by $W_{n}$ the 
pull-back of $W_{n}$ to any variety of the form $X\times
(\P^{1})^{n}$.
If $Z$ is an irreducible subvariety of $X\times
\square^{n}$ intersecting properly all the faces and $\widetilde Z$
is a resolution of singularities of the closure $\overline Z$, then
using the fact that $W_{n}$ vanishes when one of the coordinates
$t_{i}=1$, we get that the pull-back of $W_{n}$ to $\widetilde Z$ is locally integrable.  
Therefore, for any cycle $Z\in Z^{p}(X,\ast)_{0}$, we have a well defined current
\begin{displaymath}
\delta _{Z}\bullet W_{n}(n-1)\coloneqq \delta_{Z}\wedge W_{n}(n-1)\in
\fD^{2p+n}_{D}(X\times (\P^{1})^{n}, p+n). 
\end{displaymath}
Then, Goncharov regulator current is given by (see \cite[\S 6.1]{Gon95})
\begin{equation}\label{eq:30}
  \caP(Z)=(\pi _{X})_{\ast}(\delta _{Z}\bullet  W_{n}(n-1))\in \fD_{D}^{2p-n}(X,p)
\end{equation}
where $\pi _{X}\colon X\times (\P^{1})^{n}\to X$ is the
projection. The fact that $\caP$ is a morphism of complexes follows
from equation \eqref{eq:34}.

\subsection{Extension class and differential forms}
Let $X$ be a smooth projective complex variety.  
Write $H=\rH^{2p-n-1}(X; \R(p))$.
Given a refined cycle $Z\in Z^{p}(X,n)_{00}$, under the isomorphism 
\begin{displaymath}
\Ext^{1}_{MHS}(\R(0), H)\cong \rH^{2p-n}_{\fD}(X, \R(p)),
\end{displaymath}
where the right hand side is the real Deligne cohomology of $X$,
the real Beilinson regulator
of $\rho(Z)\in \rH^{2p-n}_{\fD}(X, \R(p))$ can be interpreted as an
element in $\Ext^{1}_{MHS}(\R(0), H)$. 
Attached to this extension, there are
differential forms $\{\eta_{Z}, \theta_{Z}, g_{Z}\}$ satisfying
several properties, including the differential equation 
\begin{displaymath}
dg_{Z}=\frac{1}{2}(\eta_{Z}+(-1)^{p-1}\overline{\eta}_{Z})-\theta_{Z}.
\end{displaymath}
See \cite[\S 3]{BGGP:Height} (and Proposition \ref{prop-difformZ}
later) for a detailed description of these forms.
The class of $\theta_{Z}$ represents the real regulator class
of the cycle $Z$ (\cite[Proposition 3.6, Proposition
3.8]{BGGP:Height}).
In case that the
real regulator class of $Z$ is trivial, by \cite[Corollary
3.9]{BGGP:Height}, one can choose $\theta_{Z}=0$, and correspondingly
$\eta_{Z}=2\partial g_{Z}$.  

\section{Height of a Mixed Hodge Structure}\label{sec:deligne-splitting-height}
In this section, we introduce a notion of \emph{framed mixed
Hodge structures} and generalize the height of a biextension of
Hodge structures to define the height of a framed mixed Hodge
structure. In fact we give two such generalizations that agree (up to
a trivial normalization factor) for biextensions of Hodge structures
and mixed Hodge structures with the shape \eqref{eq:1}, but that differ 
for more general mixed Hodge structures. 

\subsection{The Deligne splitting}
\label{sec:deligne-splitting}
We recall some  basic properties of mixed Hodge structures.
Let $V$ be a $\Q$-vector space. 
A mixed Hodge structure $H=(F,W)$ on $V$ induces a unique functorial bigrading
\cite[Theorem 2.13]{CKS:dhs}
\begin{equation}
    V_{\C } = \bigoplus_{p,q}\, I^{p,q}     \label{deligne-bigrading}
\end{equation}  
of the underlying complex vector space $V_{\C }$ such that
\begin{enumerate}
\item $F^p= \oplus_{p'\geq p,q}\, I^{p',q}$;
\item $W_k = \oplus_{p+q \leq k}\, I^{p,q}$;
\item $\overline{I^{p,q}}
  \equiv I^{q,p} \mod\oplus_{p'<p,q'<q}\, I^{q',p'}$.
\end{enumerate}
The subspaces $I^{p,q}$ are given by
\begin{equation}
  \label{eq:72}
  I^{p,q}=F^{p}\cap W_{p+q}\cap \left(\overline{F^{q}}\cap W_{p+q} +
    \overline{U^{q-1}_{p+q-2}}\right), 
\end{equation}
where
\begin{displaymath}
  U^{r}_{s}=\sum_{j\ge 0} F^{r-j}\cap W_{s-j}.
\end{displaymath}
When we need to specify the mixed Hodge structure we will denote the
different pieces as $I^{p,q}_{H}$ or as $I^{p,q}H$ if the symbol
denoting $H$ in a particular situation is too long. 
\begin{df}
The bigrading \eqref{deligne-bigrading} is called the
\emph{Deligne bigrading} of $H$. The associated semisimple endomorphism
$Y = Y_{H}$ of $V_{\C }$ which acts as multiplication by $p+q$ on
$I^{p,q}$ is called the \emph{Deligne grading} of $H$.  
\end{df}
 We denote by $\Pi _{p,q}\colon V_\C \to  V_\C$ the projector of $V_{\C}$ that is the
 identity on $I^{p,q}$ and zero on $I^{p',q'}$ for $(p',q')\not =
 (p,q)$ and write $\Pi _{k}\coloneqq\sum_{p+q=k}\Pi _{p,q}$. We denote also by
 $\pi _{k}\colon V_{\C}\to \Gr^{W}_{k}V_{\C}$ the projection induced
 by $\Pi _{k}$ and the isomorphism $\bigoplus _{p+q=k}I^{p,q}\simeq
 \Gr^{W}_{k}V_{\C}$.  We also denote by $\iota_{k}\colon
 \Gr^{W}_{k}V_{\C}\to V_{\C}$ the monomorphism induced by the inverse of
 the previous isomorphism. In particular
 \begin{displaymath}
   \Pi _{k}=\iota_{k}\circ \pi _{k}.
 \end{displaymath}

\begin{rmk} In general, the maps $\pi _{k}$ and $\iota_{k}$ are not
  compatible with the real structures of $V_{\C}$ and $\Gr_{k}^{W}V_{\C}$
  unless the mixed Hodge structure is split over $\R$
\end{rmk}
The semisimple endomorphism $Y$ is given by
\begin{equation}\label{eq:37}
  Y=\sum _{k\in \Z} k \Pi _{k}.
\end{equation}
Let
\begin{equation}
     \gl(V_{\C})^{a,b} = \{\,\alpha\in \gl(V_{\C})
        \mid \alpha(I^{p,q})\subseteq I^{p+a,q+b}\,\}
     \label{gl-hodge-comp}
\end{equation}
be the Hodge decomposition of $\gl(V)$ and define
\begin{equation}
     \Lambda^{-1,-1} = \bigoplus_{a<0,b<0}\, \gl(V_{\C })^{a,b}.
     \label{lambda-def}
\end{equation}
Then, $\overline{\Lambda^{-1,-1}} = \Lambda^{-1,-1}$
\cite[Eq.~2.19]{CKS:dhs}.
For an element $\lambda \in \gl(V_{\C})$ we denote $\lambda =\sum
\lambda ^{a,b}$ its decomposition into Hodge components. 

There exists a unique
real element $\delta = \delta_{H}\in\Lambda^{-1,-1}$ such that 
\begin{equation}
       \overline{Y_{H}} = e^{-2i\delta}\cdot Y_{H}  \label{delta-def}
\end{equation}
where $g\cdot \alpha \coloneqq \Ad(g)\alpha$ denotes the adjoint action of
$\GL(V_{\C })$
on $\gl(V_{\C })$, \cite[Prop 2.20]{CKS:dhs}. The element
$\delta$ defined by \eqref{delta-def}
is called the \emph{Deligne splitting} of $H$.

\begin{ex}\label{exm:7}
  Let $H$ be a mixed Hodge structure and $\overline H$ the complex
  conjugate as in Example \ref{exm:6}. Then, it follows from the
  definition of the Deligne splitting that
  \begin{displaymath}
    I^{p,q}_{\overline H}= \overline {I^{p,q}_{H}}.
  \end{displaymath}
  It follows that $Y_{\overline H}=\overline {Y_{H}}$, and therefore
  $\delta _{\overline H}=-\delta _{H}$.  
\end{ex}

For an element $g\in \GL(V_{\C})$ we denote by $g\cdot F$ the
filtration given by $(g\cdot F)^{p}V_{\C}=g(F^{p}V_{\C})$. In general if
$(F,W)$ is a mixed Hodge structure on $V$, the pair of filtrations
$(g\cdot F,W)$ do not form a mixed Hodge structure. However, we have
the following lemma. 

\begin{lem}[{\cite[Lemma 4.11]{Pearlstein:vmhshf}}]
  \label{lambda-equiv} Let $H=(F,W)$ be a 
  mixed Hodge structure on $V$ and $\Lambda^{-1,-1}$ the associated
  subalgebra 
  \eqref{lambda-def}.  Then, $\lambda\in\Lambda^{-1,-1}$ implies that
   $(e^{\lambda}\cdot F,W)$ is a mixed Hodge structure on $V$ and that
   \begin{equation} \label{eq:2}
     I^{p,q}_{(e^{\lambda}\cdot F,W)} = e^{\lambda}(I^{p,q}_{(F,W)}).
   \end{equation}
 \end{lem}
 \begin{proof}
   The condition that the pair of filtrations $(e^{\lambda }\cdot F, W)$
   is a mixed Hodge structure is that, for every $n$, the filtration induced by
   $e^{\lambda }\cdot F$ on $\Gr_{n}^{W}V$ defines a pure Hodge structure of
   weight $n$. Since $\lambda \in \Lambda _{-1,-1}$, the filtrations
   induced by $F$ and by $e^{\lambda }\cdot F$ on $\Gr_{n}^{W}V$ are the
   same. So if $F$ 
   induces a pure Hodge structure in each $\Gr_{n}^{W}V$ the same is
   true for $e^{\lambda }\cdot F$. The equation \eqref{eq:2} is
   \cite[Lemma 4.11]{Pearlstein:vmhshf}.
 \end{proof}

 \begin{cor}\label{cor:2}
   Let $H=(F,W)$ be a  mixed Hodge structure on $V$, $\lambda \in
   \Lambda^{-1,-1}$ and $\hat H=(e^{\lambda }\cdot F,W)$ the new mixed
   Hodge structure. Denote by $\hat Y$, $\hat \Pi _{k}$, $\hat \pi
   _{k}$ and $\hat \iota_{k}$ the operators defined using the
   splitting of $\hat H$. Then
   \begin{displaymath}
     \hat Y = e^{\lambda }\cdot Y,\qquad
     \hat \Pi _{k} = e^{\lambda }\cdot \Pi _{k},\qquad
     \hat \pi _{k} = \pi _{k}\circ e^{-\lambda },\qquad
     \hat \iota _{k} = e^{\lambda }\circ \iota _{k}.
   \end{displaymath}
 \end{cor}

 \subsection{Framed mixed Hodge structures}
\label{sec:framed-mixed-hodge}
We now introduce the type of mixed Hodge structures that we will encounter later on. 

\begin{df}\label{framed-mhs}
Given a mixed Hodge structure $H=(F,W)$ on $V$, an $(a,b)$-\textit{framing} of
$H$ is the data of morphisms
\begin{displaymath}
  \phi\colon
  \Q(-a)\rightarrow \Gr^{W}_{2a}(V)\text{ and }
  \psi\colon \Gr^{W}_{2b}(V)\rightarrow \Q(-b)
\end{displaymath}
of pure Hodge structures.
A mixed Hodge structure with an $(a,b)$-framing will be called
$(a,b)$-\emph{framed} or just \emph{framed} if the weights are clear from the
context.  Sometimes, an $(a,b)$-framed mixed Hodge $(H,\phi ,\psi )$ will be
denoted just by the mixed Hodge structure $H$ or if we want to
emphasize the involved weights by $H_{a,b}$.
\end{df}

\begin{rmk}
This notion of framing is not new and was introduced in \cite[\S 1.3.4]{BGSV}, albeit for mixed Hodge-Tate structures. More generally, it has also appeared in \cite{AJ03}. There could be other references that the authors are not aware of. We present it in a way that is more suitable for our later purpose of attaching such mixed Hodge structures to pairs of higher cycles.
\end{rmk}

\begin{rmk}
The oriented mixed Hodge structure defined in \cite[Definition 2.3]{BGGP:Height} is an example of a framed mixed Hodge structure,
where the homomorphisms $(\phi,\psi)$ are isomorphisms.
\end{rmk}

We denote by $\caH^{a}_{b}$ the set of all $(a,b)$-framed mixed
Hodge structures. Given two framed mixed Hodge
structures $(H, \phi,\psi)$ and $H'=(H',\phi', \psi')$ in
$\caH^{a}_{b}$, a morphism $f\colon H\to H'$ of 
mixed Hodge structures is called a \textit{framed} morphism if
$\phi'=\Gr_{2a}^{W}(f)\circ \phi$ and
$\psi=\psi'\circ \Gr_{2b}^{W}(f)$.
\begin{ex}
  Let $(H,\phi ,\psi )$ be a framed mixed Hodge structure and $f\colon
  H\to H'$ be an isomorphism of mixed Hodge structures. Defining
  $\phi'=\Gr_{2a}(f) \circ \phi $ and $\psi '=\psi \circ
  \Gr_{2b}(f^{-1})$ we turn $f$ into an isomorphism of framed mixed
  Hodge structures. 
\end{ex}

\begin{df}\label{rmk-dual-frm}
Let $(H,\phi ,\psi )$ be an $(a,b)$-framed mixed Hodge structure on $V$.
\begin{enumerate}
\item The  dual $(H^{\vee},\psi ^{\vee},\phi ^{\vee})$ is the $(-b,-a)$-framed mixed Hodge
  structure on $V^{\vee}$ given by
  \begin{align*}
    \psi^{\vee}&\colon \Q(b)=\Q(-b)^{\vee}
    \to \Gr^{W}_{2b}(V)^{\vee}= \Gr^{W}_{-2b}(V^{\vee}),\\
    \phi^{\vee}&\colon
    \Gr^{W}_{-2a}(V^{\vee})=\Gr^{W}_{2a}(V)^{\vee}\to \Q(-a)^{\vee}=\Q(a).
  \end{align*}
\item Let $p$ be an integer, then the twisted framed mixed Hodge
  structure  $H(p)$ is the $(a-p,b-p)$-framed mixed Hodge structure
  \begin{align*}
    \phi \otimes \Id&\colon \Q(-a+p) =\Q(-a)\otimes
    \Q(p)\longrightarrow \Gr^{W}_{2a}(V)\otimes
    \Q(p)= \Gr^{W(p)}_{2a-2p}(V),\\
    \psi\otimes \Id &\colon \Gr^{W(p)}_{2b-2p}(V)=
    \Gr^{W}_{2b}(V)\otimes \Q(p)
         \longrightarrow \Q(-b)\otimes
    \Q(p)=\Q(-b+p).
\end{align*}
By abuse of notation we will denote $\phi \otimes \Id$ and
$\psi\otimes \Id$ by $\phi $ and $\psi $.
\end{enumerate}
\end{df}

A framing together with the Deligne splitting determine elements in
the mixed Hodge structure and its dual.

\begin{df}\label{def:9} Let $H=(F,W, \phi,\psi)$ be an $(a,b)$-framed
  mixed Hodge structure. Let $\bfone(-a)_{\Betti}$ be the Betti generator of
  $\Q(-a)$.
  We denote by $e_{H}$ the element
  \begin{displaymath}
    e_{H}=\iota_{2a}\circ \phi (\bfone(-a)_{\Betti}) 
  \end{displaymath}
  This element belongs to $I^{a,a}_{H}$.
\end{df}

\begin{rmk}\label{rem:1}
  The element $e_{H}$ is characterized by the conditions
  \begin{enumerate}
  \item $e_{H}\in I_{H}^{a,a}\subset W_{2a}V_{\C}$.
  \item The class $[e_{H}]\in \Gr_{2a}^{W}V_{\C}$ satisfies
    $[e_{H}]=\phi (\bfone(-a)_{\Betti})$.
  \end{enumerate}
  In other words, $e_{H}$ is the unique lifting of $\phi
  (\bfone(-a)_{\Betti})\in \Gr_{2a}^{W}V_{\C}$ 
  that belongs to $I^{a,a}_{H}$.
\end{rmk}

  Applying Definition \ref{def:9} to the dual framed mixed Hodge structure of
  Remark \ref{rmk-dual-frm} we have an element $e_{H^{\vee}}\in
  I^{-b,-b}_{H^{\vee}}$.

  \begin{rmk}\label{rem:11}
    Conversely, if $e_H$ is an element of $I^{a,a}_{H}$ whose image
    in $\Gr_{2a}^{W}H$ belongs to $\Gr_{2a}^{W}H_{\Betti}$, then we can
    define $\phi \colon \Q(-a)\longrightarrow \Gr_{2a}^{W}H$ as the
    unique map that sends $\bfone(-a)_{\Betti}$ to $e_{H}$. Similarly
    $e_{H^{\vee}}$ determines the map $\psi $.  Therefore, we can
    alternatively  define framed mixed Hodge structure using the
    elements $e_{H}$ and $e_{H^{\vee}}$. Namely,  an $(a,b)$-framed mixed
    Hodge structure can be defined as a triple $(H,e_{H},e_{H^{\vee}})$, 
  where $H$ is a mixed Hodge structure, $e_{H}$ is an element of $I^{a,a}_{H}$ whose image
  in $\Gr_{2a}^{W}H$ belongs to  $\Gr_{2a}^{W}H_{\Betti}$  and $e_{H^{\vee}}$ is
  an element of $I^{-b,-b}_{H^{\vee}}$ whose image in
  $\Gr_{-2b}^{W}H^{\vee}$ belongs to  $\Gr_{-2b}^{W}H^{\vee}_{\Betti}$. 

  With this notation some formulae and properties become more
  transparent. For instance 
  if $H=(H,e_{H},e_{H^{\vee}})$ is a framed mixed Hodge structure,
  then the dual one is
  $(H^{\vee},e_{H^{\vee}},e_{H})$, while the twisted framed mixed
  Hodge structure $H(n)$ is given by
  $(H(n),e_{H}\otimes \bfone(n)_{\Betti},e_{H^{\vee}}\otimes \bfone(-n)_{\Betti})$. 
\end{rmk}

\subsection{Heights of framed mixed Hodge structures}
\label{sec:heights-framed-mixed}
With a framed mixed Hodge structure, one can associate two notions of
height. We will now introduce these notions.

\begin{df}\label{height-framed}
Let $H=(F,W, \phi,\psi)\in \caH^{a}_{b}$ be a framed mixed Hodge
structure, such that $b< a$.
Then the first height function is defined as 
\begin{equation}\label{eq:18}
  \Ht_{1}(H)=\im\langle e_{H^{\vee}}, \overline {e_{H}}\rangle
\end{equation}
where $\langle -,-\rangle $ denotes the duality between $H$ and
$H^{\vee}$. 
\end{df}

\begin{rmk}
  Since we are assuming $b<a$, we have that $\langle
  e_{H^{\vee}},e_{H}\rangle =0$. The fact that the height $\Ht_{1}$
  may be non zero is related to the fact that the Deligne splitting is
  not compatible with complex conjugation, hence, in general
  $\overline{e_{H}}\not \in I^{a,a}_{H}$. In particular, if the mixed
  Hodge structure is split over $\R$ then the height is zero.  
\end{rmk}

We next show that the function $\Ht_{1}$ is related to the exponential
of the Deligne splitting.
\begin{prop}\label{prop-ht-1-delta}
Let $\delta_{H}$ be the Deligne splitting associated with the framed mixed Hodge structure $H$. Then 
\begin{displaymath}
  \overline{e_{H}}=
  e^{-2i\delta_{H}}(e_{H}).
\end{displaymath}
\end{prop}
\begin{proof}
  The element $e_{H}$ satisfies $Y(e_{H})=2a\, e_{H}$. Its conjugate $\overline {e_{H}}$ is
  uniquely characterized by the properties 
\begin{enumerate}
\item\label{it-1}
  $\overline{e_{H}}\in W_{2a}(V_{\C}),$
\item\label{it-2} $\pi _{2a}(\overline {e_{H}})=\pi _{2a}(e_{H})$
\item\label{it-3}
$\overline{Y}(\overline{e_{H}})=2a\overline{e_{H}}$
\end{enumerate}
 Clearly $\overline{e_{H}}$ satisfies properties \ref{it-1} and
\ref{it-3}. Property \ref{it-2} follows from the computation
\begin{displaymath}
  \pi _{2a}(\overline {e_{H}})=
  \overline {\pi _{2a}(e_{H})}
  =\overline{\phi(\bfone(-a)_{\Betti})}=\phi(\bfone(-a)_{\Betti})
  =\pi _{2a}(e_{H}).
\end{displaymath}
 For
  the first equality, note that although the map $\pi _{k}\colon
  V_{\C}\to \Gr_{k}^{W}V_{\C}$ may not be compatible with complex
  conjugation, the restriction $\pi _{k}\colon
  W_{k}V_{\C}\to \Gr_{k}^{W}V_{\C}$ always is.
  Let now $x\in V_{\C}$ be an element satisfying conditions \ref{it-1}
  to \ref{it-3}. Conditions \ref{it-1}
  and \ref{it-3} imply that
  \begin{displaymath}
    x\in \bigoplus _{p+q=2a}\overline{I^{p,q}}.
  \end{displaymath}
  Since $\pi _{2a}$ restricts to an isomorphism
  \begin{displaymath}
    \bigoplus _{p+q=2a}\overline{I^{p,q}} \overset{\simeq}{\longrightarrow}
    \Gr_{2a}^{W}(V_{\C})
\end{displaymath}
condition \ref{it-2} nails down the equality $x=\overline{e_{H}}$.

We now show that the element $e^{-2i\delta_{H}}(e_{H})$ also
  satisfies properties \ref{it-1}, \ref{it-2} and
  \ref{it-3}. Properties \ref{it-1} and \ref{it-2} are immediate since
  $e^{-2i\delta_{H}}$ preserves the weight filtration and induces the
  identity on the graded vector space $\Gr^{W}_{\ast}$. For property
  \ref{it-3}, notice that
\begin{multline*}
  \overline{Y}\left(e^{-2i\delta_{H}}(e_{H})\right)=
  e^{-2i\delta_{H}} \circ Y\left(e^{2i\delta_{H}}\circ e^{-2i\delta_{H}}(e_{H})\right)\\
  =e^{-2i\delta_{H}} (Y(e_{H}))
  =2a e^{-2i\delta_{H}} (e_{H}).
\end{multline*}
 Since $\overline{e_{H}}$ is the unique element which satisfies these
 properties, we have the equality $\overline {e_{H}}= e^{-2i\delta_{H}} (e_{H})$. 
\end{proof}

\begin{cor}\label{cor:1}
  The height $\Ht_{1}$ is given by the formula
  \begin{displaymath}
  \Ht_{1}(H)=\im\langle e_{H^{\vee}},e^{-2i\delta_{H}} (e_{H})\rangle.
  \end{displaymath}
\end{cor}

Corollary \ref{cor:1} suggests that an alternative definition of height
can be obtained by using $\delta _{H}$ in the place of $e^{-2i\delta
  _{H}}$. This new definition would have the advantage of being real
on the nose, without the need of taking the imaginary part. 

 We need the following result:
\begin{prop}\label{prop-ht-2}
Let $\delta_{H}$ be the Deligne splitting associated with a mixed Hodge
structure $H=(F,W)$ on $V$. Then for every pair of integers $k' < k$,
and every $r>0$, the operator
\begin{displaymath}
  \pi_{k'}\circ \delta^{r}_{H} \circ \iota_{k}\colon
  \Gr_{k}^{W}V \longrightarrow \Gr_{k'}^{W}V 
\end{displaymath}
is real.
\end{prop}
\begin{proof}
Let $\hat{H}=(e^{-i\delta_{H}}\cdot F, W)$ be the associated
real-split mixed Hodge structure. By Corollary \ref{cor:2}, the
operators $\hat \iota_{k}$ and 
$\hat \pi _{k'}$ corresponding to $\hat H$ are given by the formulae 
$\hat{\iota}_{k}=e^{-i\delta_{H}}\circ \iota_{k}$ and
$\widehat{\pi}_{k'}= \pi_{k'}\circ
e^{i\delta_{H}}$. Since $\hat H$ is real split these operators are
real. Hence, so is $\hat{\pi}_{k'}\circ \delta_{H} \circ
\hat{\iota}_{k}$.
Putting the relations between the corresponding operators, we get
\begin{displaymath}
\hat{\pi}_{k'}\circ \delta^{r}_{H}\circ \hat{\iota}_{k}=\pi_{k'}\circ
e^{i\delta_{H}}\circ \delta^{r}_{H}\circ e^{-i\delta_{H}}\circ \iota_{k}
=\pi_{k'}\circ \delta^{r}_{H}\circ \iota_{k}.
\end{displaymath}
\end{proof}
This allows us to give the definition of the alternative height.
\begin{df}\label{def-ht-2} Let $H=(F,W, \phi,\psi)$ be 
an  $(a,b)$-framed mixed Hodge structure,
the \emph{second height function} is given by 
\begin{displaymath}
  \Ht_{2}(H)= \langle e_{H^{\vee}},\delta_{H} (e_{H})\rangle.
  \end{displaymath}
\end{df}
\begin{lem}\label{lem-2-17}
  \begin{displaymath}
    \Ht_{2}(H)\in \R.
  \end{displaymath}
\end{lem}
\begin{proof}
  Unraveling the definition of the elements $e_H$ and $e_{H}^{\vee}$
  in the definition of the height, we obtain
  \begin{displaymath}
    \Ht_{2}(H)= \langle\psi ^{\vee}(\bfone(b)_{\Betti}),
    \pi _{2b}\circ \delta _{H}\circ \iota _{2a}(\phi (\bfone(-a)_{\Betti}))\rangle,
  \end{displaymath}
  where now $\langle -,- \rangle $ denotes the duality between
  $\Gr_{2b}^{W}V_{\C}$ and $\Gr_{-2b}^{W}V^{\vee}_{\C}$.  Since
  $\phi (\bfone(-a)_{\Betti})$ and $\psi ^{\vee}(\bfone(b)_{\Betti})$ are real, the fact
  that $\Ht_{2}(H)\in \R$ is a consequence of Proposition
  \ref{prop-ht-2}.
\end{proof}

When a framed mixed Hodge structure is a generalized  biextension,
we have seen in \cite[Lemma 2.6]{BGGP:Height} that the two heights
agree up to a factor of $-\frac{1}{2}$.  More generally the following result
holds.

\begin{prop}\label{prop:1} Let $H\in \caH^{a}_{b}$ be a framed mixed Hodge
  structure satisfying $\delta _{H}^{3}(e_H)=0$. Then
  \begin{displaymath}
    \Ht_{2}(H)=-\frac{1}{2}\Ht_{1}(H).
  \end{displaymath}
\end{prop}
\begin{proof}
  The condition $\delta _{H}^{3}(e_{H})=0$ implies that
  \begin{displaymath}
    e^{-2i\delta _{H}} (e_{H}) = (e_{H}-2\delta _{H}^{2}(e_{H}))+ i(-2\delta _{H}(e_{H})).
  \end{displaymath}
  Therefore, using Proposition \ref{prop-ht-2}
  \begin{displaymath}
    \Ht_{1}(H)=\im\langle e_{H^{\vee}},
    e^{-2i\delta_{H}} (e_{H})\rangle
    =
    \langle e_{H^{\vee}},
    -2\delta_{H} (e_{H})\rangle
    =-2\Ht_{2}(H).
  \end{displaymath}
\end{proof}

However, the example below shows that this is not the case in general.
\begin{ex}\label{ex-ht}
Let $H(F,W)$ be a mixed Hodge structure on a vector space $V$, whose graded pieces
satisfy
\begin{displaymath}
  \Gr_{k}^{W}(V)\cong
  \begin{cases}
    \Q(0),&\text{ if }k=0,\\
    \Q(1),&\text{ if }k=-2,\\
    \Q(2),&\text{ if }k=-4,\\
   \Q(3),&\text{ if }k=-6,\\
    0,&\text{ otherwise.}
  \end{cases}
\end{displaymath}
With the $(0,-3)$-framing given by the isomorphisms
$\Q(0)\to \Gr_{0}^{W}(V)$ and $\Gr_{-6}^{W}(V) \to \Q(3)$.
Then 
\begin{displaymath}
  \Ht_{2}(H)=-\frac{1}{2}\Ht_{1}(H)-\frac{2}{3}
  \langle e_{H^{\vee}},\delta _{H}^{3}(e_{H})\rangle, 
\end{displaymath}
as can be shown by expanding $e^{-2i\delta_{H}}$.
\end{ex}

\begin{ex}\label{exm:13} We illustrate with the easiest example the
  relationship between the height of mixed Hodge structures we have
  defined and the classical height of cycles. Consider $X=\P^{1}$ and
  cycles $Z=\{0,\infty\}$ and $W=\{1,t\}$, with $t\not =
  0,1,\infty$. The mixed Hodge structure 
  $H=\rH^{1}(\P^{1}\smallsetminus Z,W)$ is of Tate--Hodge type and
  satisfies
  \begin{displaymath}
    \Gr_{2}^{W}(H)\cong \Q(-1),\qquad
    \Gr_{0}^{W}(H)\cong \Q(0)
  \end{displaymath}
  and $\Gr_{k}^{W}(H)=0$ for all $k\not = 0,2$. Let $f\colon \P^{1}\to
  \R$ be a smooth function such that $f(0)=f(1)=f(\infty)=0$ and $f(t)=1$. Then
  $H_{\dR}$ is generated by the classes $\{df\}$ and $\{dz/z\}$,
  with
  \begin{displaymath}
    W_{0}H_{\dR}=\C \{df\},\qquad F^{1} H_{\dR}=\C \{dz/z\}.
  \end{displaymath}
  Let
  $\gamma $ be a small circle around $0$ and $\eta$ a path between $1$
  and $t$. Then $\gamma $ and $\eta$  define generators of
  \begin{displaymath}
    H^{\vee}=\rH_{1}(\P^{1}\smallsetminus Z,W).
  \end{displaymath}
  The dual basis $\{\eta^{\ast},\gamma ^{\ast}\}$ is a $\Q$-basis of
  $H_{\Betti}$. The comparison isomorphism between Betti and de Rham
  is given by the equalities
  \begin{align*}
    \{dz/z\}&=2\pi i \gamma ^{\ast}+\log t \eta^{\ast},\\
    \{df\}&=\eta^{\ast}.
  \end{align*}
  It is easy to check that $\{dz/z\}\in I^{1,1}_{H}$. Then
  $\{dz/(2\pi iz)\}=\gamma ^{\ast}+\frac{1}{2\pi i}\log t \eta^{\ast}$
  is a Betti generator of $I^{1,1}_{H}$. Moreover, $\{df\}$ is a
  generator of $I^{0,0}$. Then
  \begin{displaymath}
    \overline{\{dz/(2\pi iz)\}}=\{dz/(2\pi iz)\}-\frac{1}{2\pi i}\log
    (t\bar t) \eta^{\ast}.
  \end{displaymath}
  So
  \begin{displaymath}
    \Pi _{0,0}(\overline{\{dz/(2\pi iz)\}})=-\frac{1}{2\pi i}\log
    (t\bar t) \eta^{\ast}.
  \end{displaymath}
  Therefore
  \begin{displaymath}
    \Ht_{1}(H)=\im(-\frac{1}{2\pi i}\log
    (t\overline t) )=\frac{1}{2\pi }\log(t\bar t),
  \end{displaymath}
  while
  \begin{displaymath}
    \Ht_{2}(H)=\frac{-1}{2\pi }\log(|t|),    
  \end{displaymath}
  Up to a normalization factor of $2\pi $ this agrees with the archimedean
  component of the height pairing between the cycles $Z$ and $W$.  
\end{ex}

\begin{rmk}\label{rem:14}
  We have defined the height of a mixed Hodge structure by comparing
  Betti generators. Similarly one can define other variants of the height
  by comparing de Rham generators. Let  $(H, \phi,\psi)$ be an $(a,b)$-framed
  mixed Hodge structure. Let $\bfone(-a)_{\dR}$ be the de Rham generator of
  $\Q(-a)$.
  We denote by $e_{H,\dR}$ the element
  \begin{displaymath}
    e_{H,\dR}=\iota_{2a}\circ \phi (\bfone(-a)_{\dR})=
    (2\pi i)^{a}e_{H}.
  \end{displaymath}
  Then we define
  \begin{displaymath}
    \Ht_{1,\dR}(H)=
    \begin{cases}
      \re\langle e_{H^{\vee},\dR},\overline{e_{H,\dR}}\rangle,&\text{ if $a-b$
                                                     odd},\\
      \im\langle e_{H^{\vee},\dR},\overline{e_{H,\dR}}\rangle,&\text{ if
                                                     $a-b$ even}.
    \end{cases}
  \end{displaymath}
  and
  \begin{displaymath}
    \Ht_{2,\dR}(H)=i\langle e_{H^{\vee},\dR},\delta_{H}(e_{H,\dR})\rangle.
  \end{displaymath}
  Then
  \begin{displaymath}
    \Ht_{1,\dR}(H)=
    \begin{cases}
      (-1)^{\frac{1-a-b}{2}}(2\pi )^{a-b}\Ht_{1}(H),&\text{ if $a-b$
                                                      odd},\\
      (-1)^{-\frac{a+b}{2}}(2\pi )^{a-b}\Ht_{1}(H),&\text{ if $a-b$
                                                     even},\\
    \end{cases}
  \end{displaymath}
  and
  \begin{displaymath}
    \Ht_{2,\dR}(H)= i(-1)^{a}(2\pi i)^{a-b}\Ht_{2}(H).
  \end{displaymath}
  Note that $\Ht_{1,\dR}(H)$ is still always real while
  $\Ht_{2,\dR}(H)$ is real or imaginary depending on the parity of
  $a-b$. This is chosen to make the normalization factors disappear in
  Remark \ref{rem:15}.

  Essentially the Betti and the de Rham heights carry the same
  information but the normalization factors are different. 
\end{rmk}

\begin{ex}
  Let $H$ be the $(1,0)$ framed mixed Hodge structure of Example
  \ref{exm:13}. Then
  \begin{displaymath}
    \Ht_{1,\dR}(H)=2\pi \Ht_{1}(H)=\log(t\bar t), \qquad
    \Ht_{2,\dR}(H)= -\log |t|
  \end{displaymath}
\end{ex}

We next compute both heights for the polylogarithm variation of mixed
Hodge structures. We will see that each height is related with a
different extension of the Bloch-Wigner dilogarithm to single valued
higher polylogarithms.

\begin{ex}\label{exm:2}
  We start computing the height $\Ht_{1}$. For the polylogarithm variation
  of mixed Hodge structures we follow \cite[\S
  1]{beilinsondeligne:_inter_zagier}. Let $\Li_{k}$ be the $k$-th
  polylogarithm function given, for $|z|<1$ by
  \begin{displaymath}
    \Li_{k}(z)=\sum_{n=1}^{\infty}\frac{z^{n}}{n^{k}}.
  \end{displaymath}
  and extended analytically to a multivalued function on
  $\C\smallsetminus \{0,1\}$. Fix $N>0$. Consider the square matrices
  with rows and columns indexed by $[0,N ]$ 
  \begin{displaymath}
    L(z)\coloneqq
    \begin{pmatrix}
      1& 0 & \cdots&&&\\
      -\Li_{1}(z) & 1 & 0& \cdots&&\\
      -\Li_{2}(z)&\log z & 1 & 0 & \cdots &\\
      -\Li_{3}(z)& \frac{(\log z)^{2}}{2!}&\log z & 1 & 0 & \cdots\\
      \vdots & \vdots  & \vdots &&\ddots &\ddots
    \end{pmatrix},
  \end{displaymath}
  $\tau (\lambda )=\diag (1,\lambda ,\lambda ^2,\lambda ^3,\dots)$ and
  $A(z)=L(z)\cdot \tau (2\pi i)$. Let $e_{0}$ be the matrix
  \begin{displaymath}
    e_{0}=
    \begin{pmatrix}
      0& 0 & 0 & 0 & 0 & \dots\\
      0& 0 & 0 & 0 & 0 & \dots\\
      0& 1 & 0 & 0 & 0 & \dots\\
      0& 0 & 1 & 0 & 0 & \dots\\
      0& 0 & 0 & 1 & 0 & \dots\\
      \vdots & \vdots  & \vdots & & \vdots &\ddots
    \end{pmatrix}
  \end{displaymath}
  and $\ell$ the column vector
  \begin{displaymath}
    \ell(z) \coloneqq ( -\Li_{1}(z),-\Li_{2}(z),\dots)^{T}.
  \end{displaymath}
  Then
  \begin{displaymath}
    L(z) =
    \begin{pmatrix}
      1 & 0\\
      \ell(z) & \Id
    \end{pmatrix}e^{\log (z) e_{0}}.
  \end{displaymath}
  We now consider the
  mixed $\Q$-Hodge
  structure $H(z)$ with
  \begin{gather*}
    H_{\Betti}=\Q^{[0,N]},\qquad H_{\dR}=\C^{[0,N]},\qquad \alpha =A(z)\\
    F^{-k}H_{\dR}= \C^{[0,k]},\qquad W_{-2k}H_{\dR}= \C^{[k,N]}
    ,\qquad W_{-2k}H_{\Betti}= \Q^{[k,N]},
  \end{gather*}
  where $\alpha $ is the comparison morphism between the Betti part
  $H_{\Betti}$ and the de Rham part $H_{\dR}$. These mixed Hodge structures
  glue together to define a variation of mixed Hodge structures. 
  Let $v_{k}$, $k=0,\dots,N$
  be the column vector with value $1$ in the $k$-th position and zero
  otherwise seen as an element of $H_\dR$. Similarly let $w_{k}$ the
  row vector with value $1$ in the $k$-th position and zero
  otherwise, seen as an element of $H_\dR^{\vee}$. 
  In this example and the next, since we are dealing with concrete
  matrices and complex numbers relative to a de Rham basis, in order
  to distinguish between the complex conjugation of complex numbers
  and the conjugation 
  with respect to the real Betti structure, we will use $x\mapsto \bar
  x$ for the former and $x\mapsto c(x)$ for the latter. 
  The  conjugation with respect to the Betti real structure is
  computed as follows. If $x\in H_{\dR}$, then
  \begin{displaymath}
    c(x)= A(z) \overline {A(z)^{-1}}\bar x, 
  \end{displaymath}
  while if $X\colon H_{\dR}\to H_{\dR}$ is a linear operator, then
  \begin{equation}\label{eq:31}
    c(X)=(A(z)\overline{A(z)^{-1}} )\overline X
    (A(z)\overline{A(z)^{-1}})^{-1}. 
  \end{equation}
  In particular $c(v_{i})=A(z)\overline{A(z)^{-1}} v_{i}$. The fact
  that the matrix $A(z)\overline{A(z)^{-1}}$ is lower triangular
  ensures that $v_{k}\in I^{-k,-k}$, and in fact, 
  $I^{-k,-k}\subset H_{\dR}$ is generated by the vector $v_{k}$. Since,
  for each $k$, $\Gr_{-2k}^{W}H(z)\simeq \Q(k)$, for every pair of
  integers $0\le a\le b \le N$ we obtain a $(-a,-b)$-framed mixed
  Hodge structure $H(z)_{-a,-b}$. In this framed mixed Hodge structure
  we have
  \begin{displaymath}
    e_{H(z)_{-a,-b}}=(2\pi i)^{a}v_{a},\qquad
    e_{H(z)_{-a,-b}^{\vee}}=(2\pi i)^{-b}w_{b}. 
  \end{displaymath}.
  With this
  notation   formula \eqref{eq:18} reads
  \begin{displaymath}
    \Ht_{1}(H(z)_{-a,-b})=\im\langle e_{H(z)_{-a,-b}^{\vee}},
    c(e_{H(z)_{-a,-b}})\rangle.
  \end{displaymath}
  We compute
  \begin{align*}
    \Ht_{1}(H(z)_{-a,-b})&=\im\langle e_{H(z)_{-a,-b}^{\vee}},
                          c(e_{H(z)_{-a,-b}})\rangle\\
                           & =\im\langle (2\pi i)^{-b}w_{b},
                           A(z) \overline {A(z)^{-1}(2\pi
                             i)^{a}v_{a}}\rangle \\
                         &=\im \frac{(-1)^{a}}{(2\pi i)^{b-a}}(A(z)\overline{A(z)^{-1}})_{b,a}, 
  \end{align*}
  where $(A(z)\overline{A(z)^{-1}})_{b,a}$ is the entry of the matrix
  $A(z)\overline{A(z)^{-1}}$ in row $b$,
  column $a$.  Thus our task
  now is to compute the entries of
  $A(z)\overline{A(z)^{-1}}$. 
  \begin{align*}
    A(z)\overline{A(z)^{-1}}
    &=L(z)\tau (2\pi i)\tau (-(2\pi i)^{-1})\overline {L(z)^{-1}}\\
    &=L(z)\tau (-1)\overline {L(z)^{-1}}\\
    &=  \begin{pmatrix}
      1 & 0\\
      \ell(z) & \Id
    \end{pmatrix}e^{\log (z) e_{0}}
      \tau (-1)e^{- \log (\bar z) e_{0}}
      \begin{pmatrix}
      1 & 0\\
      -\overline{\ell(z)} & \Id
    \end{pmatrix}.
  \end{align*}
  Since
  \begin{displaymath}
    \tau (-1)e^{- \log (\bar z) e_{0}}
    =e^{- \log (\bar z) \tau (-1)e_{0}\tau (-1)^{-1}}\tau (-1)
  \end{displaymath}
  and $\tau (-1)e_{0}\tau (-1)^{-1}=-e_{0}$, we obtain
  \begin{displaymath}
    A(z)\overline{A(z)^{-1}}=
    \begin{pmatrix}
      1 & 0\\
      \ell(z) & \Id
    \end{pmatrix}e^{\log (z\bar z) e_{0}}
      \tau (-1)
      \begin{pmatrix}
      1 & 0\\
      -\overline{\ell(z)} & \Id
    \end{pmatrix}.
  \end{displaymath}
  From this formula it follows easily that
  \begin{displaymath}
    (A(z)\overline{A(z)^{-1}})_{b,a}=
    \begin{cases}
      0,&\text{if  }a>b,\\
      (-1)^{a},& \text{if }a=b,\\
      (-1)^{a}\frac{(\log z\bar
      z)^{b-a}}{(b-a)!},&  \text{if }b>a>0,\\
      \displaystyle
      -\Li_{b}(z)+\sum_{k=1}^{b}(-1)^{k}\frac{(\log z\bar
      z)^{b-k}}{(b-k)!}\overline {\Li_{k}(z)},&  \text{if }b>a=0.
    \end{cases}
  \end{displaymath}
  Let $\caL_{b}$, $b\ge 1$ be the single valued polylogarithm
  functions introduced in \cite{brown04:_singl}. They are given by
  \begin{displaymath}
    \caL_{b}(z)= \Li_{b}(z)-\sum_{k=0}^{b-1}(-1)^{b-k}\frac{(\log z\bar
      z)^{k}}{k!}\overline {\Li_{b-k}(z)}.
  \end{displaymath}
Thus $(A(z)\overline{A(z)^{-1}})_{b,0}=-\caL_{b}(z)$.
Then, for $b>a=0$,
  \begin{equation}
    \label{eq:27}
    \Ht_{1}(H(z)_{0,-b})=\im\left( \frac{-1}{(2\pi i)^{b}}\caL_{b}(z)\right),
  \end{equation}
  while for $b> a>0$
  \begin{equation}
    \label{eq:28}
    \Ht_{1}(H(z)_{-a,-b})=\im\left( \frac {1}{(b-a)!}\left( \frac{\log
        z\bar z}{2\pi i}\right)^{b-a}\right).
  \end{equation}
  Note that this last height is only non zero if $b-a$ is odd.
\end{ex}

  \begin{rmk}
    In terms of the de Rham height of Remark \ref{rem:14} the formulas
    are simpler.
    \begin{displaymath}
      \Ht_{1,\dR}(H(z)_{0,-b})=
      \begin{cases}
        -\re(\caL_{b}(z)),&\text{ if $b$ is odd},\\
        -\im(\caL_{b}(z)),&\text{ if $b$ is even}.
      \end{cases}
    \end{displaymath}
    While, for $b-a$ odd we have
    \begin{displaymath}
      \Ht_{1,\dR}(H(z)_{-a,-b})= (-1)^{a}\frac{\log(z\bar z)^{b-a}}{(b-a)!}.
    \end{displaymath}
 \end{rmk}
  
\begin{ex}\label{exm:14}
  In this example we continue studying the polylogarithm variation of
  mixed Hodge structures of Example \ref{exm:2}, but this time we
  compute the height $\Ht_{2}$. We have
  \begin{displaymath}
    \Ht_{2}(H_{-a,-b})=\langle e_{H^{\vee}_{-a,-b}},\delta
    e_{H_{-a,-b}}\rangle = \frac{1}{(2\pi i)^{b-a}} (\delta )_{b,a}.
  \end{displaymath}
  Therefore we have to compute the entries of the matrix $\delta
  $. Write
  \begin{displaymath}
    B(z)=A(z)\overline {A(z)^{-1}} \tau (-1).
  \end{displaymath}
  By the explicit description of the entries of $A(z)\overline
  {A(z)^{-1}}$ it is clear that $B(z)$ is a lower triangular matrix
  with $1$ in the diagonal. Therefore
  \begin{displaymath}
    \log B(z)=\sum _{k=1}^{\infty}(-1)^{k+1}\frac{(B(z)-\Id)^{k}}{k}
  \end{displaymath}
  is a well defined matrix. We next show that the matrix $i/2 \log
  B(z)$ satisfies the conditions that characterize $\delta $. Write
  momentarily $\delta _{0}=i/2 \log B(z)$. Since $\delta _{0}$ is
  strictly lower triangular, by the shape of the weight and Hodge
  filtrations of
  $H_{\dR}$ we obtain that $\delta _{0}\in \Lambda ^{1,1}$.
  Moreover, since $\tau (-1)$ commutes with $Y$, then 
  \begin{displaymath}
    \exp(-2i\delta _{0})\cdot Y=B(z)Y B(z)^{-1}= A(z)\overline {A(z)^{-1}}
    Y \overline {A(z)}A(z)^{-1}=c(Y).
  \end{displaymath}
  It remains to be seen that $\delta _{0}$ is a real operator with
  respect to the Betti real structure, that is
  $c(\delta _{0})=\delta _{0}$.
  So we have to prove that
  \begin{equation}\label{eq:29}
    A(z)\overline {A(z)^{-1}}
    \overline \delta _{0} \overline {A(z)}A(z)^{-1}=\delta _{0}.
  \end{equation}
  Unraveling the definitions of $\delta _{0}$ and $B$, equation
  \eqref{eq:29} is equivalent to the  identity
  \begin{displaymath}
    -\log(\tau (-1)\overline{A(z)}A(z)^{-1})=
    \log(A(z)\overline{A(z)^{-1}}\tau (-1))
  \end{displaymath}
  that follows from the equality $(A(z)\overline{A(z)^{-1}\tau
    (-1)})^{-1}=\tau (-1)\overline{A(z)}A(z)^{-1}$.

  The entries of the matrix $-1/2 \log B(z)$ are computed in
  \cite[1.5]{beilinsondeligne:_inter_zagier}. For $b>0$ consider the
  higher analogues of the Bloch--Wigner
  dilogarithm appearing in \cite{beilinsondeligne:_inter_zagier} (note
  that there are small differences in the definition of $D_{b}$ in the
  literature).  
  \begin{displaymath}
    D_b(z)=
    \begin{cases}
            \sum_{k=0}^{b-1} b_k\frac{(\log z\bar
        z)^{k}}{k!}\re(\Li_{b-k}(z))&\text{ if }b \text{ odd},\\
      i\sum_{k=0}^{b-1} b_{k}\frac{(\log z\bar
      z)^{k}}{k!}\im(\Li_{b-k}(z))&\text{ if }b \text{ even},
    \end{cases}
  \end{displaymath}
  where the $b_{k}$ are the Bernoulli numbers $1,-1/2,1/6,0,-1/30,0,\dots$
    Then, for $b>a=0$,
  \begin{equation}\label{eq:45}
    \Ht_{2}(H(z)_{0,-b})=\frac{1}{i}\frac{1}{(2\pi i)^{b}}D_{b}(z)
  \end{equation}
  while for $b-1=a>0$
  \begin{equation}\label{eq:44}
    \Ht_{2}(H(z)_{-a,-a-1})=\frac{-\log
      z\bar z}{4\pi}
  \end{equation}
  When $a\not = 0$ and $b\not = a+1$ we have $\Ht_{2}(H(z)_{-a,-b})=0$.
\end{ex}

\begin{rmk}\label{rem:15}
  The de Rham heights can also be written in terms of higher logarithms
  and we obtain
  \begin{displaymath}
    \Ht_{2,\dR}(H(z)_{0,-b})=D_{b}(z)
  \end{displaymath}
  and
  \begin{displaymath}
    \Ht_{2,\dR}(H(z)_{-a,-a-1})=(-1)^{a}\log|t|.
  \end{displaymath}
\end{rmk}

\begin{rmk}\label{cont-ht}
Examples \ref{exm:2} and \ref{exm:14} completely generalize \cite[Example 6.7]{BGGP:Height}, and define a Hodge--Tate variation of mixed Hodge structures over
$\P^{1}_{\C}\setminus \{0,1,\infty\}$. We observe that while $\Ht_{2}(H(z)_{-a,-b})$ extends continuously to $z=0$
except when $b=a+1$, $\Ht_{1}(H(z)_{-a,-b})$ does not in case
$b>a>0$.
\end{rmk}

We next discuss the functoriality properties of the height. In fact, 
both heights are functorial with respect to framed morphisms of mixed
Hodge structures. 
\begin{prop}\label{height-ind-morphism}
Let $f\colon H\rightarrow H'$ be a framed morphism of $(a,b)$-framed
mixed Hodge structures. Then for $i=1,2$,
$\Ht_{i}(H)=\Ht_{i}(H')$. 
\end{prop}
\begin{proof}
We show the result for $\Ht_{2}$, since it involves less notation. Let
$H=(F,W, \phi,\psi)$ and $H'=(F',W', \phi',
\psi')$, such that $\phi'=\Gr_{2a}(f)\circ \phi$ and
$\psi=\psi'\circ \Gr_{2b}(f)$. Since the Deligne splitting is functorial
with respect to morphisms of mixed Hodge structures we have $f\circ
\delta_{H_{a,b}}=\delta_{H'_{a,b}}\circ f$. Moreover $f(I^{p,q}_{H})\subset
I^{p,q}_{H'}$ and $f^{\vee}(I^{p,q}_{H'{}^{\vee}})\subset
I^{p,q}_{H}$. This implies that $e_{H'}=f(e_{H})$ and
$e_{H^{\vee}}=f^{\vee}(e_{H'{}^{\vee}})$. 

From the definition we get
\begin{multline*}
\Ht_{2}(H')=\langle  e_{H'{}^{\vee}},\delta _{H'}(e_{H'})\rangle
=
\langle  e_{H'{}^{\vee}},\delta _{H'}(f(e_{H}))\rangle =\\
\langle  e_{H'{}^{\vee}}, f(\delta _{H}(e_{H}))\rangle =
\langle  f^{\vee}(e_{H'{}^{\vee}}), \delta _{H}(e_{H})\rangle =
\langle  e_{H^{\vee}},\delta _{H}(e_{H})\rangle =
\Ht_{2}(H),
\end{multline*}
\end{proof}
\begin{rmk}\label{rem:correspondence}
  More generally, given two framed mixed Hodge
structures $(H, \phi,\psi)$ and $H'=(H',\phi', \psi')$ in
$\caH^{a}_{b}$ and a morphism $f\colon H\to H'$ of 
mixed Hodge structures, that is not a framed one but satisfies   
  $\phi'=m_{1}(\Gr_{2a}(f)\circ \phi)$ and $\psi=m_{2}(\psi'\circ
  \Gr_{2b}(f))$, for $m_{1},m_{2}\in \Z$, then
  $m_{1}\Ht_{i}(H)=m_{2}\Ht_{i}(H')$. 
\end{rmk}

The following result shows that the heights change sign with respect to duality.
\begin{prop}\label{prop-dual-ht-2}
Let $H^{\vee}$ be the dual of a framed mixed Hodge structure $H$. Then
for $i=1,2$, $\Ht_{i}(H^{\vee})=-\Ht_{i}(H)$.
\end{prop}
\begin{proof}
Since duality is compatible with complex conjugation and the Deligne
splitting, we have 
\begin{displaymath}
  \langle e_{H},\overline {e_{H^{\vee}}} \rangle=
  \langle \overline {e_{H^{\vee}}},e_{H} \rangle=
  \overline {\langle e_{H^{\vee}},\overline {e_{H}} \rangle}.
\end{displaymath}
Thus
\begin{align*}
  \Ht_{1}(H^{\vee})&=\im\left(\langle e_{H}, \overline{e_{H^{\vee}}}\rangle\right)\\
&=\im \left(\overline {\langle e_{H^{\vee}},\overline {e_{H}} \rangle}\right)\\
&=-\im \left(\langle e_{H^{\vee}},\overline {e_{H}} \rangle\right)\\
& =-\Ht_{1}(H).
\end{align*}
Clearly,  $Y_{H^{\vee}}=-Y^{T}_{H}$. We also have
$\delta_{H^{\vee}}=-\delta^{T}_{H}$. Indeed
\begin{displaymath}
  e^{2i\delta_{H}^{T}}\cdot Y_{H^{\vee}}=
  -e^{2i\delta_{H}^{T}}Y_{H}^{T} e^{-2i\delta_{H}^{T}}=\\
  -(e^{-2i\delta_{H}}Y_{H} e^{2i\delta_{H}})^{T}=
  -\overline Y_{H}^{T}=\overline Y_{H^{\vee}},
\end{displaymath}
so $-\delta^{T}_{H}$ satisfies the conditions characterizing $\delta
_{H^{\vee}}$. 

From this we compute
\begin{align*}
\Ht_{2}(H^{\vee})&=\langle e_{H}, \delta_{H^{\vee}}(e_{H^{\vee}})\rangle\\
&=\langle e_{H}, -\delta^{T}_{H}(e_{H^{\vee}})\rangle\\
&=-\langle \delta_{H}(e_{H}), e_{H^{\vee}}\rangle\\
&=-\langle e_{H^{\vee}}, \delta_{H}(e_{H})\rangle\\
&=-\Ht_{2}(H).
\end{align*}
\end{proof}
On the other hand, if $H$ is a framed mixed Hodge structure with
$(a,b)$-framing given by $(e_{H}, e_{H^{\vee}})$, then $\overline{H}$
is also a $(a,b)$-framed mixed Hodge structure, with framing given by
$((-1)^{a}\overline{e_{H}},
(-1)^{-b}\overline{e_{H^{\vee}}})$.
To explain the sign, note that the composition 
\begin{displaymath}
  \Q(a) \longrightarrow \Gr_{2a}^WH \longrightarrow H 
\end{displaymath}
induces a diagram
\begin{displaymath}
  \Q(a) \longrightarrow \overline {\Q(a)} \longrightarrow
  \overline{\Gr_{2a}^WH}
  \longrightarrow \overline{H},
\end{displaymath}
where the first arrow is the isomorphism of Example \ref{exm:8}. By
the last composition $\bfone(a)_{\Betti}$ is sent to $(-1)^a\overline {e_H}$.

\begin{rmk}\label{rem:12} Here we see yet another use of the symbol
  $x\mapsto \overline x$ as the map from $H$ to $\overline
  H$. Nevertheless if we identify $H_{\Betti}$ with $(\overline
  H)_{\Betti}$ as they are the same $\Q$-vector space, then this map
  agrees with complex conjugation with respect to the real
  structure  given by $H_{\Betti}$.
\end{rmk}

\begin{prop}\label{conj-ht}
The heights of $H$ and $\overline{H}$ satisfy the relation
\begin{displaymath}
\Ht_{i}(\overline{H})=(-1)^{a-b+1}\Ht_{i}(H)~~~\text{for}~i=1,2.
\end{displaymath}
\end{prop}
\begin{proof}
Using the fact that complex conjugation sends $\im$ to $-\im$ we get.
 for $\Ht_{1}$
\begin{align*}
  \Ht_{1}(\overline{H})&=\im\left(\langle (-1)^{-b}\overline {e_{H^{\vee}}}, \overline{(-1)^{a}\overline{e_{H}}}\rangle\right)\\
&=(-1)^{a-b}\im \left(\overline {\langle e_{H^{\vee}},\overline {e_{H}} \rangle}\right)\\
&=(-1)^{a-b+1}\Ht_{1}(H).
\end{align*}
Note that in this computation we are using Remark \ref{rem:12} to
identify all the complex conjugations as the complex conjugation with
respect to the Betti real structure.

Next, using the fact that $\delta_{\overline{H}}=-\delta_{H}$ (Example
\ref{exm:7}) and that $\delta_{H}$ is a real operator, we get
\begin{align*}
\Ht_{2}(\overline{H})&=\langle (-1)^{-b}\overline {e_{H^{\vee}}},
                       \delta_{\overline H}((-1)^{a}\overline {e_{H}})\rangle\\
&=(-1)^{a-b+1}\overline {\langle e_{H^{\vee}}, \delta_{H}(e_{H})\rangle}\\
&=(-1)^{a-b+1}\Ht_{2}(H).
\end{align*}
In the last equality we have used Lemma \ref{lem-2-17}.
\end{proof}

Finally we observe that the twist does not change the height.

\begin{lem}\label{lem-p}
Let $(H,e_{H}, e_{H^{\vee}})$ be a framed mixed Hodge
  structure and $p\in \Z$. Then, for $i=1,2$,
  \begin{displaymath}
    \Ht_{i}(H)= \Ht_{i}(H(p))
  \end{displaymath}
\end{lem}
\begin{proof}
 As mentioned in Remark \ref{rem:11}, it is easy to check that
 $e_{H(p)}=e_{H}\otimes \bfone(p)_{\Betti}$, while 
 $e_{H(p)^{\vee}}=e_{H^{\vee}}\otimes \bfone(-p)_{\Betti}$. The statement
 now follows from the fact that $\bfone(p)_{\Betti}$ and $\bfone(-p)_{\Betti}$ are invariant
 with respect to the Betti complex conjugation and satisfy $\langle
 \bfone(-p)_{\Betti},\bfone(p)_{\Betti}\rangle =1$.  
\end{proof}

\section{Cohomological results}
\label{sec:cohom-results}
From now, in order to save some space in big diagrams, we will omit
the coefficients in cohomology and write for instance, for a pair of spaces
$Y\subset X$, 
\begin{displaymath}
  \rH^{\ast}(X,Y;a)=\rH^{\ast}(X,Y;\Q(a)).
\end{displaymath}
\subsection{Local product situation and duality}
\label{sec:local-prod-situ}
Given a smooth projective complex variety $X$, and two closed subsets $A$ and
$B$, with $Z=A\cap B$, the
cohomology groups of the pairs $(X\smallsetminus A, B\smallsetminus Z)$ and
$(X\smallsetminus B, A\smallsetminus Z)$ will 
be the main focus of our study. It is not always the case that these
groups  
are in duality. However, when $A$ and $B$ are in \textit{local
  product situation}, the relative cohomology groups are indeed dual
to each other. In this section, we recall the definition of local
product situation, and gather some results related to relative cohomology
groups of the above type.

Before proceeding further we introduce a shorthand that will be useful
later when we consider complicated closed subsets. 
Note that $Z$ does not add  information as it can be deduced from $A$
and $B$ so it is not needed to include it in the notation. Thus we
will write
\begin{displaymath}
  \rH^{n}(X\smallsetminus A,B) \coloneqq
  \rH^{n}(X\smallsetminus A,B\smallsetminus (A\cap B)).
\end{displaymath}

\begin{df}\label{def:sec-1-product}
  Let $A$ and $B$ be closed subvarieties of $X$. We say that $A$ and
  $B$ are in \emph{local product situation} if, for any point $x\in X$,
  there is a
  neighborhood $U$ of $x$, a decomposition $U=U_{A}\times U_{B}$,
  where $U_{A}$ and $U_{B}$ are open disks of smaller dimension, and
  analytic subvarieties $A'\subset U_{A}$ and $B'\subset U_{B}$ such
  that
  \begin{displaymath}
    A\cap U = A'\times U_{B},\qquad B\cap U = U_{A}\times B'.
  \end{displaymath}
\end{df}

The following result is proved in \cite[Lemma 6.1.1]{BKV:Fihnf}.

\begin{lem}\label{lemm:1} Let $X$ be a smooth irreducible projective variety.
  Let $A$ and $B$ be closed
  subvarieties of $X$ in local product situation. Then, for every
  $a,r\in \Z$, there is an
  isomorphism of mixed Hodge structures
  \begin{displaymath}
    \rH^{r}(X\smallsetminus A,B;a)\xrightarrow{\cong}
    \rH^{2d-r}(X\smallsetminus B,A;d-a)^{\vee}.
  \end{displaymath}
\end{lem}

\begin{ex} We show now a simple example of subvarieties that are not
  in local product situation.  Let $X=\P^{2}$ and let $\ell_{1}$,
  $\ell_{2}$ and $\ell_{3}$ three different lines through the same
  point $p$. Let $A=\ell_{1}\cup \ell_{2}$ and $B=\ell_{3}$. Then
  \begin{displaymath}
    \rH^{r}(X\smallsetminus A,B)=
    \begin{cases}
      \Q(-1),&\text{ if }r=1,\\
      0,&\text{ if }r\not =1.
    \end{cases}
    \quad
    \rH^{r}(X\smallsetminus B,A)=
    \begin{cases}
      \Q(0),&\text{ if }r=1,\\
      0,&\text{ if }r\not =1.
    \end{cases}
  \end{displaymath}
  Thus $A$ and $B$ are not in local product situation and do not
  satisfy the conclusion of Lemma \ref{lemm:1}.
\end{ex}

For us, the main example of local product
  situation  is as follows. 
\begin{ex}\label{exm:1}
  Let $A$ and $B$ be divisors
  without common components of a complex manifold $X$ such that $A\cup B$ is a normal
  crossing divisor. Then $A$ and $B$ are in local product situation. 
\end{ex}

Since this example is the main situation we will consider, from now on
$X$ will denote a smooth irreducible projective manifold and $A$, $B$ will denote
simple normal crossing divisors without common components, such that $D=A\cup B$ is a normal
crossing divisor. In this case we will use the following terminology. 

\begin{df}\label{def:11}
  A coordinate neighborhood $V$ with coordinates $(z_1,\dots , z_d)$ is called adapted to
$D$ if there are two disjoint subsets $I_A$ and $I_B$ of $[1,d]$ such that
\begin{displaymath}
  V\cap A=\{z\in V\mid \prod_{i\in I_A} z_i =0\}\quad
  V\cap B=\{z\in V\mid \prod_{i\in I_B} z_i =0\}
\end{displaymath}
and, for all $z\in V$ and  $i\in I_A\cup I_B$,  $|z_i|\le 1/2$.
\end{df}

\subsection{Relative de Rham cohomology}
\label{sec:relative-de-rham}
We recall how to compute the de Rham component of the mixed
Hodge structure  
$\rH^{n}(X\smallsetminus A,B)$ with its Hodge filtration.

\begin{df}\label{def:7} Let $\widetilde B$ be the normalization of $B$
  that is a smooth variety, and let $\iota \colon \widetilde B\to X$
  be the induced map. The \emph{sheaf of holomorphic differential
    forms with logarithmic poles along $A$ and vanishing along $B$} is
  the sheaf
  \begin{displaymath}
    \Sigma _{B}\Omega ^{\ast}(\log A)=
    \{\omega \in \Omega ^{\ast}(\log A)\mid \iota ^{\ast}\omega =0\}.
  \end{displaymath}
\end{df}

\begin{lem}\label{lemm:2} There is a canonical isomorphism
  \begin{displaymath}
    \Sigma _{B}\Omega ^{\ast}(\log A) \cong
    \Omega ^{\ast}(\log A\cup B)\otimes \caO(-B)
    .
  \end{displaymath}
\end{lem}
\begin{proof}
  A differential form $\omega $ belongs to $\Omega ^{\ast}(\log A\cup
  B)\otimes \caO(-B)$ if and only if it can be written in each
  neighborhood adapted to $D$ as
  \begin{equation}\label{eq:26}
    \omega =\sum _{I,J,K} f_{I,J,K}\prod _{\ell\in I_B}z_\ell
    \bigwedge_{i\in I}\frac{dz_i}{z_i}
    \bigwedge _{j\in J}\frac{dz_j}{z_j}
    \bigwedge_{k\in K}dz_k
  \end{equation}
where the sum is over all subsets $I\subset I_A$, $J\subset I_{B}$ and
$K\subset \{1,\dots,d\}\smallsetminus I_A\cup I_B$ and the functions
$f_{I,J,K}$ are holomorphic. But a differential
form can be written as in \eqref{eq:26} if and only if it can be
written as
  \begin{equation}\label{eq:46}
    \omega =\sum _{I,J,K} f_{I,J,K}\prod _{\ell\in I_B\smallsetminus J}z_\ell
    \bigwedge_{i\in I}\frac{dz_i}{z_i}
    \bigwedge _{j\in J}dz_j
    \bigwedge_{k\in K}dz_k,
  \end{equation}
which is exactly the condition for belonging to $\Sigma _{B}\Omega
^{\ast}(\log A)$. 
\end{proof}

\begin{prop}
  There is a natural isomorphism between $\rH^{n}(X\smallsetminus
  A,B;\C)$ and the hypercohomology of the complex $\Omega ^{\ast}_{X}(\log
  A\cup B)\otimes \caO(-B)$. Moreover the Hodge filtration of
  $\rH^{n}(X\smallsetminus
  A,B;\C)$ is induced by the b\^ete filtration
  \begin{displaymath}
    F^{p}\Omega ^{\ast}_{X}(\log A\cup B)\otimes \caO(-B) =
    \bigoplus _{p'\ge p} \Omega ^{p'}_{X}(\log A\cup B)\otimes \caO(-B).
  \end{displaymath}
\end{prop}
\begin{proof}
  In view of Lemma \ref{lemm:2} it is enough to prove the
  corresponding result for the complex of sheaves $\Sigma _B\Omega
  ^{\ast}(\log A)$. To compute the relative cohomology  $$\rH^{n}(X\smallsetminus
  A,B;\C)$$ we need a complex that computes the cohomology of
  $X\setminus A$, a complex that computes the cohomology of
  $B\setminus (B\cap A)$ and a map between the two complexes that
  represents the pullback in cohomology. Then the relative cohomology
  is the cohomology of the cone of that map shifted by one. The
  cohomology of $X\setminus A$ is the hypercohomology of the complex
  $\Omega ^{\ast}_{X}(\log A)$. We next explain how to compute the
  cohomology of $B\setminus (A\cap 
  B)$.
  Assume for simplicity that $B$ is a simple
  normal crossings divisor and let $B=B_{1}\cup\dots\cup B_{k}$ be the
  decomposition into smooth irreducible subvarieties. For each
  $I\subset [1,k]$, put $B_{I}=\bigcap_{i\in I}B_{i}$ and
  $\iota_{I}\colon B_{I}\to X$ for the inclusion. Then the
  hypercohomology of the total complex of the double complex 
  \begin{multline*}
    \bigoplus _{i=1}^{k}
    (\iota_{i})_{\ast}\Omega^{\ast}_{B_{i}}(\log B_{i}\cap A)
    \longrightarrow
\bigoplus _{|I|=2}
    (\iota_{I})_{\ast}\Omega^{\ast}_{B_{I}}(\log B_{I}\cap A)
    \longrightarrow
    \dots \\
    \longrightarrow
    \bigoplus _{|I|=d}
    (\iota_{I})_{\ast}\Omega^{\ast}_{B_{I}}(\log B_{I}\cap A)
  \end{multline*}
  is the cohomology of $B\setminus (B\cap A)$. Moreover, the Hodge
  filtration of this cohomology is defined by the b\^ete filtration
  \begin{displaymath}
    F^{p} \bigoplus _{|I|=r}
    (\iota_{I})_{\ast}\Omega^{\ast}_{B_{I}}(\log B_{I}\cap A)=
    \bigoplus _{|I|=r}
    (\iota_{I})_{\ast}\bigoplus _{p'\ge p}\Omega^{p'}_{B_{I}}(\log
    B_{I}\cap A). 
  \end{displaymath}
  The map from $\Omega ^{\ast}_{X}(\log A)$ to the above total complex
  is induced by the pullback to each component $B_i$. The lemma
  follows from the fact that, for every $p$ the sequence of sheaves on $X$
  \begin{multline*}
    0\longrightarrow
    F^{p}\Sigma _B\Omega^{\ast}_{X}(\log A)
    \longrightarrow
    F^{p}\Omega^{\ast}_{X}(\log A)
    \longrightarrow\\
    F^{p}\bigoplus _{i=1}^{k}
    (\iota_{i})_{\ast}\Omega^{\ast}_{B_{i}}(\log B_{i}\cap A)
    \longrightarrow  \dots
    \longrightarrow
    F^{p}\bigoplus _{|I|=r}
    (\iota_{I})_{\ast}\Omega^{\ast}_{B_{I}}(\log B_{I}\cap A)\longrightarrow
    \\\dots \longrightarrow
    F^{p}\bigoplus _{|I|=d}
    (\iota_{I})_{\ast}\Omega^{\ast}_{B_{I}}(\log B_{I}\cap A)
    \longrightarrow 0
  \end{multline*}
  is exact.
\end{proof}

\begin{rmk}  The duality of Lemma \ref{lemm:1} has a nice
  interpretation in terms of differential forms. 
  There is a morphism of complexes of sheaves
  \begin{multline}
    \label{eq:22}
    (\Omega ^{\ast}_{X}(\log A\cup B)\otimes \caO(-B))
    \otimes
    (\Omega ^{\ast}_{X}(\log A\cup B)\otimes \caO(-A))\\
    \longrightarrow \Omega ^{\ast}_{X}(\log A\cup B)\otimes \caO(-A-B) 
  \end{multline}
  that induce maps
  \begin{displaymath}
    \rH^{k}(X\smallsetminus A,B)\otimes
    \rH^{\ell}(X\smallsetminus B,A)
    \longrightarrow
    \rH^{k+\ell}(X,A\cup B).
  \end{displaymath}
  This product, together with the canonical isomorphism
  $\rH^{2d}(X,A\cup B;d)\simeq \rH^{2d}(X;d) \simeq \Q(0)$ induces the duality 
  of Lemma \ref{lemm:1}.
\end{rmk}

\subsection{Relative cohomology and blow-up}
\label{sec:relat-cohom-blow}
The mixed Hodge structures on the cohomology of a pair of varieties
satisfy many functorial properties, like inverse images by morphisms
of pairs of spaces or direct images for proper morphisms. In many
cases, it is interesting, not only to know that a functorial property
is satisfied, but to have a concrete morphisms of complexes that
represents said functoriality. For instance, to represent the direct
image in cohomology for a proper morphism we need to go from
differential forms to  currents, then use the direct image on currents
and, if needed, go back to differential forms. In this section we will
review some functorial properties that can be easily represented using 
the complexes of  differential forms  with logarithmic singularities
along a divisor, that vanish along a second divisor. So for our
purposes, Remark \ref{rem:6} is more interesting than 
Proposition \ref{prop:10}, that is only a particular case of a well
known result.  

We first state the obvious functoriality.

\begin{prop}\label{prop:10}
  Let $f\colon X'\to X$ be a morphism of projective complex manifolds
  and $D'=A'\cup B'$ a normal crossing divisor with $A'$ and $B'$
  normal crossing divisors without common components such that
  $f^{-1}(A)\subset A'$ and $f(B')\subset B$. Then there is a morphism of mixed
  Hodge structures
  \begin{displaymath}
    f^{\ast}\colon \rH^{\ast}(X\smallsetminus A,B )\longrightarrow
    \rH^{\ast}(X'\smallsetminus A',B').
  \end{displaymath}
\end{prop}
\begin{proof}
  The conditions $f^{-1}(A)\subset A'$ and $f(B')\subset B$ imply that
  there is a  morphisms of pairs $(X'\smallsetminus A',B'\smallsetminus(A'\cap
  B'))\to
  (X\smallsetminus A,B\smallsetminus(A\cap
  B))$ that gives the claimed functoriality.
\end{proof}
\begin{rmk}\label{rem:6}
  With the hypothesis of Proposition \ref{prop:10}, since
  $f^{-1}(A)\subset A'$, there is a pullback
  morphism $f^{\ast}\colon \Omega ^{\ast}_{X}(\log A)\to \Omega
  ^{\ast}_{X'}(\log A')$. Since $f(B')\subset B$ this morphism sends
  $\Sigma _B\Omega ^{\ast}_{X}(\log A)$ to $\Sigma _{B'}\Omega
  ^{\ast}_{X}(\log A')$. This map induces 
  the de Rham component of the map of Proposition \ref{prop:10}.
\end{rmk}

The second functoriality result is the following. 
\begin{prop}\label{prop:6}  Let $W\subset B$ be a smooth
  irreducible variety that has only normal crossing with $D$ (see
  Definition \ref{def:6}). Let
  $\pi\colon \widetilde X\rightarrow X$ be the blow up of $X$ along
  $W$, and $E$ the exceptional divisor of this blow up. By the
  condition on $W$,  $\pi^{-1}(A\cup B)$ is still a normal 
  crossing divisor. Let
  $\widehat A$ and $\widehat B$ be the strict transforms of $A$ and
  $B$, then there is a natural isomorphism of mixed Hodge structures
  \begin{equation}\label{eq:10}
    \varphi \colon \rH^{\ast}(X\smallsetminus A,B)
    \xrightarrow{\cong}
    \rH^{\ast}(\widetilde X\smallsetminus \widehat A,\widehat
    B\cup E).
  \end{equation}
\end{prop}
\begin{proof}
  If $W\subset A\cap B$ then there is a morphism of mixed Hodge structures, 
  \begin{displaymath}
    \psi\colon \rH^{\ast}(\widetilde X\smallsetminus \widehat A,\widehat
    B\cup E) \longrightarrow
    \rH^{\ast}(\widetilde X\smallsetminus \widehat A\cup E,\widehat
    B) =
    \rH^{\ast}(X\smallsetminus A,B),
  \end{displaymath}
  While, if $W\not \subset A$, then the map of pairs of spaces
  $(\widetilde X\smallsetminus \widehat A,(\widehat B\cup E)\smallsetminus
  \widehat A)\to (X\smallsetminus A,B\smallsetminus A)$ induces a morphism of
  mixed Hodge structures
  \begin{displaymath}
    \rH^{\ast}(X\smallsetminus A,B)
    \longrightarrow
    \rH^{\ast}(\widetilde X\smallsetminus \widehat A,\widehat B\cup E).
  \end{displaymath}
  We have to show that both morphisms give isomorphisms of vector
  spaces.  
Since $A$ and $B$ do not have any common components, and $A\cup B$ is
a normal crossing divisor, $A$ and $B$ are in local product
situation. Hence, by Lemma \ref{lemm:1} there are duality
isomorphisms 
\begin{align*}
  \rH^{n}(X\smallsetminus A, B)
  &\cong \rH^{2d-n}(X\smallsetminus B, A)^{\vee}\\
  \rH^{n}(\widetilde X\smallsetminus \widehat A, \widehat B\cup E)
  &\cong \rH^{2d-n}(\widetilde X\smallsetminus \widehat B\cup E, \widehat A)^{\vee}. 
\end{align*}
From the fact that $X\smallsetminus B = \widetilde {X}\smallsetminus
\widehat{B}\cup E$, and $A\smallsetminus B=\widehat{A}\smallsetminus
\widehat{B}\cup E$ (using $W\subset B$), we get that the groups on the
right of the above isomorphisms are also isomorphic. Thus the groups
on the left are isomorphic.  
Since this isomorphism of groups is compatible with the previously
defined  morphisms of mixed Hodge
structures, we deduce that  it is an isomorphism of mixed Hodge
structures.
\end{proof}

\begin{rmk}\label{rem:3}
  The de Rham component of the map \eqref{eq:10} is induced by the
  morphism of sheaves 
  \begin{equation}\label{eq:25}
    \pi ^{\ast}\colon
    \Omega ^{\ast}_{X}(\log A\cup B)\otimes \caO(-B)\longrightarrow
    \Omega ^{\ast}_{\widetilde X}(\log \widehat A\cup \widehat B\cup
    E)\otimes \caO(-\widehat B-E).
  \end{equation}
\end{rmk}

Next we want to see that, if we drop the condition $W\subset B$ in
Proposition \ref{prop:6} we still get an isomorphism on the condition
that we restrict to
certain pieces of the Hodge and weight filtration. For this we need
some preliminaries on the cohomology of projective spaces.  

\begin{lem}\label{lemm:4} In $\P^{n}$ with homogeneous coordinates $(z_{1}\colon
  \dots\colon z_{n+1})$, denote by $H_{i}$ the coordinate hyperplane
  $z_{i}=0.$ Then, for $r,s \ge 0$ the mixed Hodge structures of the
  groups $\rH^{\ast}(\P^{n}\smallsetminus H_{1}\cup\dots \cup
  H_{r}, H_{r+1}\cup \dots \cup H_{r+s})$ are Hodge--Tate. Moreover, the
  only non zero groups are 
  \begin{enumerate}
  \item if $r=s=0$, then
    \begin{displaymath}
      \rH^{2k}(\P^{n})= \Q(-k),\ k=0,\dots,n;  
    \end{displaymath}
  \item if $s=0$ and $r>0$, then
    \begin{displaymath}
      \rH^{k}(\P^{n}\smallsetminus H_{1}\cup\dots\cup H_{r})=
      \Q(-k)^{\binom{r-1}{k}},\ k=0,\dots,r-1;
    \end{displaymath}
  \item if $r=0$ and $s>0$, then
    \begin{displaymath}
      \rH^{2n-k}(\P^{n},H_{1}\cup\dots\cup H_{s})=
      \Q(n-k)^{\binom{s-1}{k}},\ k=0,\dots s-1; 
    \end{displaymath}
  \item if $r>0$ and $s>0$, none. That is
    \begin{displaymath}
      \rH^{k}(\P^{n}\smallsetminus H_{1}\cup\dots \cup
  H_{r},H_{r+1}\cup \dots \cup H_{r+s})=0.
    \end{displaymath}
  \end{enumerate}
\end{lem}
\begin{proof}
  The first statement is well known. The second follows from the
  isomorphism
  \begin{displaymath}
    \P^{n}\smallsetminus H_{1}\cup\dots\cup H_{r}=\A^{n-r+1}\times \G_{m}^{r-1}.
  \end{displaymath}
  The third follows from the second by duality. The last statement is
  proved using induction and the long exact sequence
  \begin{multline*}
    \rH^{k}(\P^{n}\smallsetminus H^{(r)},H^{(s)}) \longrightarrow
        \rH^{k}(\P^{n}\smallsetminus H^{(r)},H^{(s-1)})
        \longrightarrow\\
            \rH^{k}(\P^{n-1}\smallsetminus H^{(r)},H^{(s-1)})
            \longrightarrow
            \dots
          \end{multline*}
          where $H^{(r)}=H_{1}\cup \dots \cup H_{r}$ and
          $H^{(s)}=H_{r+1}\cup\dots \cup H_{r+s}$. Indeed, the case
          $r>0$ and $s=1$ is a consequence of the fact that the map
          $\rH^{k}(\P^{n}\smallsetminus H^{(r)})\to
          \rH^{k}(\P^{n-1}\smallsetminus H^{(r)}) $ is an isomorphism and
          the above exact sequence. While the same exact sequence
          shows that once the vanishing is established for $s$, then it
          is also true for $s+1$. 
\end{proof}

\begin{lem} \label{lemm:5}
  Let $\pi
  \colon Y\to X$ be a locally trivial $\P^{n}$-bundle with a
  normal crossing divisor $E=E_{1}\cup E_{2}$ such that there is an open covering
  $\{U_{i}\}$ of $X$  with isomorphisms over $U_{i}$, 
  \begin{displaymath}
    (\pi ^{-1}(U_{i}),E_{1}\cap \pi ^{-1}(U_{i}),E_{2}\cap \pi ^{-1}(U_{i}))\simeq
    U_{i}\times (\P^{n},H^{(r)},H^{(s)}),
  \end{displaymath}
  with $H^{(r)}$ and $H^{(s)}$ as in the proof of Lemma \ref{lemm:4}.
  Put $A_{0}=\pi ^{-1}(A)$, $B_{0}=\pi ^{-1}(B)$.
  \begin{enumerate}
  \item \label{item:1} If $r>0$ and $p\ge d+r$ (resp. $m\ge 2d+2r-1$), then
    \begin{displaymath} 
      F^{p} \rH^{\ast}(Y\smallsetminus A_{0}\cup E_{1},B_{0}\cup E_{2})=0.\quad
      ( \text{resp. }\Gr_{m}^{W} \rH^{\ast}(Y\smallsetminus A_{0}\cup
      E_{1};B_{0}\cup E_{2})=0).
    \end{displaymath}
  \item \label{item:2} If $s>0$  and $p\le n-s$ (resp. $m\le 2n-2s+1$), then 
    \begin{multline*}
      F^{p} \rH^{\ast}(Y\smallsetminus A_{0}\cup E_{1},B_{0}\cup
      E_{2})=\rH^{\ast}(Y\smallsetminus A_{0}\cup E_{1},B_{0}\cup E_{2}).\\ 
      ( \text{resp. }\Gr_{m}^{W} \rH^{\ast}(Y\smallsetminus A_{0}\cup
      E_{1},B_{0}\cup E_{2})=0).
    \end{multline*}
  \end{enumerate}
\end{lem}
\begin{proof}
  We will discuss only the proof of the statement regarding the weight
  filtration as the proof of the statement regarding the Hodge
  filtration is analogous.
  
  Using the fact that the Mayer--Vietoris spectral sequence is a spectral
  sequence of mixed Hodge structures, it is enough to treat the case when
  \begin{displaymath}
    (Y,E_{1},E_{2})=X\times (\P^{n},H^{(r)},H^{(s)}),
  \end{displaymath}
  where $H^{(r)}$ and $H^{(s)}$ are unions of coordinate
  hyperplanes without common components.

  By K\"unneth formula, there is an isomorphism of mixed Hodge
  structures  
  \begin{displaymath}
    \rH^{\ell}(Y\smallsetminus A_{0}\cup E_{1},B_{0}\cup E_{2})=
    \bigoplus _{i+j=\ell} \rH^{i}(X\smallsetminus A,B)\otimes
    \rH^{j}(\P^{n}\smallsetminus H^{(r)},H^{(s)}).
  \end{displaymath}
  Since $X$ has dimension $d$, the mixed Hodge structure
  $\rH^{\ast}(X\smallsetminus A,B)$ has weight between $0$ and
  $2d$. By Lemma \ref{lemm:4}, when $r>0$, the mixed Hodge structure
  $\rH^{\ast}(\P^{n}\smallsetminus H^{(r)}, H^{(s)})$ has weights
  between $0$ and $2r-2$, from which the first statement follows. If
  $s>0$ then $\rH^{\ast}(\P^{n}\smallsetminus H^{(r)}, H^{(s)})$ has
  weights between $2(n-s+1)$ and $2n$, from which the second statement
  follows. 
\end{proof}

\begin{prop} \label{prop:2}
  Let $W\subset X$ be a smooth irreducible subvariety of
  dimension $d_{W}$ such that $W$
  has only normal crossings with $D$ (see
  Definition \ref{def:6}). Let $\pi \colon \widetilde X\to X$ be
  the blow up of  $X$ along $W$, $\widehat A$ and $\widehat B$ the
  strict transforms of $A$ and $B$ and $E$ the exceptional
  divisor. Let $r$ be the number of local analytic components of $A$ containing
  $W$.
  \begin{enumerate}
  \item\label{item:3} If $p>d_{W}+r$ then the map
    \begin{displaymath}
      F^{p}\rH^{\ast}(X\smallsetminus A,B)\longrightarrow
      F^{p}\rH^{\ast}(\widetilde X\smallsetminus \pi ^{-1} A,\widehat
      B) 
    \end{displaymath}
    factors through an isomorphism 
    \begin{displaymath}
      \xymatrix{
       &F^{p}\rH^{\ast}(\widetilde X\smallsetminus \widehat A,
        \widehat B\cup E)\ar[d]\\
        F^{p}\rH^{\ast}(X\smallsetminus A,B)\ar[r]\ar^{\simeq}_{\varphi}[ru]
        &F^{p}\rH^{\ast}(\widetilde X\smallsetminus \pi ^{-1}A,\widehat B)
      }.
    \end{displaymath}
  \item \label{item:4} If $m>2d_{W}+2r$ then the morphism 
    \begin{displaymath}
      \Gr^{W}_{m}\rH^{\ast}(X\smallsetminus A,B)\longrightarrow
      \Gr^{W}_{m}\rH^{\ast}(\widetilde X\smallsetminus \pi ^{-1} A,\widehat B)
    \end{displaymath}
    factors through an isomorphism
    \begin{displaymath}
      \xymatrix{
        &\Gr^{W}_{m}\rH^{\ast}(\widetilde X\smallsetminus \widehat A,
        \widehat B\cup E)\ar[d]\\
        \Gr^{W}_{m}\rH^{\ast}(X\smallsetminus A,B)\ar[r]\ar^{\simeq}[ru]
        &\Gr^{W}_{m}\rH^{\ast}(\widetilde X\smallsetminus \pi ^{-1} A,
        \widehat B)
      }
    \end{displaymath}
  \end{enumerate}
 \end{prop}
 \begin{proof} If $W\subset B$, the result is a particular case of
   Proposition \ref{prop:6}. Thus we can assume that $W\not \subset B$.
   We only prove \ref{item:4} as the proof of \ref{item:3} is
   analogous. If $r=0$, then $W\not \subset A$.  In this case, the formula for
   the cohomology of a blow up 
   implies  that
   \begin{equation}\label{eq:9}
     \rH^{\ast}(X\smallsetminus A,B\cup W) =
     \rH^{\ast}(\widetilde X\smallsetminus \widehat A,\widehat B\cup E).
   \end{equation}
   The long exact sequence
   \begin{displaymath}
     \rH^{\ast}(X\smallsetminus A,B\cup W) \longrightarrow
     \rH^{\ast}(X\smallsetminus A,B)\longrightarrow 
     \rH^{\ast}(W\smallsetminus A,B) \longrightarrow \dots
   \end{displaymath}
   and the fact that the cohomology of $W$ has weights smaller or
   equal than $2d_{W}$
 implies that, for $m>2d_{W}$,
 \begin{displaymath}
           \Gr^{W}_{m}\rH^{\ast}(X\smallsetminus A,B)
           \simeq
       \Gr^{W}_{m}\rH^{\ast}(\widetilde X\smallsetminus \widehat A,
        \widehat B\cup E),
 \end{displaymath}
proving the result in this case. 

   Assume now that $r>0$. It is enough to show that, for $m>2d_{W}+2r$, the maps
   \begin{equation}
     \label{eq:6}
     \Gr_{m}^{W}\rH^{\ast}(\widetilde X\smallsetminus \widehat A,
        \widehat B\cup E)\longrightarrow 
       \Gr_{m}^{W}\rH^{\ast}(\widetilde X\smallsetminus \widehat A,
        \widehat B)
   \end{equation}
   and
   \begin{equation}
     \label{eq:7}
     \Gr_{m}^{W}\rH^{\ast}(\widetilde X\smallsetminus \widehat A,
        \widehat B)\longrightarrow 
       \Gr_{m}^{W}\rH^{\ast}(\widetilde X\smallsetminus \widehat A\cup
      E,\widehat B)
   \end{equation}
   are isomorphisms. Note that, since $\widetilde X\smallsetminus
   \widehat A\cup E=X\smallsetminus A$ and $(\widetilde X\smallsetminus
   \widehat A\cup E)\cap \widehat B=(X\smallsetminus A)\cap B$, in the
   right hand side of the morphism \eqref{eq:7} we have
   \begin{displaymath}
     \Gr_{m}^{W}\rH^{\ast}(\widetilde X\smallsetminus \widehat A\cup
      E,\widehat B)=\Gr_{m}^{W}\rH^{\ast}(X\smallsetminus A,B).
   \end{displaymath}

   There is an exact sequence
\begin{displaymath}
  \cdots\rightarrow\rH^{k}(\widetilde X\smallsetminus \widehat A,
        \widehat B\cup E)\longrightarrow 
       \rH^{k}(\widetilde X\smallsetminus \widehat A,
       \widehat B)
       \longrightarrow
       \rH^{k}(E\smallsetminus \widehat A,
       \widehat B)\longrightarrow \dots
\end{displaymath}
Since $E$ is the exceptional divisor of the blow up along $W$ we know
that $E$ is a locally trivial $\P^{n}$ bundle over $W$ with
$n=d-d_{W}-1$. Since $W$ has only normal crossings with $A\cup B$ we
deduce that the components of $A$ that do not contain $W$ define a
normal crossing divisor $A_{0}$ on $W$, while $B$, that by assumption
does not contain $W$, defines a normal crossing divisor
$B_{0}$. Each component of $A$ that contains $W$ determines a divisor
of $E$ that is horizontal over $W$ and locally can be written as a
coordinate hyperplane of $\P^{n}$. Thus we are
exactly in the situation of Lemma \ref{lemm:5} with $r$ equal to the
number of components of $A$ containing $W$ and $s=0$. If
$m>2d_{W}+2r> 2d_{W}+2r-1$, Lemma \ref{lemm:5} implies that
$\Gr_{m}^{W}\rH^{\ast}(E\smallsetminus \widehat A,\widehat B)=0$, and we obtain
the isomorphism \eqref{eq:6}. To prove the isomorphism \eqref{eq:7},
by duality, it is enough to prove that, for $m<2(n+1-r)=2(d-d_{W}-r)$,
the natural map
 \begin{equation}\label{eq:11}
     \Gr_{m}^{W}\rH^{\ast}(\widetilde X\smallsetminus \widehat B,
        \widehat A\cup E)\longrightarrow 
       \Gr_{m}^{W}\rH^{\ast}(\widetilde X\smallsetminus \widehat
       B,\widehat A) 
     \end{equation}
     is an isomorphism.
     We now consider the long exact sequence
\begin{displaymath}
  \dots\longrightarrow\rH^{k}(\widetilde X\smallsetminus \widehat B,
        \widehat A\cup E)\longrightarrow 
       \rH^{k}(\widetilde X\smallsetminus \widehat B,
       \widehat A)
       \longrightarrow
       \rH^{k}(E\smallsetminus \widehat B,
       \widehat A)\longrightarrow \dots
\end{displaymath}
In order to analyze $\rH^{k}(E\smallsetminus \widehat B, \widehat A)$,
we are again in the situation of Lemma \ref{lemm:5} but with the roles
of $r$ and $s$ interchanged. Therefore, if $m\le 2n-2r+1$, Lemma \ref{lemm:5} yields
\begin{displaymath}
\Gr_{m}^{W}       \rH^{k}(E\smallsetminus \widehat B,
       \widehat A)=0,
\end{displaymath}
hence the sought isomorphism.
 \end{proof}

 \begin{rmk}\label{rem:4} The morphism $\varphi$ in Proposition
   \ref{prop:2}~\ref{item:3} can be interpreted 
   in terms of the complex $\Omega ^{\ast}_{X}(\log A\cup B)\otimes
   \caO(-B)$. If $W\subset B$ it is just the map induced by
   \eqref{eq:25}. If $W\not \subset B$, let $r$ be the number of
   local analytic components of $A$ containing $W$. If $p>d_{W}+r$, we want to show
   that the map 
   \begin{displaymath}
     \pi ^{\ast} \colon \Omega ^{p}_{X}(\log A\cup B)\otimes \caO(-B)
     \longrightarrow
     \Omega ^{p}_{\widetilde X}(\log \widehat A\cup \widehat B\cup E)\otimes
     \caO(-\widehat B) 
   \end{displaymath}
   factors through $\Omega ^{p}_{\widetilde X}(\log \widehat A\cup
   \widehat B\cup E)\otimes
     \caO(-\widehat B-E)$.
   Let $\omega $ be a
   local section of $\Omega ^{p}_{X}(\log A\cup B)\otimes \caO(-B)$
   on an open coordinate set $U$, where $A$ is given by $\prod_{i\in
     I_{A}}z_{i}=0$, $B$ is given by $\prod_{i\in
     I_{B}}z_{i}=0$ and $W$ is given by $z_{i}=0$, for $i\in
   I_{W}$. Moreover $I_{A}\cap I_{W}$ contains exactly $r$ elements
   and $I_{B}\cap I_{W}=\emptyset$. We can write
   \begin{displaymath}
     \omega = \sum \eta_{i}\wedge \omega_{i}, 
   \end{displaymath}
   where $\eta_{i}$ is a wedge product of forms $dz_{k}/z_{k}$ with
   $k\in I_{A}\cap I_{W}$ and $\omega _{i}$ does not contain any
   $dz_{k}$ for $k\in I_{A}\cap I_{W}$.  Since $p>d_{W}+r$ we
 deduce that $\omega _{i}$ contains at least one term of the form $dz_{k_{i}}$ with
 $k_{i}\in I_{W}\setminus I_{A}$. So $\omega _{i}=\omega
 _{k_{i}}'\wedge dz_{ki}$ and 
   \begin{displaymath}
     \omega' _{k_{i}}\in \Omega _{X}^{\ast}(\log A\cup B)\otimes \caO(-B).
   \end{displaymath}
   Since $k_{i}\in I_{W}$ we have that $\pi ^{\ast}dz_{k_{i}} \in \Omega
   _{\widetilde X}(\log E)\otimes \caO(-E)$.
In consequence 
 \begin{displaymath}
   \pi ^{\ast}(\omega )\in
   \Omega _{\widetilde X}^{\ast}(\log \widehat{A}\cup \widehat B\cup
   E)\otimes \caO(-\widehat B-E).
 \end{displaymath}
 \end{rmk}
 \begin{cor}\label{cor:4}
   With the hypothesis of Proposition \ref{prop:2}, for every $p,q$
   with $p>d_W+r$, there is an isomorphism
   \begin{displaymath}
\varphi\colon 
I^{p,q}\rH^{\ast}(X\smallsetminus A,B)
\longrightarrow
     I^{p,q}\rH^{\ast}(\widetilde X\smallsetminus \widehat A,
        \widehat B\cup E).
      \end{displaymath}
 \end{cor}

 \subsection{An exceptional cup-product}
\label{sec:an-exceptional-cup}
The last ingredient we need is an exceptional cup-product. We give
first the topological interpretation:

Let $M$ be a topological space and $B\subset M$ a closed subset such
that there is an open neighborhood $U$ of $B$ such that $B$ is a
deformation retract of $U$.
There is a canonical isomorphism
$\rH^{\ast}(M,B)\simeq \rH^{\ast}(M,U)$. By excision there is also a
canonical isomorphism
$\rH^{\ast}(M,U)\simeq \rH^{\ast}(M\smallsetminus B,U\smallsetminus
B)$. Summing up we have an isomorphism
\begin{displaymath}
  \rH^{\ast}(M,B)\simeq
  \rH^{\ast}(M\smallsetminus B,U\smallsetminus B). 
\end{displaymath}

\begin{df}\label{def:4} Let $M$ be a topological space and
  $Z\subset M$ a closed subset such that there is an open
  neighborhood $U$ of $Z$ and  $Z$ is a deformation retract of
  $U$. Then the exceptional cup-product
  \begin{displaymath}
    \cup_{\exc}\colon \rH^{k}(M,Z)\otimes \rH^{\ell }(M\smallsetminus Z) \longrightarrow
    \rH^{k+\ell}(M,Z).
  \end{displaymath}
  is defined by the commutative diagram
  \begin{equation}\label{eq:4}
    \xymatrix{
      \rH^{k}(M,Z)\otimes \rH^{\ell }(M\smallsetminus Z) \ar[d]_{\simeq}
      \ar[r]^{\cup_{\exc}} &
    \rH^{k+\ell}(M,Z) \ar[d]^{\simeq}\\
    \rH^{k}(M\smallsetminus Z,U\smallsetminus Z)\otimes \rH^{\ell }(M\smallsetminus
    Z)
    \ar[r]_-{\cup_{\rel}} &
    \rH^{k+\ell}(M\smallsetminus Z,U\smallsetminus Z)
    } 
  \end{equation}
  where $\cup_{\rel}$ is the usual cup-product in relative cohomology.
\end{df}

 \begin{prop}\label{prop:7}
  Let $X$ be a complex algebraic variety. Let $A$ and $B$ be simple normal crossing divisors without common components, such that $D=A\cup B$ is also a simple normal crossing divisor. Then the exceptional cup-product of
  Definition \ref{def:4} induces a cup product of
  mixed Hodge structures
  \begin{displaymath}
    \rH^{k}(X\smallsetminus A, B)\otimes
    \rH^{\ell}(X\smallsetminus A\cup B)
    \longrightarrow
        \rH^{k+\ell}(X\smallsetminus A, B).
      \end{displaymath}
\end{prop}
\begin{proof}
For a complex algebraic variety, the K\"unneth isomorphism is known to be an isomorphism of  mixed
Hodge structures (\cite[Theorem 5.44]{MHS}). One can extend this result
to the relative setting showing that, if $U$ and $V$ are complex
varieties and $Z\subset U$ is a Zariski closed subset, then there are
natural isomorphisms of mixed Hodge structures
\begin{displaymath}
  \bigoplus _{\ell+m=n}\rH^{\ell}(U,Z)\otimes \rH^{m}(V)
  \longrightarrow
  \rH^{n}(U\times V,Z\times V).
\end{displaymath}
In $X\times X$ we consider the divisors $A_{1}=X\times A\cup A\times
X$, $B_{1}=B\times X$ and $B_{2}=X\times B$. Let $\Delta\colon X\to
X\times X$ be the diagonal morphism. By abuse of notation we also
denote by $\Delta $ the image of $X$ by $\Delta $. Then there is a
proper transform $Y\to X\times X$ that is obtained by succession of blow-ups with
smooth centers contained in the total transform of
$B_{1}\cup B_{2}$, and with only normal crossings with the total
transform of $A_{1}\cup B_{1}\cup B_{2}$ such that the strict
transforms $\widehat B_{1}$, $\widehat
B_{2}$ and $\widehat \Delta $
of $B_{1}$, $B_{2}$ and $\Delta $ do not meet. Let $E$ be
the exceptional divisor  of $Y$.   We can assume furthermore that
$\widehat \Delta $ is obtained from $\Delta \simeq X$ by a succession of blow-ups
contained in $B$. Let $E_{1}$ be the exceptional divisor of the map
$\widehat \Delta \to X$ and let $\widehat A$ and $\widehat B$ be the
strict transforms of $A$ and $B$ in $\widehat \Delta $. By abuse of
notation we still denote by $\Delta \colon \widehat \Delta \to
Y$. Note that $\Delta ^{-1}(\widehat B_{1}\cup \widehat
B_{2})=\emptyset$, $\Delta ^{-1}(\widehat A_{1})\subset \widehat A$
and $ \widehat B\cup E_{1}\subset \Delta ^{-1}(E)$. Therefore there is
a pullback map
\begin{displaymath}
   \rH^{\ell+m}(Y\smallsetminus \widehat A_{1}\cup \widehat
  B_{2},\widehat B_{1}\cup E)
  \xrightarrow{\Delta ^{\ast}}
  \rH^{\ell+m}(\widehat \Delta \smallsetminus \widehat A,\widehat
  B\cup E_{1}).
\end{displaymath}
Then, using
Proposition \ref{prop:6}, the 
composition 
\begin{multline*}
  \rH^{\ell}(X\smallsetminus A,B)\otimes
  \rH^{m}(X\smallsetminus A\cup B)
  \longrightarrow
  \rH^{\ell+m}(X\times X\smallsetminus A_{1}\cup B_{2},B_{1})
  \overset{\ast}{\simeq}\\
  \rH^{\ell+m}(Y\smallsetminus \widehat A_{1}\cup \widehat
  B_{2},\widehat B_{1}\cup E)
  \xrightarrow{\Delta ^{\ast}}
  \rH^{\ell+m}(\widehat \Delta \smallsetminus \widehat A,\widehat
  B\cup E_{1})
  \overset{\ast}{\simeq}
    \rH^{\ell+m}(X\smallsetminus A,B),
  \end{multline*}
  is a morphism of mixed Hodge structures. In the previous
  composition, the isomorphisms marked with $\ast$ follow from
  Proposition \ref{prop:6}. Since the cup-product in cohomology is
  given  by
  composing the external product and the pullback by the diagonal
  morphisms, it is clear that the previous composition agrees with $\cup_{\exc}$ defined above. 
\end{proof}

\begin{rmk}\label{rem:2} Using Remark \ref{rem:3}, the de Rham counterpart of Proposition
  \ref{prop:7}  is induced by the product
  \begin{displaymath}
    (\Omega ^{\ast}_{X}(\log A\cup B)\otimes  \caO_{X}(-B))
    \otimes \Omega ^{\ast}_{X}(\log A\cup B) \longrightarrow
    \Omega ^{\ast}_{X}(\log A\cup B)\otimes  \caO_{X}(-B).
  \end{displaymath}
\end{rmk}

 \begin{cor}\label{cor:3}
     Let $X$ be a smooth complex proper variety, and $A$, $B$, $C$ be
  simple normal crossing divisors without common
  components such that $A\cup B\cup C$ is still a
  simple normal crossing divisor. Then there is a cup product of
  mixed Hodge structures
  \begin{displaymath}
    \rH^{k}(X\smallsetminus A, B\cup C)\otimes
    \rH^{\ell}(X\smallsetminus A\cup B, C)
    \longrightarrow
        \rH^{k+\ell}(X\smallsetminus A, B\cup C).
      \end{displaymath}
\end{cor}
\begin{proof}
 Since there is a morphism of mixed Hodge structures
 \begin{displaymath}
\rH^{\ell}(X\smallsetminus A\cup B, C)\to \rH^{\ell}(X\smallsetminus A\cup B)\to \rH^{\ell}(X\smallsetminus
A\cup B\cup C),
\end{displaymath}
we can apply Proposition \ref{prop:7} to $A$ and
$B\cup C$.
\end{proof}

\section{Spaces of differential forms}
\label{sec:spac-diff-forms}

\subsection{Slowly increasing and rapidly decreasing differential
  forms}
\label{sec:slowly-incr-rapidly}
In this section we recall the complexes of slowly increasing and
rapidly decreasing differential forms. They can be used to represent
the cohomology of the complement of a normal
crossings divisor $A$ in a projective complex manifold relative to a
second normal crossing divisor 
$B$, assuming that they do not have common components and $A\cup B$ is
still a normal crossings divisor. The advantage of the complex of
slowly increasing/rapidly decreasing forms is that they provide
the cohomology groups together with the Hodge filtration and the real
structure. Moreover these complexes are compatible with the duality of
Lemma \ref{lemm:1}, the pullbacks  of propositions \ref{prop:6} and
\ref{prop:2} and the cup-product of Proposition \ref{prop:7}.

As in the previous section, $X$ will denote a smooth complex
projective variety of dimension $d$ and 
$D=A\cup B$ a normal crossing divisor such that $A$ and $B$ 
are also normal crossing divisors with no common components. Write
$U=X\setminus D$ and $\iota\colon U\to X$ the inclusion. 

We will use multi-index notation. In particular,
for any pair of multi-indexes $\alpha =(\alpha _1,\dots,\alpha
_d)$, $\beta  =(\beta  _1,\dots,\beta 
_d)$, we write
\begin{gather*}
  z^{\alpha }=\prod_{i=1}^{d}z_{i}^{\alpha _{i}},\qquad \bar
  z^{\alpha}=\prod_{i=1}^{d}\bar z_{i}^{\alpha _{i}},\qquad
  z^{\alpha,\beta }=z^{\alpha }\bar z^{\beta }=\prod_{i=1}^{d}z_{i}^{\alpha _{i}}\bar z_{i}^{\beta
    _{i}},\\
  \partial^{\alpha }=\prod_{i=1}^{d}\frac{\partial^{\alpha _{i}}}{\partial z_{i}^{\alpha _{i}}},\qquad \bar
  \partial^{\alpha}=\prod_{i=1}^{d}\frac{\partial^{\alpha
      _{i}}}{\partial \bar z_{i}^{\alpha _{i}}},\qquad
  \partial^{\alpha,\beta }=\partial^{\alpha }\bar \partial^{\beta }=\prod_{i=1}^{d}
  \frac{\partial^{\alpha _{i}}}{\partial z_{i}^{\alpha _{i}}}
  \frac{\partial^{\beta _{i}}}{\partial \bar z_{i}^{\beta _{i}}},
\end{gather*}

We introduce the differential forms
\begin{equation}\label{eq:13}
  \xi_i =
  \begin{cases}
    dz_i/z_i,&\text{ for }i\in I_A\cup I_B,\\
    dz_i,&\text{ for }i\not \in I_A\cup I_B,\\    
  \end{cases}
  \qquad
  \overline{\xi}_i =
  \begin{cases}
    d\bar z_i/\bar z_i,&\text{ for }i\in I_A\cup I_B,\\
    d \bar z_i,&\text{ for }i\not \in I_A\cup I_B.
  \end{cases}
\end{equation}

\begin{df}\label{def:3} The sheaf of functions on $X$ that are
  slowly increasing along $A$ and rapidly decreasing along $B$,
  denoted by $\sE^0_{X}(\si A,\rd B)$ is the
  subsheaf of $\iota_\ast\sE^0_U$ given locally by the functions $f$ that
  satisfy that, in each coordinate neighborhood $V$ adapted to $D$,
  there exists an integer $N_1\in \Z$ and for each integer $N_2\in \Z$
  and each pair of multi-indexes $\alpha $ and $\beta $, there is a
  constant $C_{N_1,N_2,\alpha ,\beta }$ such that the estimate
  \begin{equation}\label{eq:15}
    \left |
      \partial^{\alpha ,\beta }f
    \right|\le C_{N_1,N_2,\alpha ,\beta }
    \frac{\prod_{i\in I_A}|\log|z_i||^{N_1}
      \prod_{j\in I_B}|\log|z_j||^{N_2}}
    {\prod_{i\in I_{A}\cup I_{B}}|z_{i}|^{\alpha _{i}+\beta _{i}}}.
  \end{equation}
  holds. The sheaf of differential forms that are slowly increasing
  along $A$ and rapidly decreasing along $B$, denoted $\sE_{X}^\ast(\si
  A,\rd B)$ is the locally free
  $\sE_{X}^0(\si A,\rd B)$-module generated locally by $\xi_{i}$ and
  $\overline {\xi}_{i}$, $i=1,\dots,d$. The space of global sections
  will be denoted as
  \begin{displaymath}
    E_{X}^{\ast}(\si A,\rd B)=\Gamma (X, \sE_{X}^\ast(\si A,\rd B)).
  \end{displaymath}
\end{df}
\begin{rmk}\label{rmk-4-2}
  A function $f\in E_{X}^0(\si A,\rd B)$ satisfies in particular that, for
  any neighborhood $V$ adapted to $D$, there is a integer $N_{1}$ and
  for every integer $N_{2}$ there is a constant $C_{N_{1},N_{2}}$ such
  that
  \begin{displaymath}
    | f |\le C_{N_1,N_2}
    \prod_{i\in I_A}|\log|z_i||^{N_1}
      \prod_{j\in I_B}|\log|z_j||^{N_2}.    
    \end{displaymath}
    Since $N_{1}$ is fixed this means that when we approach a
    component of $A$, say $z_{i}=0$, $i\in I_{A}$, the
    function $f$ can grow at most like a power of $|\log |z_{i}||$,
    but since $N_{2}$ can be arbitrarily negative, when we approach a
    component of $B$, say $z_{j}$, $j\in I_{B}$, then the function $f$
    should go to zero faster than any quotient $1/|\log |z_{i}||^{k}$,
    $k>0$. This is the reason for the terminology slowly
    increasing/rapidly decreasing. We also point out that a rapidly
    decreasing function   $f\in E_{X}^0(\rd B)$ extends continuously to
    $B$ taking the value $0$ but does not extend smoothly. For
    instance the function $z^{2}/\bar z$ is rapidly decreasing with
    respect to $\{z=0\}$ but it is not smooth.  It is instructive to
    check that the function $z^{2}/\bar z$ satisfies the condition
    \eqref{eq:15}. 
\end{rmk}

Clearly, $\sE_{X}^\ast(\si A,\rd B)$ has a real structure
$\sE_{X,\R}^\ast(\si A,\rd B)$ and a bigrading
\begin{displaymath}
  \sE_{X}^n(\si A,\rd B)= \bigoplus_{p+q=n} \sE_{X}^{p,q}(\si A,\rd B).
\end{displaymath}
The following result is a combination of
\cite{BCLR:_temper_delig_shimur},  propositions 2.16 and 2.17. Can be
proved using the techniques 
of \cite[Proposition 2.17]{BCLR:_temper_delig_shimur}.
\begin{thm}\label{thm:3}
  For every $p\ge 0$ the sequence
  \begin{displaymath}
    0\to \Omega_{X} ^p(\log D)\otimes \caO(-B)
    \to \sE_{X}^{p,0}(\si A,\rd B)
    \overset{\bar \partial}{\longrightarrow}\dots
    \overset{\bar \partial}{\longrightarrow}
    \sE_{X}^{p,d}(\si A,\rd B)
  \end{displaymath}
  is exact. In consequence the cohomology of the complex
  \begin{math}
    E_{X}^{\ast}(\si A,\rd B)
  \end{math}
  agrees with the cohomology groups $\rH^{\ast}(X\setminus A,B)$ with
  the Hodge filtration of their natural Hodge structure. 
\end{thm}
The real variant of the previous theorem is  
\begin{cor}\label{thm:4}
  The complex 
  \begin{displaymath}
    E^{\ast}_{X,\R}(\si A,\rd B)=\Gamma(X,\sE_{X,\R}^{\ast}(\si A,\rd B))
  \end{displaymath}
  computes the cohomology groups $\rH^{\ast}(X\setminus A,B,\R)$.  
\end{cor}

The following lemmas will be useful when working with the estimates
that define slowly increasing and rapidly decreasing differential
forms.

\begin{lem}\label{lemm:3} Let $n>0$ be an integer and let
  $a_1,\dots,a_n> 0$ be positive real numbers. Let $N$ be an
  integer.
  \begin{enumerate}
  \item\label{item:9} If $N\ge 0$ and there is an $\varepsilon >0$ such that $a_i\ge
    \varepsilon $ for all $i$, then there is a constant $C$ depending
    on $n$, $N$ and $\varepsilon $ such that
    \begin{displaymath}
      \left( \sum_{i=1}^{n} a_i\right)^N\le C
      \prod_{i=1}^{n} a_{i}^{N}.
    \end{displaymath}
  \item\label{item:10} If $N\le 0$ then
    \begin{displaymath}
      \left( \sum_{i=1}^{n} a_i\right)^N\le 
      \prod_{i=1}^{n} a_{i}^{N/n}
    \end{displaymath}
  \end{enumerate}
\end{lem}

\begin{lem}\label{lemm:6}
  Let $a$ and $b$ be real numbers satisfying the conditions
  \begin{align}
    a &\ge K_1 >0,\label{eq:32}\\
    |a+b| &\ge K_2 >0,\label{eq:36}\\
    |b|  &\le K_3>0.\label{eq:38}
  \end{align}
  Then there are constants $C_1,C_2>0$ depending only on $K_1$, $K_2$
  and $K_3$ such that
  \begin{displaymath}
    C_1|a| \le |a+b| \le C_2 |a|.
  \end{displaymath}
\end{lem}
\begin{proof}
  If $|a|\ge 2|b|$ then
  \begin{displaymath}
    |a|/2\le |a+b| \le 3|a|/2,
  \end{displaymath}
  while, if $|a|\le 2|b|$ on the one hand we have
  \begin{displaymath}
    |a| \le 2 K_3 = \frac {2K_3}{K_2}K_2\le \frac {2K_3}{K_2}|a+b|, 
  \end{displaymath}
  and on the other hand
  \begin{displaymath}
    |a+b|\le 3|b| \le 3K_3=\frac{3K_3}{K_1} K_1 \le \frac{3K_3}{K_1}|a|. 
  \end{displaymath}
\end{proof}

\subsection{Functoriality properties}
\label{sec:funct-prop}
We next show that the complex of slowly increasing/rapidly decreasing
forms $E_{X}^\ast(\si A,\rd B)$ is compatible
with the different operations of Section  \ref{sec:cohom-results}.
The first one is the obvious pullback of Proposition
\ref{prop:10}. Note that the hypotheses we need here are more restrictive than that
of  Proposition \ref{prop:10}.
\begin{prop}\label{prop:4-4'}
  Let $f\colon X'\to X$ be a morphism of projective complex manifolds
  and $D'=A'\cup B'$ a normal crossing divisor with $A'$ and $B'$
  normal crossing divisors without common components such that
  $f^{-1}(A)\subset A'$,  $f(B')\subset B$ and 
  $f^{-1}(D)\subset D'$. If $\omega\in E^{\ast}_{X}(\si A, \rd B)$,
  then $f^{\ast}\omega\in E^{\ast}_{X'}(\si A', \rd B')$. Further, if
  $\omega$ is closed, then $f^{\ast}\{\omega\}=\{f^{\ast}\omega\}$.
\end{prop}
\begin{proof}
  Let $U$ and $U'$ be coordinate neighborhoods in $X$ and $X'$
  respectively, adapted to $D$ and $D'$, with coordinates $(z_{1},\dots,
  z_d)$ and 
  $(z'_{1},\dots, z'_{d'})$ satisfying $f(U')\subset U$. In these
  coordinates the map $f$ is expressed as holomorphic functions
  $z_i=z_i(z'_{1},\dots ,z'_{d'})$. The condition $f^{-1}(A)\subset
  A'$ is translated to the condition
  \begin{equation}\label{eq:39}
    z_i=\prod_{j\in I_{A'}}(z'_{j})^{n_{i,j}}u_i, \qquad \forall i\in I_{A},
  \end{equation}
  where $n_{i,j}\ge 0$ and $u_i$ is a unit in $U'$. In the presence of
  the previous condition, the new condition $f^{-1}(D)\subset D'$ is
  equivalent to  
  \begin{equation}\label{eq:40}
    z_i=\prod_{j\in I_{A'}\cup I_{B'}}(z'_{j})^{n_{i,j}}u_i, \qquad
    \forall i\in I_{B}, 
  \end{equation}
  where again $n_{i,j}\ge 0$ and $u_i$ is a unit in $U'$. Finally the
  condition $f(B')\subset B$ is equivalent to the condition
  \begin{equation}
    \label{eq:41}
    \forall j\in I_{B'},\ \exists i\in I_{B},\ \text{with}\ n_{i,j}>0.
  \end{equation}
  We write $\xi'_{j}$ and $\overline{\xi'}_{j}$ the analogues of the
  differential forms \eqref{eq:13} but in the $z_j'$ coordinates. The
  condition $f^{-1}(D)\subset D'$ implies that $f^{\ast}\xi_{i}$ is a
  linear combination of the forms $\xi_{j}'$ with holomorphic
  coefficients. Hence, in order to prove  that
  $f^{\ast}E^{\ast}_{X}(\si A,\rd B)\subset E^{\ast}_{X'}(\si A',\rd
  B')$, it is enough to show that $f^{\ast}E^{0}_{X}(\si A,\rd B)\subset E^{0}_{X'}(\si A',\rd
  B')$. Let $\varphi\in E^{0}_{X}(\si A,\rd B)$. Using a partition of
  unity we can assume that $\varphi$ has compact support contained in
  the adapted coordinate neighborhood $U$. We start by showing that
  $f^{\ast}\varphi=\varphi\circ f$ restricted to $U'$ satisfies the
  required growth conditions. By hypothesis, there is an integer
  $N_{1}$ and for all integers $N_{2}$ there is a constant
  $C=C_{N_{1},N_{2}}$ such that 
  \begin{displaymath}
    f^{\ast}\varphi(z'_{1}\dots,z'_{d'})
    = \varphi(z_{1}\dots,z_{d})
    \le C\prod_{i\in I_{A}} (-\log|z_{i}|)^{N_{1}}
    \prod_{i\in I_{B}} (-\log|z_{i}|)^{N_{2}}.
  \end{displaymath}
  We fix for the moment $N_{2}$ and without loss of generality we
  assume that $N_{1}\ge 0$ and that $N_{2}\le 0$. The conditions
  \eqref{eq:39} and \eqref{eq:40} together  with the fact that, since $U$ is adapted to $D$, for $i\in
  I_{A}\cup I_{B}$, we have $|\log|z_i||=-\log|z_i|$, 
  imply that the function $ f^{\ast}\varphi(z'_{1}\dots,z'_{d'})$ is
  bounded by the product  
  \begin{multline*}
    C\prod_{i\in I_{A}} \left(\sum_{j\in
        I_{A'}}-\log|z'_{j}|^{n_{i,j}}-\log|u_{i}|\right)^{N_{1}}\cdot
    \\
    \prod_{i\in I_{B}} \left(\sum_{j\in I_{A'}\cup I_{B'}}-\log|z'_{j}|^{n_{i,j}}-\log|u_{i}|\right)^{N_{2}}.
  \end{multline*}
  Since we are working on an adapted coordinate neighborhood, the
  terms $-\log |z'_{j}|$ are bounded below by $\log 2$. The terms
  $-\log |u_{i}|$ can be positive or negative, but since $u_{i}$ is a
  unit they are bounded in $U'$. Since $U$ is also adapted the terms
  $-\log|z_{i}|$ are also bounded below by $\log 2$. Lemma
  \ref{lemm:6} implies that we can absorb the terms $\log|u_{j}|$ in
  the constant $C$. Thus there is a new constant $C'$ such that     
  \begin{equation}\label{eq:42}
    f^{\ast}\varphi\le 
    C'\prod_{i\in I_{A}} \left(\sum_{j\in I_{A'}}-\log|z'_{j}|^{n_{i,j}}\right)^{N_{1}}
    \prod_{i\in I'_{B}} \left(\sum_{j\in I_{A'}\cup I_{B'}}-\log|z'_{j}|^{n_{i,j}}\right)^{N_{2}},
  \end{equation}
  where $I'_{B}\subset I_{B}$ is the set of $i\in I_{B}$ such that
  there is a $j\in I_{A'}\cup I_{B'}$ with $n_{i,j}>0$. 
  Using Lemma \ref{lemm:3}~\ref{item:9} we obtain the bound
  \begin{displaymath}
    \prod_{i\in I_{A}} \left(\sum_{j\in
        I_{A'}}-\log|z'_{j}|^{n_{i,j}}\right)^{N_{1}}
    \le C''  \prod_{j\in I_{A'}}(-\log|z'_{j}|)^{|I_A|N_{1}}. 
  \end{displaymath}
  While using Lemma \ref{lemm:3}~\ref{item:10} and condition
  \eqref{eq:41} we deduce the bound
  \begin{displaymath}
    \prod_{i\in I'_{B}} \left(\sum_{j\in I_{A'}\cup
        I_{B'}}-\log|z'_{j}|^{n_{i,j}}\right)^{N_{2}}\le
    C''' \prod_{j\in I_{B'}} (-\log|z'_{j}|)^{N_{2}/|I_{B'}|}.
  \end{displaymath}
  Combining this two estimates with the inequality \eqref{eq:42} we
  obtain
  \begin{displaymath}
    f^{\ast}\varphi\le 
    C_{1}\prod_{j\in I_{A'}} (-\log|z'_{j}|)^{|I_{A}|N_{1}}
    \prod_{j\in I_{B'}} (-\log|z'_{j}|)^{N_{2}/|I_{B'}|}.
  \end{displaymath}
  Which gives the required bound for $f^{\ast}\varphi$. In order to
  bound the derivatives of $f^{\ast}\varphi$ we observe that the
  conditions \eqref{eq:39} and \eqref{eq:40} imply that there are
  holomorphic functions $\gamma _{i,j}$ such that, if $j\in I_{A'}\cup
  I_{B'}$
  then
  \begin{displaymath}
    \frac{\partial}{\partial z'_{j}}=
    \frac{1}{z'_{j}}\sum_{i\in I_{A}\cup I_{B}}\gamma
    _{i,j}z_{i}\frac{\partial}{\partial z_{i}}+
    \sum _{i\not \in I_{A}\cup I_{B}}\gamma _{i,j} \frac{\partial}{\partial z_{i}}.
  \end{displaymath}
  While if $j\not \in I_{A'}\cup
  I_{B'}$
  then
  \begin{displaymath}
    \frac{\partial}{\partial z'_{j}}=
    \sum_{i\in I_{A}\cup I_{B}}\gamma
    _{i,j}z_{i}\frac{\partial}{\partial z_{i}}+
    \sum _{i\not \in I_{A}\cup I_{B}}\gamma _{i,j} \frac{\partial}{\partial z_{i}}.
  \end{displaymath}
  These two equalities together with a computation similar to the
  previous one shows the required bound for all the derivatives of
  $f^{\ast}\varphi$. This proves the first statement. The second
  statement follows from the commutativity of the diagram
  \begin{displaymath}
    \xymatrix{\Sigma _B\Omega ^{\ast}_{X}(\log A)
      \ar[r]^{f^{\ast}}\ar[d]&
      \Sigma _{B'}\Omega ^{\ast}_{X}(\log A')\ar[d]\\
      E^{\ast}_{X}(\si A, \rd B)\ar[r]^{f^{\ast}}&
      E^{\ast}_{X'}(\si A', \rd B')
    }
  \end{displaymath}
  and Remark \ref{rem:6}.
\end{proof}

\begin{rmk}\label{rmk:4-5'}
  Compared with Proposition \ref{prop:10}, in Proposition
  \ref{prop:4-4'} we have the extra condition $f^{-1}(D)\subset
  D'$. This condition is needed because the differential forms in
  $E^{n}_{X}(\rd B)$ are not necessarily smooth along $B$. This
  contrast with $\Sigma _{B}E^{n}_{X}\subset E^{n}_{X}$.   
\end{rmk}
The next one is the compatibility with the duality of Lemma \ref{lemm:1}. 

\begin{prop}\label{prop:4-5}
  Let $\omega \in E_{X}^{n}(\si A, \rd B)$
  and $\eta \in E_{X}^{2d-n}(\si B, \rd A)$. Then $\omega \wedge \eta$ is
  locally integrable and the pairing
  \begin{equation}\label{eq:24}
    \rH^{n}(X\smallsetminus A,B;\R(a))\otimes
    \rH^{2d-n}(X\smallsetminus B,A;\R(d-a))\longrightarrow \R(0)
  \end{equation}
  is given, for $\omega $ and $\eta$ closed, by 
  \begin{displaymath}
    \omega(a) \otimes \eta(d-a)
    \mapsto \frac{1}{(2\pi i)^{d}}\int_{X}\omega \wedge \eta.
  \end{displaymath}
\end{prop}
\begin{proof}
  From the definition of slowly increasing/locally decreasing, it is
  clear that the wedge product of differential forms gives maps 
  \begin{equation}\label{eq:23}
    E_{X}^{k}(\si A,\rd B)\otimes
    E_{X}^{\ell}(\si B,\rd A)\longrightarrow
    E_{X}^{k+\ell}(\rd A\cup B).
  \end{equation}
  For more details on this product, see the proof of Proposition \ref{prop:9} below.
  Since any form in $E_{X}^{\ast}(\rd A\cup B)$ is locally integrable we
  have the first statement.
 The second statement follows from this observation
  and the fact that the maps \eqref{eq:23} and \eqref{eq:22} are
  compatible.  
\end{proof}

Next  comes the compatibility with the pullback in Proposition
\ref{prop:6}. 

\begin{prop}\label{prop:5} With the hypothesis and notation of
  Proposition \ref{prop:6},  if $\omega\in E_{X}^\ast(\si A,\rd B)$ then
  \begin{displaymath}
    \pi ^\ast\omega \in E_{\widetilde X}^\ast(\si \widehat A,\rd \widehat B\cup E),
  \end{displaymath}
  Moreover, the isomorphism $\varphi $ of Proposition \ref{prop:6} is
  given for any closed form $\omega \in E_{X}^\ast(\si A,\rd B)$ by
  $\varphi\{\omega\}=\{\pi ^{\ast}\omega\}$, where $\{\omega\}$
  denotes the cohomology class of $\omega$.
\end{prop}
\begin{proof}
  To prove the statement, we work with the local description of the
  blow-up. Let $U\coloneqq X\setminus D$, and
  $V\coloneqq \{z=(z_{1},\cdots , z_{d}); |z_{i}|< \frac{1}{2}\}$ be a
  coordinate neighborhood adapted to $D$ such that there are subsets
  $I_{A}$, $I_{B}$ and 
  $I_{W}$ of $\{1,\dots,d\}$ with 
  \begin{displaymath}
    A\cap V=\{ \prod_{i\in I_{A}} z_{i}=0\},\quad
    B\cap V=\{ \prod_{i\in I_{B}} z_{i}=0\},\quad
    W\cap V=\{ z_{i}=0,\ \forall i\in I_{W}\}.
  \end{displaymath}
  Then $I_{A}$ and $I_{B}$ are disjoint, but $I_{B}\cap I_{W}$ is
  nonempty. The subsets of the shape of $V$ cover $X$ because by
  assumption $W$ has only 
  normal crossings with $D$.

  Let $s=|I_{W}|$. The blow-up
  $\widetilde{V}$ of $V$ at $W\cap V$ is given by the closed subset of
  $V\times \P^{s-1}$ satisfying the set of equations
  $\{z_{i}\alpha_{j}-z_{j}\alpha_{i}=0, \ i,j\in I_{W}\}$,
  where the projective coordinates of $\P^{s-1}$ are
  $(\alpha_{i}\colon)_{i\in I_{W}}$. The blow up $\widetilde V$ can
  be covered by open subsets, $\widetilde V_{j}$, $j\in I_{W}$ defined by the
  condition $\alpha _{j}\not =0$. Fix any element $j\in I_{W}$. We can
  consider the coordinates in $\widetilde V_{j}$ given by
  \begin{displaymath}
    z'_{i}=z_{i}, \text{ for }i\not \in I_{W}, \qquad z'_{i}=\alpha
    _{i}/\alpha _{j}, \text{ for }j\not = i \in I_{W},\qquad z_{j}'=z_{j}.
  \end{displaymath}
  This is not an adapted coordinate neighborhood because the $|z'_{i}|$
  can be bigger than $1/2$. For the sake of brevity we discuss only
  the region $|z_{i}'|\le 1/2$. The remainder of $\widetilde V$ can be
  discussed similarly.
  
  The map $\pi $ is given by the equations
  \begin{equation}\label{eq:12}
    z_{i}=z_{i}', \text{ for }i\not \in I_{W},\qquad 
    z_{i}=z_{i}'z_{j}', \text{ for }i \in I_{W}\smallsetminus \{j\},\qquad
    z_{j}=z_{j}'.
  \end{equation}
  We write $\xi'_{i}$ as in \eqref{eq:13} with respect to the
  variables $z_{i}'$, taking into account that the exceptional divisor
  $E=\{z_{j}=0\}$ is part of the normal crossings divisor.  That is
  \begin{displaymath}
    \xi'_{i}=dz'_{i}/z'_{i},\text{ if }i\in I_{A}\cup
    I_{B}\cup\{j\},\qquad
    \xi'_{i}=dz'_{i},\text{ otherwise.}
  \end{displaymath}
  A simple computation using \eqref{eq:12} shows that 
  \begin{equation}\label{eq:14}
    \pi ^{\ast}\xi_{i}=
    \begin{cases}
      \xi_{i}',&\text{ if }i\not \in I_{W},\\
      \xi'_{j},&\text{ if }i=j\text{ and }j\in I_{A}\cup I_{B},\\
      z_{j}'\xi'_{j} &\text{ if }i=j\text{ and }j\not \in I_{A}\cup I_{B},\\
       \xi_{i}'+\xi_{j}',&\text{ if }i\in (I_{A}\cup I_{B})\cap
                           I_{W}\smallsetminus \{j\},\\
      z_{i}'z_{j}'\xi'_{j}+z_{j}'\xi_{i}'&\text{ if }i\not \in I_{A}\cup
                                           I_{B}\text{ and } i\in  I_{W}\smallsetminus \{j\},
    \end{cases}
  \end{equation}
and similar relations for $\overline{\xi}_{i}$. Since the pullback of
each $\xi_{i}$ is a linear combination of the $\xi_{i}'$  with smooth
coefficients, it is enough to show that, if $f\in E_{X}^0(\si A,\rd
B)$, then
\begin{displaymath}
  \pi ^{\ast}f = f\circ \pi \in E_{\widetilde X}^0(\si \widehat A,\rd \widehat B\cup E).
\end{displaymath}
We start studying the growth condition of the function and later we
will study the growth of the derivatives. 
  Now, let $f\in E^{0}(\si A, \rd B)$. There exists an integer $N_1\in
  \Z$ and for each integer $N_2\in \Z$  there is a
  constant $C_{N_1,N_2}$
\begin{displaymath}
 \left | f
    \right|\le C_{N_1,N_2}
    \prod_{i\in I_A}|\log|z_i||^{N_1}
    \prod_{k\in I_B}|\log|z_k||^{N_2}.
  \end{displaymath}
  Then 
  \begin{align*}
    \left | f\circ \pi 
    \right|&\le C_{N_1,N_2}
    \prod_{i\in I_A\cap (I_{W}\smallsetminus \{j\})}|\log|z'_{j}z'_i||^{N_1}
     \prod_{i\in I_A\smallsetminus (I_{W}\smallsetminus \{j\})}|\log|z'_i||^{N_1}\\
    &\qquad \qquad \qquad \prod_{i\in I_B\cap (I_{W}\smallsetminus \{j\})}|\log|z'_{j}z'_i||^{N_2}
     \prod_{i\in I_B\smallsetminus (I_{W}\smallsetminus \{j\})}|\log|z'_i||^{N_2}
  \end{align*}
  Since $|z'_{i}|\le 1/2$ for all $i$, there is a constant $C>0$ such
  that
  \begin{displaymath}
    \log|z'_{j}z_i'|=\log|z'_{j}|+\log|z'_i|\le C\log|z'_{j}|\log|z'_i|. 
  \end{displaymath}
  Therefore, there is another constant $C'_{N_{1},N_{2}}$ such that 
  \begin{equation}\label{eq:20}
    \left | \pi ^{\ast }f 
    \right|\le C'_{N_1,N_2} |\log|z'_j||^{aN_{1}+bN_2}
    \prod_{i\in I_A\smallsetminus \{j\}}|\log|z'_i||^{N_1}
   \prod_{i\in I_{B}\smallsetminus \{j\}}|\log|z'_i||^{N_2},
 \end{equation}
 where $a=|I_{A}\cap I_{W}|$ and $b=|I_{B}\cap I_{W}|$. By hypothesis
 $W\subset B$ so $I_{B}\cap I_{W}\not = \emptyset $ and $b>0$. Since
 $N_{1}$ is fixed, and $N_{2}$ is arbitrary, we can make
 $aN_{1}+bN_{2}$ arbitrarily negative. This shows that   $\pi
 ^{\ast}f$ satisfies the required bounds.

 The chain rule together with equations \eqref{eq:12} show that
 \begin{displaymath}
   \partial/\partial z'_{i}=
   \begin{cases}
     \partial/\partial z_{i},
     &\text{ if }i\not \in I_{W},\\
     z_{j}\partial/\partial z_{i},
     &\text{ if }i\in I_{W}\smallsetminus \{j\}\\
     \sum_{k\in I_{W}} (z_{k}/z_{j})\partial/\partial z_{k},
     & \text{ if }i = j. 
   \end{cases}
 \end{displaymath}
 This implies that the differential operator  $\partial'{} ^{\alpha,\beta  }$
 can be written as a linear combination of terms of the form
 \begin{math}
   z^{\alpha '',\beta ''}\partial^{\alpha ',\beta '}
 \end{math}
 where
 \begin{equation}
   \label{eq:17}
   \begin{gathered}
      \alpha '_{j}\le \alpha _{j},\qquad \alpha '_{i}\ge \alpha _{i}
 \text{ for } i\in I_{W}\smallsetminus \{j\},\qquad \alpha '_{i}=\alpha _{i}\text{ for }i\not \in
 I_{W},\\
 \alpha _{i}''=\alpha '_{i}-\alpha _{i}, \qquad
 \alpha ''_{j}=\sum _{i\in I_{W}\smallsetminus \{j\}}\alpha _{i}+\alpha '_{j}-\alpha _{j},
\\
      \beta '_{j}\le \beta _{j},\qquad \beta '_{i}\ge \beta _{i}
 \text{ for } i\in I_{W}\smallsetminus \{j\},\qquad \beta '_{i}=\beta _{i}\text{ for }i\not \in
 I_{W},\\
 \beta _{i}''=\beta '_{i}-\beta _{i}, \qquad
 \beta ''_{j}=\sum _{i\in I_{W}\smallsetminus \{j\}}\beta _{i}+\beta '_{j}-\beta _{j},
   \end{gathered}
 \end{equation}
 Thus, we have to show that any function of the form
 $z^{\alpha '',\beta ''}\partial^{\alpha ',\beta '}f$ with
 $\alpha', \alpha '', \beta ', \beta ''$ satisfying the conditions
 \eqref{eq:17}, when expressed in terms of the coordinates $z'$
 satisfies the bounds required to $\partial'{}^{\alpha ,\beta }\pi
 ^{\ast}f$. For brevity of the exposition we write only the case
 $\partial^{\alpha }$ as the general case follows the same principle
 and is only more cumbersome to write. We have
 \begin{displaymath}
   \left |z^{\alpha ''} \partial ^{\alpha '}f\right | \le C_{N_1,N_2,\alpha' }
    \frac{\prod_{i\in I_A}|\log|z_i||^{N_1}
      \prod_{j\in I_B}|\log|z_j||^{N_2}}
    {\prod_{i\in I_{A}\cup I_{B}}|z_{i}|^{\alpha' _{i}-\alpha _{i}''}}.
  \end{displaymath}
  Using the relations \eqref{eq:17} and \eqref{eq:12} one proves
  that,
  \begin{displaymath}
    \prod_{i\in I_{A}\cup I_{B}}|z_{i}|^{\alpha' _{i}-\alpha _{i}''}=
    |z'_{j}|^{\alpha''' _{j}}\prod_{i\in I_{A}\cup I_{B}\smallsetminus \{j\}}|z'_{i}|^{\alpha_{i}},
  \end{displaymath}
where, 
\begin{displaymath}
  \alpha ^{'''}_{j}=
  \begin{cases}
    \alpha _{j}-\sum_{i\in I_{W}\smallsetminus (I_{A}\cup
    I_{B})}\alpha _{i}, &\text{ if }j\in I_{A}\cup I_{B}\\
    \alpha _{j}-\sum_{i\in I_{W}\smallsetminus (I_{A}\cup
    I_{B}\cup \{j\})}\alpha _{i}-\alpha '_{j}, &\text{ if }j\not \in
                                                 I_{A}\cup I_{B}.
  \end{cases}
\end{displaymath}
In any case, $\alpha '''_{j}\le \alpha _{j}$.  Combining the bound
 \eqref{eq:20} that we obtained for the numerator, and the previous
 computation taking into account that $\alpha _{j}'''\le \alpha _{j}$
 we deduce
 \begin{displaymath}
   \left |\partial'{} ^{\alpha }\pi ^{\ast }f\right |\le
   C'_{N_{1},N_{2},\alpha }|\log|z'_j||^{aN_{1}+bN_2}\frac 
    {\displaystyle  \prod_{i\in I_A\smallsetminus \{j\}}|\log|z'_i||^{N_1}
   \prod_{i\in I_{B}\smallsetminus
     \{j\}}|\log|z'_i||^{N_2}}{\displaystyle  \prod_{i\in I_{A}\cup
     I_{B}\cup\{j\}}|z_{i}|^{\alpha _{i}}} 
 \end{displaymath}
 with $b> 0$

The last statement follows from the fact that the map
 \begin{displaymath}
 \pi ^{\ast}\colon E_{X}^\ast(\si A,\rd B) \to  E_{\widetilde X}^\ast(\si \widehat A,\rd
 \widehat B\cup E)
 \end{displaymath}
  is compatible with the map \eqref{eq:25}.
\end{proof}

Even without the assumption that $W\subset B$, we can show that the map $\varphi$ in Proposition
\ref{prop:2}~\ref{item:3} and the isomorphism of Corollary \ref{cor:4}
can be represented using slowly
increasing/rapidly decreasing forms.

\begin{prop}\label{prop:8}   Let $W\subset X$ be a smooth irreducible subvariety of
  dimension $d_{W}$ such that $W$
  has only normal crossings with $D$. Let $\pi \colon \widetilde X\to X$ be
  the blow up of  $X$ along $W$, $\widehat A$ and $\widehat B$ the
  strict transforms of $A$ and $B$ and $E$ the exceptional
  divisor. Let $r$ be the number of local analytical components of $A$ containing
  $W$ and $p> d_{W}+r$. Let
  \begin{math}
    \omega \in F^{p}E_{X}(\si A,\rd B).
  \end{math}
  Then
  \begin{displaymath}
    \pi ^{\ast}\omega \in F^{p}E_{\widetilde X}(\si \widehat A,\rd
    \widehat B\cup E).
  \end{displaymath}
 Moreover, if $\omega $ is closed, then $\varphi\{\omega \}=\{\pi
 ^{\ast} \omega \}$. Where $\{\omega \}$ denotes the cohomology class
 of $\omega $.
\end{prop}
\begin{proof}
  If $W\subset B$ this follows from Proposition \ref{prop:5} so we
  assume that $W\not \subset B$. As in the proof of Proposition
  \ref{prop:5} we use the local expression of the blow up. Let
  $I_{A}$, $I_{B}$ and $I_{W}$ be as in the said proof. But now
  $I_{W}$ is disjoint from $I_{B}$ and $I_{W}\cap I_{A}$ has $r$
  elements. Now the argument is very similar to the one in Remark
  \ref{rem:4}. We can decompose $\omega $ as a sum of monomials
   \begin{displaymath}
     \omega = \sum \eta_{i}\wedge \omega_{i}, 
   \end{displaymath}
   where $\eta_{i}$ is a wedge product of forms $dz_{k}/z_{k}$ with
   $k\in I_{A}\cap I_{W}$ and $\omega _{i}$ does not contain any
   $dz_{k}$ for $k\in I_{A}\cap I_{W}$.
   Since $p>d_{W}+r$ we
 deduce that $\omega _{i}$ contains at least one term of the form $dz_{k_{i}}$ with
 $k_{i}\in I_{W}\setminus I_{A}$. So $\omega _{i}=\omega
 _{k_{i}}'\wedge dz_{k_i}$ and 
   \begin{displaymath}
     \omega' _{k_i}\in E _{X}^{\ast}(\si A, \rd B).
   \end{displaymath}
   Since $k_{i}\in I_{W}$ we have that
   \begin{displaymath}
     \pi ^{\ast}dz_{k_{i}} =
     \begin{cases}
       z_{j}'dz_{k_{i}}'+
       z_{j}'z_{k_{i}}'\frac{dz_{j}'}{z_{j}'},&\text{ if }k_{i}\not =
                                                j,\\
       z_{j}'\frac{dz_{j}'}{z_{j}'}\text{ if }k_{i}=
                                                j,
     \end{cases}
   \end{displaymath}
   In any case,
   \begin{displaymath}
     \pi ^{\ast}dz_{k_{i}} \in \Omega^{1}
     _{\widetilde X}(\log E)\otimes \caO(-E)
     \subset E^{1}_{\widetilde X}(\rd E).
   \end{displaymath}
In consequence 
 \begin{displaymath}
   \pi ^{\ast}(\omega )\in
   E_{\widetilde X}^{\ast}(\si \widehat{A}, \rd\widehat B\cup
   E).
 \end{displaymath}
 The last statement follows from the compatibility of this map and the
 one in Remark \ref{rem:4}.
\end{proof}

\subsection{The exceptional cup-product revisited}
\label{sec:except-cup-prod}
The last compatibility we check is that the cup-products of
Proposition \ref{prop:7} and Corollary \ref{cor:3} can also be
represented with slowly increasing/rapidly decreasing differential
forms.  

\begin{prop}\label{prop:9}
  Let $C$ be a normal crossing divisor without common components with
  $A$ and $B $ such that $A\cup B \cup C$ is still a normal crossing
  divisor.  
  Then the wedge product of differential forms induce a product of
  complexes 
  \begin{displaymath}
    E^{\ast}_{X}(\si A\cup B,\rd C)\otimes
    E^{\ast}_{X}(\si A,\rd B\cup C)\longrightarrow
    E^{\ast}_{X}(\si A,\rd B\cup C)
  \end{displaymath}
  that induces the product of Corollary \ref{cor:3} in cohomology. 
\end{prop}
\begin{proof}
For simplicity, we can take $C=\emptyset$. Let $V$ be a coordinate
neighborhood adapted to $D=A\cup B$. For $\eta\in E_{X}^n(\si A\cup B)$
and $\omega\in E_{X}^{m}(\si A, \rd B)$, we want to show that $\eta\wedge
\omega\in E_{X}^{n+m}(\si A, \rd B)$. Since $E_{X}^{\ast}(\si
A\cup B)$  and  $E_{X}^{\ast}(\si A,\rd B)$ are  a $E_{X}^0(\si A\cup
B)$-module and a $E_{X}^{0}(\si A, \rd B)$-module respectively with the
same generators, we need to show the
following: Given elements $f\in E_{X}^0(\si A\cup B)$ and $g\in E_{X}^0(\si A,
\rd B)$, then the product $fg\in E_{X}^{0}(\si A, \rd B)$. To do this, we use
the product rule of (partial) differentiation: 
\begin{displaymath}
\partial^{\alpha,\beta}(fg) = \sum_{(\alpha ',\beta  ')+(\alpha '',\beta
  '')=(\alpha ,\beta )}  C_{\alpha ,\beta ,\alpha ',\beta
  '}(\partial^{\alpha',\beta'}f)(\partial^{\alpha'',\beta''}g), 
\end{displaymath}
with $C_{\alpha ,\beta ,\alpha ',\beta'}$ a combinatorial number
depending on the multi-indices. 
Since $f\in E^0(\si A\cup B)$, in $V$ there exists an integer $N_{f}
\in  \Z$,
such that for each pair of multi-indexes $\alpha'$ and $\beta'$ in
$\mathbb{N}^{d}$, there is a constant $C_{N_{f},\alpha',\beta'}$ such
that the estimate
\begin{displaymath}
|(\partial ^{\alpha',\beta'}f)(z)|\le C'_{N_{f},\alpha' ,\beta'}
    \frac{\prod_{i\in I_A\cup I_{B}}|\log|z_i||^{N_{f}}}{\prod_{i\in
        I_{A}\cup I_{B}}|z_{i}|^{\alpha' _{i}+\beta' _{i}}} 
  \end{displaymath}
  is satisfied. 
Since $g\in E^{0}(\si A, \rd B)$, there exists an integer $N_{g,1}$,
such that for each $M\in \mathbb{Z}$, and for each pair of
multi-indexes $\alpha'',\beta''\in\mathbb{N}^{d}$, there is a constant
$C_{N_{g,1}, M, \alpha'',\beta''}$ satisfying 
\begin{displaymath}
|(\partial^{\alpha'',\beta''}g)(z)|\le C''_{N_{g,1},N_{2}-N_{f},\alpha'' ,\beta''}
    \frac{\prod_{i\in I_A}|\log|z_i||^{N_{g,1}}
    \prod_{j\in I_B}|\log|z_j||^{M}}{\prod_{i\in I_{A}\cup I_{B}}|z_{i}|^{\alpha'' _{i}+\beta'' _{i}}}
\end{displaymath}
Now, using the triangle inequality, and writing $N_{1}=N_{f}+N_{g,1}$ and 
\begin{displaymath}
C_{N_{1},N_{2},\alpha,\beta}=\sum_{(\alpha ',\beta  ')+(\alpha '',\beta
  '')=(\alpha ,\beta )} C_{\alpha ,\beta ,\alpha ',\beta
  '}C'_{N_{f},\alpha' ,\beta'} C''_{N_{g,1}, M, \alpha'',\beta''},  
\end{displaymath}
we get that for every $\alpha \in \N^{d}$, $\beta\in \N^{d} $ and
$N_{2}\in \Z$,
\begin{displaymath}
  |\partial^{\alpha,\beta}(fg)(z)|\le C_{N_{1},N_{2},\alpha,\beta}
  \frac{\prod_{i\in I_A}|\log|z_i||^{N_{1}}
    \prod_{j\in I_B}|\log|z_j||^{N_{2}}}{\prod_{i\in I_{A}\cup
      I_{B}}|z_{i}|^{\alpha _{i}+\beta _{i}}} .
\end{displaymath}
Hence we get the  desired estimate showing $fg\in E_{X}^{0}(\si A, \rd
B)$.

The last statement follows from the compatibility of this product with
the one in Remark \ref{rem:2}.
\end{proof}

\subsection{A different complex computing the same cohomology}
\label{sec:forms-vanishing-ncd}
We next recall the complex of differential forms vanishing along a
normal crossings divisor with logarithmic singularities along a second normal
crossings divisor and compare it with the complex of slowly
increasing/rapidly decreasing forms.

\begin{df}\label{def:5}
Let $E^{\ast}_{X}(\log A)$ denote the complex of differential forms on
$X$ with logarithmic singularities along $A$ introduced in
\cite{Burgos:CDc}. Let $\iota \colon \widetilde B \to X$ be the map
from the normalization $\widetilde B$ of $B$, that is a smooth
variety, to $X$. Then the \emph{complex of differential forms on $X$
  with logarithmic singularities along $A$ and vanishing along $B$} is
the complex 
\begin{equation*}
  \Sigma _{B}E^{\ast}_{X}(\log A)=
  \{\omega \in E^{\ast}_{X}(\log A)\mid
  \iota ^{\ast} \omega =0\}
  \subset E^{\ast}_{X}(\log A).
\end{equation*}
It is the space of global sections of a complex of sheaves
$\Sigma _{B}\caE^{\ast}_{X}(\log A)$.
\end{df}

\begin{prop}\label{prop:11} There is an inclusion
  \begin{displaymath}
      \Sigma _{B}E^{\ast}_{X}(\log A) \hookrightarrow E^{\ast}_{X}(\si
      A, \rd B)
  \end{displaymath}
  that is a filtered quasi-isomorphism with respect to the Hodge
  filtration. 
\end{prop}
\begin{proof}
  Let $\omega \in \Sigma _{B}E^{\ast}_{X}(\log A)$ and 
  $V$ a coordinate neighborhood adapted to $D$. Since, in particular,
  $\omega \in E^{\ast}_{X}(\log A)$, we can write
  \begin{displaymath}
    \omega =\sum \omega_{\underline n,I,J}   \prod \log|z_i|^{n_i}\xi_I\wedge \bar \xi_J,
  \end{displaymath}
  where $\underline {n}=(n_i)_{i\in I_A}$, $n_i\ge 0$ are integers, $I,J\subset
  I_A$ and $\omega_{\underline n,I,J}$ are smooth differential forms
  that contain neither $dz_i$ nor $d\bar z_i$ for $i\in I_A$. The fact
  that $\iota ^\ast\omega =0$ implies that each $\omega
  _{\underline{n},I,J}$ can be written as a sum of terms, each one
  having at least a factor of the form $z_i, \bar z_i, dz_i, d\bar z_i$ for
  some $i\in I_B$. Since, for $i\in I_A$ we have that
  \begin{displaymath}
    \log|z_i|^{n_i}, \ \xi_i,\ \bar \xi_i  \in \caE_X^\ast(\si A)
  \end{displaymath}
  and for $i\in I_B$,
  \begin{displaymath}
    z_i,\  \bar z_i,\  dz_i,\  d\bar z_i \in \caE_X^\ast(\rd B),
  \end{displaymath}
  the first statement follows from Proposition \ref{prop:9}. The
  second statement follows from the commutativity of the diagram
  \begin{displaymath}
    \xymatrix{
      & \Sigma _B\caE^{\ast}_X(\log A)\ar[dd]\\
      \Omega ^{\ast}_X(\log D)\otimes \caO(-B)\ar[ur]\ar[dr] &\\
      & \caE^{\ast}_X(\si A,\rd B)
    }
  \end{displaymath}
  and that the two non-vertical arrows are filtered
  quasi-isomorphisms. So the same is true for the vertical one.  
\end{proof}

\section{Main Construction}   
\label{sec:main-construction}
In this section we will put the abstract construction of heights
associated with framed mixed Hodge structures into a more concrete
situation involving Bloch's higher algebraic cycles. Let $p$, $q$, $n$
be integers satisfying $p+q=d+n+1$. Recall that $d$ is the dimension
of the projective smooth complex variety $X$.

Let $Z\in
Z^{p}(X,n)_{00}$ and $W\in Z^{q}(X,n)_{00}$ be two refined cycles
intersecting properly (see Definition \ref{def:10} for the notion of proper intersection). We want to put 
the class of $Z$, and the dual of the class of $W$ in the same mixed
Hodge structure. This in turn, will allow us to define a framed mixed
Hodge structure, and corresponding heights.

\subsection{The class of a higher cycle and its regulator}
\label{sec:class-higher-cycle}
We recall briefly
the construction of the class of $Z$, extracting key details from
\cite[Sections 3.2, 3.3]{BGGP:Height}.
On $(\P^{1})^{n}$ let
 \begin{align*}
    A &= \{(t_1,\cdots , t_n)\mid \exists i, t_i=1\},\\
    B & =\{(t_1,\cdots , t_n)\mid \exists i, t_i\in \{0, \infty\}\}.
  \end{align*}
  Then $A\cup B$ is a simple normal crossing divisor. For any variety $X$ we also denote
\begin{displaymath}
  A_X\coloneqq X\times A,\qquad B_X:= X\times B.
\end{displaymath}
By \cite[Proposition 3.3]{BGGP:Height} there is a well defined class of the cycle
$Z$, denoted by $\cl(Z)$ in  $\rH^{2p}_{|Z|}(X\times
(\P^1)^n\smallsetminus A_{X}, B_{X}; p)$. A consequence of \cite[Lemma
3.4]{BGGP:Height} and semipurity is that the mixed Hodge structure $\rH^{2p}_{|Z|}(X\times
(\P^1)^n\smallsetminus A_{X}, B_{X}; p)$ is pure of weight zero, so
$\cl(Z)\in \Gr^{W}_{0}\rH^{2p}_{|Z|}(X\times
(\P^1)^n\smallsetminus A_{X}, B_{X}; p)$.
Moreover, by \cite[Lemma
3.5]{BGGP:Height} this class 
vanishes on $\rH^{2p}(X\times (\P^1)^{n}\smallsetminus A_{X},B_{X};
p)$, therefore can be lifted 
to a class in $\Gr_{0}^{W}\rH^{2p-1}(X\times (\P^1)^n\smallsetminus
A_{X}\cup |Z|,B_{X}; p)$.
This lifting is unique because
\begin{displaymath}
  \rH^{2p-1}(X\times (\P^1)^{n}\smallsetminus A_{X},B_{X}; p) \cong \rH^{2p-n-1}(X; p)
\end{displaymath}
has pure weight $-n-1$. 

We denote this lifting also by
$\cl(Z)$.
We note also that, since
$\rH^{2p-1}(X\times (\P^1)^{n}\smallsetminus A_{X},B_{X}; p)$ is pure
of weight $-n-1$ and $\rH^{2p}_{|Z|}(X\times
(\P^1)^n\smallsetminus A_{X}, B_{X}; p)$ is pure of weight zero and
Hodge type $(0,0)$, there
is an isomorphism
\begin{multline*}
  \Gr_{0}^{W}\rH^{2p-1}(X\times (\P^1)^n\smallsetminus
  A_{X}\cup |Z|,B_{X}; p)\\ \cong
  I^{0,0} \rH^{2p-1}(X\times (\P^1)^n\smallsetminus
  A_{X}\cup |Z|,B_{X}; p).
\end{multline*}
Thus we get a class
\begin{multline*}
  \cl(Z)\in I^{0,0} \rH^{2p-1}(X\times (\P^1)^n\smallsetminus
  A_{X}\cup |Z|,B_{X}; p)\\ \subset \rH^{2p-1}(X\times (\P^1)^n\smallsetminus
  A_{X}\cup |Z|,B_{X}; p).
\end{multline*}
 In fact one can choose a closed differential form whose
class agrees with $\cl(Z)$: Let $\pi_{Z}\colon \caX_{Z}\rightarrow
X\times (\P^1)^{n}$ be a birational map that is an isomorphism outside
$|Z|$  and such that $\pi _{Z}^{-1}(A_X\cup B_X\cup |Z|)$ is a simple
normal crossings divisor. Let 
$\widehat{A}_{X}$ and $\widehat{B}_{X}$ be the corresponding strict
transforms and $E_Z=\pi ^{-1}_{Z}(Z)$. Then $\pi _{Z}$ induces an
isomorphism of pairs of spaces 
\begin{displaymath}
  (\caX_{Z}\setminus \widehat A_{X}\cup E_{Z},\widehat B_{Z})
  \cong
  (X\times (\P^{1})^{n}\setminus A_{X}\cup |Z|,B_{Z}).
\end{displaymath}
Then we can represent the class $\cl(Z)$ using forms in $\caX_{Z}$
with logarithmic singularities along $\widehat A_{X}\cup E_{Z}$ and
vanishing along $\widehat B_{X}$. 
\begin{prop}\cite[Proposition 3.6]{BGGP:Height}\label{prop-difformZ}
There are
  differential forms
  \begin{enumerate}
  \item $\eta_{Z}\in F^{p}\Sigma _{\widehat{B}_X}E^{2p-1}_{\caX_{Z}}(\log
  \widehat{A}_X\cup E_{Z})$ such that $d\eta_{Z}=0$ and, if we denote
  by $[\eta_{Z}(p)]$ the current on $X\times (\P^{1})^{n}$ determined
  by the locally integrable form $\eta_{Z}(p)$, 
  \begin{equation}\label{eq:54}
    d[\eta_{Z}(p)]=-\delta _Z.
  \end{equation}
  this implies that 
  \begin{equation}
    \label{eq:19}
    \{(0,\eta_{Z}(p))\}=\cl(Z)\in
    \rH^{2p}_{|Z|}(X\times (\P^1)^n\setminus A_{X},B_{X}; p). 
  \end{equation}  
\item $\theta_{Z} \in F^{p-n}\Sigma _{\widehat{B}_X}
  E^{2p-1}_{\caX_{Z}}(\log
  \widehat{A}_X)$ with $d\theta_{Z} =0$ and
  $\overline \theta_{Z} =(-1)^{p-1}\theta_{Z}$.
\item $g_{Z}\in F^{p-1}\cap \overline F^{p-1} \Sigma
  _{\widehat{B}_X}E^{2p-2}_{\caX_{Z}}(\log \widehat{A}_{X}\cup E_{Z})$
  satisfying $\overline g_{Z}=(-1)^{p-1}g_{Z}$ and 
  \begin{equation}
    \label{eq:21}
    d g_{Z} = \frac{1}{2}(\eta_{Z}+(-1)^{p-1}\overline \eta_{Z})-\theta _{Z}.
  \end{equation}
  \end{enumerate}
\end{prop}

\begin{rmk}\label{rem:13} The form $g_{Z}(p-1)=g_Z\otimes \bfone(p-1)_{\dR}$
  satisfies 
  \begin{displaymath}
    g_Z(p-1)\in
    \fD^{2p-1}( \Sigma
  _{\widehat{B}_X}E^{\ast}_{\caX_{Z}}(\log \widehat{A}_{X}\cup E_{Z}),p)
\end{displaymath}
and, denoting by $[g_Z(p-1)]$ the current on $X\times (\P^{1})^{n}$
associated to the differential form $g_{Z}(p-1)$,
\begin{displaymath}
  d_{\fD}[g_Z(p-1)]+\delta _Z=
  -[2\bar \partial \theta ^{p,p-1}(p)]
  \in 
    \fD^{2p}( D^{\ast}_{X\times (\P^1)^n\setminus A_{X}},p).
\end{displaymath}
Thus, $g_Z$ is a Green form for the cycle $Z$  in the sense of
\cite{Burgos:CDB} with the extra condition that $g_{Z}$ vanishes along $B_X$.
\end{rmk}
\begin{rmk}
  Since $\overline \theta_{Z} =(-1)^{p-1}\theta_{Z} $ and is a closed form,
  the tensor product $\theta _{Z}(p-1)=\theta_{Z} \otimes \bfone(p-1)_{\dR}$ represents a
  class
  \begin{displaymath}
    \{\theta_{Z} (p-1)\}\in
    F^{1-n}\cap \overline F^{1-n}\cap\rH^{2p-1}(X\times
    (\P^{1})^{n}\setminus A_{X},B_{X},\R(p-1))_{\Betti}.
  \end{displaymath}
  that we denote $\boldsymbol{\theta }_{Z}$. 
  After applying the isomorphism 
  \begin{displaymath}
    \rH^{2p-1}(X\times
    (\P^{1})^{n}\setminus A_{X},B_{X},\R(p-1))\cong
    \rH^{2p-n-1}(X,\R(p-1))
  \end{displaymath}
  and the isomorphism of example \ref{exm:11}
  the class $\boldsymbol{\theta }_{Z}$ is sent to an element of
  \begin{equation}\label{eq:48}
    F^{1-n}\cap \overline
    F^{1-n}\cap\rH^{2p-n-1}(X,\R(p-1))_{\Betti}\simeq
    \frac{\rH^{2p-n-1}(X,\R(p))_{\dR}}{F^{0}+\rH^{2p-n-1}(X,\R(p))_{\Betti}}.
  \end{equation}
  Let $H=\rH^{2p-n-1}(X,\R(p))_{\dR}$. This is a pure $\R$-Hodge
  structure of weight $-n-1$. Then the groups \eqref{eq:48} agree with
  \begin{displaymath}
    \Ext^{1}_{\RMHS}(\R(0),H)\simeq \rH^{2p-n}_{\fD}(X,p).
  \end{displaymath}
  The image of $\boldsymbol{\theta }_{Z}$ in
  $\Ext^{1}_{\RMHS}(\R(0),H)$ is the class of the extension
  associated to the cycle $Z$ \cite[Prop. 3.6]{BGGP:Height}, while the
  image of $\boldsymbol{\theta }_{Z}$ in the Deligne cohomology group
  $\rH^{2p-n}_{\fD}(X,p)$ is Goncharov regulator $\caP(Z)$.
\end{rmk}

A key consequence of Proposition \ref{prop-difformZ} is the following. 
\begin{prop}\label{i-p-q}
Let $\rH_{Z}(p)\coloneqq \rH^{2p-1}(X\times (\P^1)^{n}\setminus
A_{X}\cup |Z|, B_{X}; p)$. Then the cohomology class
$\{\eta_{Z}(p)\}$
belongs to $I^{0,0}_{\rH_{Z}(p)}$. 
 \end{prop}
\begin{proof}
From the Deligne splitting, we have
\begin{displaymath}
I^{0,0}_{\rH_{Z}(p)}=F^{0}\cap W_{0}\cap \left(\overline{F}^{0}\cap W_{0} +\sum_{j\ge 2} \overline{F}^{-j+1}\cap W_{-j}\right)\rH_{Z}(p)_{\dR}.
\end{displaymath}
At the level of cohomology, we have the equality
\begin{displaymath}
  \{\eta_{Z}(p)\}=\overline{\{\eta_{Z}(p)\}} +\{\theta_{Z}(p)\}
\end{displaymath}
 The
class $\{\eta_{Z}(p)\}$ belongs to
$F^{0}\rH_{Z}(p)$ by construction.

Since 
$\rH^{2p}_{|Z|\setminus A_{X}}(X\times (\P^1)^{n}\setminus A_{X},
B_{X}; p)$ is pure of weight zero and  
\begin{equation}\label{eq:43}
\rH^{2p-1}(X\times (\P^1)^{n}\setminus A_{X}, B_{X}; p)\cong \rH^{2p-n-1}(X; p)
\end{equation}
is a pure Hodge structure of weight $-n-1<0$ we deduce that
$W_{0}\rH_{Z}(p)=\rH_{Z}(p)$. Thus $\{\eta_{Z}(p)\}\in (F^{0}\cap
W_0)\rH_{Z}(p)_{\dR}$. Similarly 
 $\overline{\{\eta_{Z}(p)\}}\in
(\overline{F^{0}}\cap W_{0})\rH_{Z}(p)_{\dR}$.
Moreover, 
\begin{displaymath}
  \{\theta_{Z}(p)\}\in F^{-n}\rH^{2p-1}(X\times
  (\P^1)^{n}\setminus A_{X}, B_{X}; p),
\end{displaymath}
and by \eqref{eq:43}, we deduce
that its image in $\rH_{Z}(p)$ satisfies 
\begin{displaymath}
\{\theta_{Z}(p)\}\in (F^{-n}\cap W_{-n-1})\rH_{Z}(p)_{\dR}.
\end{displaymath}
Therefore we also have
\begin{displaymath}
\{\theta_{Z}(p)\}=-\overline{\{\theta_{Z}(p)\}}\in (\overline{F}^{-n}\cap 
W_{-n-1})\rH_{Z}(p)_{\dR}.
\end{displaymath}
We conclude
\begin{displaymath}
\{\eta_{Z}(p)\}\in F^{0}\cap W_{0}\cap (\overline{F}^{0}\cap W_{0}+
\overline{F}^{-n}\cap W_{-n-1})\rH_{Z}(p)_{\dR}\subset
I^{0,0}_{\rH_{Z}(p)}. 
\end{displaymath}
\end{proof}

We now consider $X\times (\P^{1})^{2n}$ with the
two projections $p_{1}$ and $p_{2}$ to $X\times (\P^{1})^{n} $. We
will denote $p^{-1}_{1}(|Z|)$ and $p^{-1}_{2}(|W|)$ still by
$|Z|$  and $|W|$ respectively. It will be clear from the context
whether we are working with the pullbacks or not. We also denote by
$A_{1}\coloneqq p^{-1}_{1}(A_X)$, $A_{2}\coloneqq p^{-1}_{2}(A_{X})$,
$B_{1}\coloneqq p^{-1}_{1}(B_{X})$,
$B_{2}\coloneqq p^{-1}_{2}(B_{X})$.

Next, recall that there is a differential form
\begin{displaymath} \alpha
\coloneqq(-1)^{n}\left(\frac{dt_{1}}{t_{1}}\wedge\cdots\wedge
\frac{dt_{n}}{t_{n}}\right)\in F^{n}\Sigma_{A}E^{n}_{(\P^1)^{n}}(\log
B).
\end{displaymath} that represents a generator
\begin{displaymath}
\{\alpha (n)\} \in \Gr^{W}_{0}\rH^{n}((\P^{1})^{n}\setminus B, A; n),
\end{displaymath}
We denote the lifting of $\alpha $ to $X\times
(\P^{1})^{n}$ still by $\alpha$, and we denote $\alpha
_{1}=p_{1}^{\ast}\alpha $ and $\alpha _{2}=p_{2}^{\ast} \alpha $. Thus 
\begin{displaymath}
\alpha_{1}=(-1)^{n}\left(\bigwedge_{i=1}^{n}\frac{dt^{1}_{i}}{t^{1}_{i}}\right),
\end{displaymath}
and
\begin{displaymath}
\alpha_{2}=(-1)^{n}\left(\bigwedge_{j=1}^{n}\frac{dt^{2}_{j}}{t^{2}_{j}}\right).
\end{displaymath}
Here $(t^{1}_{1},\cdots, t^{1}_{n})$ and $(t^{2}_{1},\cdots,
t^{2}_{n})$ denote the affine coordinates of the first and second
factor $(\P^1)^{n}$ respectively.

We next consider the space $\caX_{Z}\times (\P^{1})^{n}$ that fits in
a cartesian diagram
\begin{displaymath}
  \xymatrix{
    \caX_{Z}\times (\P^{1})^{n}\ar[r]^-{\tilde \pi _{Z}}\ar[d]_{\tilde
      p_{1}} &
    X\times (\P^{1})^{n}\times (\P^{1})^{n }\ar[d]^{p_{1}}\\
    \caX_{Z}\ar[r]_-{\pi _{Z}}& X\times (\P^{1})^{n}.
  }
\end{displaymath}
We denote by $\widehat A_{1}$ and $\widehat B_{1}$ the pullbacks of
$\widehat A_{X}$ and $\widehat B_{X}$ by $\tilde p_{1}$. They agree
with the strict transforms of $A_{1}$ and $B_{1}$ by $\widetilde \pi
_{Z}$. Also we still denote by $E_{Z}$ the pullback of $E_{Z}$ by
$\tilde p_{1}$. This
is the exceptional divisor of $\tilde \pi _{Z}$. Finally let $A_{2}$
and $B_{2}$ still denote the pullback of $A_{2}$ and $B_{2}$ by
$\tilde \pi _{Z}$. We have
\begin{displaymath}
  \caX_{Z}\times (\P^{1})^{n}\setminus E_{Z}=X\times
  (\P^{1})^{2n}\setminus  |Z|.
\end{displaymath}
We have classes
\begin{align*}
  p_{1}^{\ast}(\cl(Z))
  &\in \rH^{2p-1}(X\times
    (\P^{1})^{2n}\smallsetminus
    A_{1}\cup |Z|,B_{1}; p),\\
  \{\alpha  _{2}(n)\}
  &\in
    \rH^{n}(X\times
  (\P^{1})^{2n}\smallsetminus
  B_{2},A_{2}; n).
\end{align*}
Applying Proposition \ref{prop:7} to the total space $\caX_{Z}\times
(\P^{1})^{n}$ and the divisors $\widehat A_{1}\cup E_{Z}\cup B_{2}$
and $\widehat B_{1}\cup A_{2}$, one defines a class

\begin{displaymath}
  \nu_{Z}\coloneqq p_{1}^{\ast}(\cl(Z))\cup \{\alpha  _{2}(n)\}
\end{displaymath}
 that satisfies 
\begin{displaymath}
  \nu_{Z}\in
  \overline {F^{0}}\cap F^{0}\Gr^{W}_{0}\rH^{2p+n-1}(X\times
  (\P^{1})^{2n}\smallsetminus
  A_{1}\cup |Z|\cup 
  B_{2},A_{2}\cup B_{1}; p+n).
\end{displaymath}
The following proposition follows directly from the construction of
$\cl(Z)$ and $\{\alpha_{2}\}$ in terms of differential forms, and
Propositions \ref{prop:4-5} and \ref{prop-difformZ}.

\begin{prop}\label{prop:5-3-2} The differential form 
$p^{\ast}_{1}\eta_{Z}(p)\wedge \alpha_{2}(n)$ is a representative of  the
unique lift of $\nu_{Z}$ to
\begin{displaymath}
I^{0,0}\rH^{2p+n-1}(X\times (\P^{1})^{2n}\smallsetminus
  A_{1}\cup |Z|\cup B_{2},A_{2}\cup B_{1};p+n).
 \end{displaymath}
\end{prop}
Further, since the dimension of $W$ is
$d+2n-q=p+n-1$,
\begin{displaymath}
\Gr^{W}_{0}\rH^{\ast}(|W|\smallsetminus A_{1}\cup |Z|\cup B_{2}, A_{2}\cup B_{1};p+n)=0.
\end{displaymath}
Hence, $\nu_{Z}$ can be lifted uniquely to a class 
\begin{equation}\label{eq:5}
\nu_{Z}\in \Gr^{W}_{0}\rH^{2p+n-1}(X\times (\P^{1})^{2n}\smallsetminus
  A_{1}\cup |Z|\cup 
  B_{2},A_{2}\cup |W|\cup B_{1}; p+n).
\end{equation}
As before, the cohomology class
$\{p^{\ast}_{1}\eta_{Z}(p)\wedge\alpha_{2}(n)\}$
belongs to $I^{0,0}$. This allows us to define a morphism of mixed Hodge
structures
\begin{displaymath}
\phi_{Z}\colon \Q(0)\rightarrow \Gr^{W}_{0}\rH^{2p+n-1}(X\times (\P^{1})^{2n}\smallsetminus
  A_{1}\cup |Z|\cup 
  B_{2}, A_{2}\cup |W|\cup B_{1}; p+n).
  \end{displaymath}
Similarly, the class of $W$ defines a class in 
\begin{equation}\label{eq:8}
 \nu_{W}\in \Gr^{W}_{0}\rH^{2q+n-1}(X\times (\P^{1})^{2n}\smallsetminus
  A_{2}\cup |W|\cup 
  B_{1}, A_{1}\cup |Z|\cup B_{2}; q+n),
\end{equation}
with unique lift to $I^{0,0}$ provided by the form
$p^{\ast}_{2}\eta_{W}(q)\wedge\alpha_{1}(n)$.

\subsection{Reaching local product situation by blowing up}
\label{sec:reach-local-prod}
Note that $(2p+n-1)+(2q+n+1)=2(d+2n)=2\dim(X\times
(\P^{1})^{2n})$, and one is tempted to define a framing using
the class \ref{eq:5} together with the dual of the class
\eqref{eq:8}. However, the spaces $A_{1}\cup |Z|\cup B_{2}$ and
$A_{2}\cup |W|\cup B_{1}$ are generally not in local product
situation, and the corresponding mixed Hodge structures are not in
duality. This is precisely the reason we have to resort to blow-ups to
remedy the situation. We move to local product situation using first
Theorem \ref{thm:2} in the Appendix to separate the strict transforms of
$Z$ and $W$ and then Theorem \ref{thm:1} to resolve the singularities
of $Z$ and of $W$. Namely we construct the following diagram of spaces
\begin{displaymath}\label{key-diagram}
  \xymatrix@!C=2em{
    &&&&\caX_{Z}\ar[rd]&\\
    &\widetilde \caX_Z\ar[rr]\ar[rd]^-{\varpi_2 }\ar[rrru]^{\varphi_{Z}}&&\caX_{Z}\times
    (\P^1)^n\ar[ru]\ar[dr]&\otimes &X\times (\P^1)^n\ar[rd]^{\pi_{1}}\\
    \caX_{Z,W}\ar[ru]^-{\varpi _3}\ar[rd]_{\varpi _2'}&\otimes &\caX_{Z\cap
      W}\ar[rr]^-{\varpi _1}& 
    &X\times (\P^1)^{2n}\ar[ru]_{p_1}\ar[rd]^{p_2}&\otimes &X\\
    &\widetilde \caX_W\ar[rr]\ar[ru]_{\varpi _3'}&&\caX_{W}\times
    (\P^1)^n\ar[ru]\ar[rd]&\otimes &X\times (\P^1)^n\ar[ru]_{\pi_{2}}\\
    &&&&\caX_{W}\ar[ru]&
  }
\end{displaymath}
where the spaces $\caX_Z$ and $\caX_W$ are the ones introduced before,
the squares marked with $\otimes $ are Cartesian, in the space
$\caX_{W\cap Z}$ the strict transforms of $Z$ and $W$ are disjoint, in
$\widetilde \caX_Z$ the total transform of $|Z|$ is a simple normal crossing
divisor and in $\widetilde \caX_W$  the total transform of $|W|$ is a
simple normal crossing divisor. More precisely, using
Theorems~\ref{thm:1} and \ref{thm:2} we obtain 
a finite sequence of spaces and birational morphisms
$(\caX_{i},D_{i},C_{i}, \pi _{i}),$ $i=1,\dots ,r$, satisfying 

\begin{enumerate}
\item $\caX_{0}=X\times (\P^{1})^{2n}$, $D_{0}=A_{1}\cup B_{1}\cup
  A_{2}\cup B_{2}$.
  
\item $C_{i}$ is a smooth, irreducible subvariety of $\caX_{i}$ that
  has only normal crossings with 
  $D_{i}$ and is contained in the union of the strict
  transforms of $Z$ and $W$.
  
\item $\pi _{i+1}\colon \caX_{i+1}\to \caX_{i}$ is the blow up of
  $\caX_{i}$ along $C_{i}$ and $D_{i+1}=\pi _{i+1}^{-1}(D_{i})\cup
  E_{i+1}$ with $E_{i+1}$ the exceptional divisor of the blow up. Note
  that $D_{i+1}$ has only normal crossings.
  
\item There is an $r_{0}\le r$ and for each $i< r_{0}$ the center
  $C_{i}$ is contained in the intersection of the strict transforms of
  $Z$ and $W$, moreover the strict
  transforms of $Z$ and $W$ in  $\caX_{r_0}$ are disjoint and
  $\caX_{W\cap Z}=\caX_{r_0}$. So $\varpi _1$ is a sequence  of blow-ups
  on smooth centers contained in the intersection of the strict
  transforms of $Z$ and $W$. 

  \item There is an $r_{1}$ with $r_0\le r_1\le r$ such that, for each
    $r_0\le i< r_{1}$ the center $C_{i}$ is contained in the strict transform of
  $Z$, moreover the total transform 
  of $Z$ in $\caX_{r_1}$ is a simple normal crossings divisor and the
  rational map $\caX_{W\cap Z}\dashrightarrow \caX_Z\times (\P^1)^{n}$
  factors through a map $\caX_{r_1}\to \caX_Z\times (\P^1)^{n}$ and   
$\widetilde {\caX_{Z}}=\caX_{r_1}$. So $\varpi _2$ is a sequence  of blow-ups
  on smooth centers contained in the strict
  transform of $Z$. 

  \item For each
    $r_1\le i< r$ the center $C_{i}$ is contained in the strict transform of
  $W$, moreover the total transform 
  of $W$ in $\caX_{r}$ is a simple normal crossings divisor and the
  rational map $\caX_{W\cap Z}\dashrightarrow \caX_W\times (\P^1)^{n}$
  factors through a map $\caX_{r}\to \caX_W\times (\P^1)^{n}$ and   
$\widetilde {\caX_{W,Z}}=\caX_{r}$. So $\varpi _3$ is a sequence of blow-ups
  on smooth centers contained in the strict
  transform of $W$.   
\end{enumerate}
Note that, since in $\caX_{r_0}$ the strict transforms of $Z$ and $W$
are disjoint, the blow-ups between $r_0$ and $r_1$ and the ones
between $r_1$ and $r$ can be interchanged obtaining the maps
$\varpi _2'$ and $\varpi _3'$. 

In order to check that the differential forms we are working with are
in the right complexes we classify the components of the final normal
crossing divisor $D_r$ as follows. 

\begin{enumerate}
\item We denote by $\widehat A_{i}$ and $\widehat B_{i}$, $i=1,2$ the
strict transforms of $A_{i}$ and $B_{i}$ respectively. 

\item $E_{3}$ denote the union of components whose center $C_{i}$ is
  contained in the total transform of  $A_{1}\cap A_{2}$.
  
\item $E_{i}$, $i=1,2$  the union of components whose center is contained in
  the total transform of $A_{i}$ but not in in the total transform of $A_{1}\cap A_{2}$.
  
\item $F$ is the union of components whose center is contained in the
  intersection of the strict transforms of $Z$ and $W$ but not in  
  the total transform of $A_{1}\cup A_{2}$.
  
\item $G_{1}$ the union of components whose center is contained in the
  strict transform of $Z$ but not in the one of $W$ nor in the total
  transform of  $A_{1}\cup A_{2}$.
  
\item $G_{2}$ the union of components whose center is contained in the
  strict transform of $W$ but not in the one of $Z$ nor in the total
  transform of  $A_{1}\cup A_{2}$.
\end{enumerate}

Note that the
exceptional divisor of the composition $\varpi_{1}\circ \varpi_2\circ \varpi_3$ is $E_{1}\cup E_{2} \cup E_{3}\cup
F\cup G_{1}\cup G_{2}$.
Now we state the main technical result of this section. To make the
notation lighter, we will not distinguish a form from its pullback by
a birational map as they agree in a dense open subset. It will be
clear from the context in which space we are considering the
differential form. 
\begin{thm}\label{thm:5} With the previous notations.
  \begin{enumerate}
  \item \label{item:5} The class $\cl(Z)\in I^{0,0}\rH^{2p-1}(\caX_{Z}\smallsetminus
    \widehat A_X\cup E_Z,\widehat B_X;p) $ defines a
    class
    \begin{displaymath}
      \varphi_{Z}^{\ast}\cl(Z)\in I^{0,0}\rH^{2p-1}(\widetilde {\caX_{Z}}\smallsetminus
    \widehat A_1\cup E_1\cup E_2\cup E_3\cup F\cup G_1, \widehat B_1;p). 
  \end{displaymath}
  
  Moreover $\eta_{Z}$ satisfies  
  \begin{displaymath}
    \phi^{\ast}_{Z}\eta_{Z}(p) \in E^{2p-1}_{\widetilde {\caX_{Z}}}(\si
    \widehat A_1\cup E_1\cup E_2\cup E_3\cup F\cup G_1,\rd \widehat B_1;p)
  \end{displaymath}
  and represents the class $\varphi^{\ast}\cl(Z)$.
\item \label{item:6}The class $\{\alpha _2(n)\}\in I^{0,0}\rH^n (X\times
  (\P^1)^{2n}\setminus A_2,B_2;n)$ can be lifted to a unique class
  \begin{displaymath}
    \varpi _2^\ast\circ \varpi _1^\ast \{\alpha_2(n)\}\in
    I^{0,0}\rH^n(\widetilde {\caX_{Z}}\smallsetminus
    \widehat B_2\cup E_1\cup G_1, \widehat A_2\cup E_2\cup E_3\cup F;n).
  \end{displaymath}
  Moreover the form $\alpha _{2}(n)$ belongs to 
  \begin{displaymath}
    E^n_{\widetilde {\caX_{Z}}}(\si
    \widehat B_2\cup E_1\cup G_1, \rd \widehat A_2\cup E_2\cup E_3\cup F;n)
  \end{displaymath}
  and represents this class.
\item \label{item:7}The product
  $\varphi_{Z}^{\ast}\eta_Z(p)\wedge \alpha _2(n)$
  belongs to
  \begin{displaymath}
    E^{2p+n-1}_{\widetilde {\caX_{Z}}}(\si \widehat A_1\cup
    \widehat B_2\cup E_1\cup G_1   , \rd \widehat A_2\cup
    \widehat B_1\cup E_2\cup E_3\cup F;p+n)
  \end{displaymath}
  and  represents the cup-product class
  $\varphi_{Z}^{\ast}\cl(Z)\cup \varpi _2^\ast\circ \varpi _1^\ast
  \{\alpha_2(n)\}$ in
  \begin{displaymath}
     I^{0,0}\rH^{2p+n-1}(\widetilde {\caX_{Z}}\smallsetminus \widehat A_1\cup
    \widehat B_2\cup E_1\cup G_1   , \widehat A_2\cup
    \widehat B_1\cup E_2\cup E_3\cup F;p+n).
  \end{displaymath}
  In other words, $\varphi_{Z}^{\ast}\eta_Z(p)\wedge \alpha _2(n)$ represents
  the unique lifting of $\varpi _{2}^{\ast}\circ \varpi _{1}^{\ast}\nu
  _{Z}$ to a class in $I^{0,0}$. 
\item \label{item:8} The differential form
 $\varphi_{Z}^{\ast}\eta_Z(p)\wedge \alpha _2(n)$
  belongs to 
  \begin{displaymath}
    E^{2p+n-1}_{\caX_{Z,W}}(\si \widehat A_1\cup
    \widehat B_2\cup E_1\cup G_1   , \rd \widehat A_2\cup
    \widehat B_1\cup E_2\cup G_2\cup E_3\cup F;p+n)
  \end{displaymath}
  and  represents a class in 
  \begin{displaymath}
     I^{0,0}\rH^{2p+n-1}(\caX_{Z,W}\smallsetminus \widehat A_1\cup
    \widehat B_2\cup E_1\cup G_1   , \widehat A_2\cup
    \widehat B_1\cup E_2\cup G_2\cup E_3\cup F;p+n).
  \end{displaymath}
    In other words, $\varphi_{Z}^{\ast}\eta_Z(p)\wedge \alpha _2(n)$ represents
  the unique lifting of $\varpi _{3}^{\ast}\varpi _{2}^{\ast}\circ
  \varpi _{1}^{\ast}\nu _{Z}$ to a class in $I^{0,0}$. 
  \end{enumerate}
\end{thm}

\begin{proof}
  We start by proving \ref{item:5}. Since
  \begin{displaymath}
    \varphi_Z^{-1}(\widehat A_X\cup E_Z)\subset \widehat A_1\cup
    E_1\cup E_2\cup E_3\cup F\cup G_1 
  \end{displaymath}
  and $\varphi_Z(\widehat B_1)\subset \widehat B_X$, the existence of
  the class $\varphi_{Z}^{\ast}\cl(Z)$ is given by Proposition
  \ref{prop:10}. Since $\eta_{Z}(p)\in \Sigma _{\widehat{B}_X}E^{2p-1}_{\caX_{Z}}(\log
  \widehat{A}_X\cup E_{Z};p)$, the same argument as in Remark
  \ref{rem:6} shows that
  \begin{displaymath}
    \varphi_{Z}^\ast\eta_Z(p) \in E^{2p-1}_{\widetilde {\caX_{Z}}}(\si
    \widehat A_1\cup E_1\cup E_2\cup E_3\cup F\cup G_1,\rd \widehat B_1;p)
  \end{displaymath}
  and represents the class $\varphi_Z^\ast\cl(Z)$. The second
  statement of \ref{item:5} follows from Proposition \ref{prop:11}. 
  
  To prove \ref{item:6}, the idea is to lift the form $\alpha _{2}(n)$
  and its class $\{\alpha _2(n)\}$ by
  induction.
First we show that it can be lifted to the group
   \begin{displaymath}
     \Gr_{0}^{W}\rH^{n}(\caX_{j}\smallsetminus \widehat B_{2}\cup
     E_{1}, \widehat A_{2}\cup E_{2}\cup E_{3}\cup F; n)
   \end{displaymath}
   for $j=0,\dots, r_{0}$, and the lift is given in terms of
   differential form by the corresponding pullback of
   $\alpha_{2}(n)$. For $j=0$ it is clear because in $\caX_{0}$ 
   there are no exceptional divisors. Assume that the result is true
   for $\caX_{j}$. Let $C_{j}$ be the center of the blow up $\pi
   _{j+1}$.

   If $C_{j}$ is contained in the total transform of $A_{2}$,
   which agrees  with $\widehat A_{2}\cup E_{2}\cup E_{3}$, then the
   exceptional divisor of this blow up will be part of either $E_{2}$
   or $E_{3}$ in the next step. Then the lifting is given by Proposition
   \ref{prop:6}, which also assures that the lift is given by the
   pullback of the differential form $\alpha_{2}(n)$
   through $\pi_{j+1}$.

   If $C_{j}$ is contained in the total transform of
   $A_{1}$ but not on the total transform of $A_{2}$, then the
   exceptional divisor of this blow up will be a component of the new
   $E_{1}$. Therefore the lifting in this case is just functoriality,
   automatically ensuring that it is given by the corresponding
   pullback of $\alpha_{2}(n)$. 

   Now we can assume that $C_{j}$ is not contained in the total
   transform of $A_{1}\cup A_{2}$. If furthermore $C_{j}$ is contained
   in $F$, again we can use Proposition \ref{prop:6}.

   The final case is when $C_{j}$ is not contained in $F$ nor in the
   total transform of $A_{1}\cup A_{2}$. By the hypothesis of proper
   intersection we know that, in $\caX_{j}\smallsetminus \widehat
   A_{1}\cup \widehat A_{2}\cup E_{1}\cup E_{2}\cup E_{3}\cup F$, the
   subscheme $\widehat Z\cap \widehat W$ meets properly all the faces
   of $\widehat B_{1}\cup \widehat B_{2}$. Therefore, if $s$ is the
   number of components of $\widehat B_{2}$ containing $C_{j}$ we
   have
   \begin{displaymath}
     \dim C_{j} + s \le \dim (Z\cap W\setminus (A_{1}\cup A_{2}))\le n-1. 
   \end{displaymath}
   In this case the lifting is provided by Proposition
   \ref{prop:2}. Also, in this 
   case, the lifting is given by the
   pullback of $\alpha_{2}(n)$ by $\pi_{j+1}$.

   We next show that it can be lifted to the group
   \begin{displaymath}
     \Gr_{0}^{W}\rH^{n}(\caX_{j}\smallsetminus \widehat B_{2}\cup
     E_{1}\cup G_1, \widehat A_{2}\cup E_{2}\cup E_{3}\cup F; n)
   \end{displaymath}
   for $j=r_0,\dots, r_{1}$, and the lift is given in terms of
   differential form by the corresponding pullback of
   $\alpha_{2}(n)$. For $j=r_0$ we have already done it.  So, again, we
   assume that the result is true 
   for $\caX_{j}$ and let $C_{j}$ be the center of the blow up $\pi
   _{j+1}$. If $C_j$  is contained in the total transform of $A_1\cup
   A_2$ or in $F$ we argue as in the previous case. The remaining case is when
   $C_j$ is contained in $G_1$ or is contained in the strict transform of $Z$ but not in the
   total transform of $A_1\cup A_2$ nor in $F$. Since now the strict transforms of
   $Z$ and $W$ are disjoint the new exceptional divisor for the blow
   up will belong to the new $G_1$  and we only need to apply the
   functoriality to conclude the proof of \ref{item:6}.

   To prove \ref{item:7} we invoke Proposition \ref{prop:7} and
   Proposition \ref{prop:9} that
   imply that $\varphi_{Z}^{\ast}\eta_Z(p)\wedge \alpha _2(n)$
  belongs to
  \begin{displaymath}
    E^{2p+n-1}_{\widetilde {\caX_{Z}}}(\si \widehat A_1\cup
    \widehat B_2\cup E_1\cup G_1   , \rd \widehat A_2\cup
    \widehat B_1\cup E_2\cup E_3\cup F)
  \end{displaymath}
  and  represents the cup-product class
  $\varphi_{Z}^{\ast}\cl(Z)\cup \varpi _2^\ast\circ \varpi _1^\ast
  \{\alpha_2(n)\}$ in
  \begin{displaymath}
     I^{0,0}\rH^{2p+n-1}(\widetilde {\caX_{Z}}\smallsetminus \widehat A_1\cup
    \widehat B_2\cup E_1\cup G_1   , \widehat A_2\cup
    \widehat B_1\cup E_2\cup E_3\cup F;p+n),
  \end{displaymath}

  Finally we prove \ref{item:8}. The method is the same as the proof
  of  \ref{item:6}. We show by induction that the class
  $\varphi_{Z}^{\ast}\cl(Z)\wedge \{\alpha _2(n)\}$
  can be lifted to the group
   \begin{displaymath}
     \Gr_{0}^{W}\rH^{2p+n-1}(\caX_{j}\smallsetminus \widehat A_{1}\cup
     \widehat B_{2}\cup
     E_{1}\cup G_1, \widehat A_{2}\cup \widehat B_1\cup E_{2}\cup G_2\cup
     E_{3}\cup F; n)
   \end{displaymath}
   for $j=r_1,\dots, r$, and the lift is given in terms of
   differential form by the corresponding pullback of
   $\varphi^\ast_Z\eta_Z(p)\wedge \alpha_{2}(n)$. For $j=r_1$ it is the
   result of the previous step.  Assume that the result is true
   for $\caX_{j}$. Let $C_{j}$ be the center of the blow up $\pi
   _{j+1}$.
   If $C_{j}$ is contained in the
total transform of $A_{1}\cup A_{2}$ or in $F\cup G_1\cup G_2$ we can argue as in the
previous cases. The only new case is when 
 $C_{j}$ is
contained in the strict transform $\widehat W$ of $W$ but not in the total transform of $A_1\cup
A_2$, nor in any of the previous exceptional divisors. Then,  since we
know that $\widehat W$ meets $\widehat B_{2}$ 
properly outside the total transform of $A_1\cup
A_2$ and the previous exceptional divisors, if $s$ be the number of
components of $\widehat B_{2}$ that contain 
$C_{j}$, then 
\begin{displaymath}
  \dim C_{j} + s \le \dim W = p+n-1.
\end{displaymath}
Therefore the lifting of the class is obtained using Proposition \ref{prop:2}, and
by Proposition \ref{prop:8} it is given in terms of differential forms by the
corresponding pullback of $\varphi^\ast_Z\eta_Z(p)\wedge \alpha_{2}(n)$.
This completes the proof.
\end{proof}

\subsection{The height pairing of higher cycles}
\label{sec:height-pair-high}
Once we have a space where the relevant subvarieties are in local
product situation we can now define a mixed Hodge structure and use
the classes of $Z$ and $W$ to define a frame on it, and consequently define the heights $\Ht_{1}$ and $\Ht_{2}$.
\begin{df}\label{df-5-9}
  We denote by $H_{Z,W}$ the mixed Hodge structure
  \begin{displaymath}
    H_{Z,W}\coloneqq  \rH^{2p+n-1}(\caX_{Z,W}\smallsetminus \widehat A_1\cup
    \widehat B_2\cup E_1\cup G_1   , \widehat A_2\cup
    \widehat B_1\cup E_2\cup G_2\cup E_3\cup F;p+n)
  \end{displaymath}
  and we denote by
  \begin{displaymath}
    \cl(Z) \in I^{0,0}_{H_{Z,W}} 
  \end{displaymath}
  the class represented by the differential form
    $\varphi_{Z}^{\ast}\eta_Z(p)\wedge \alpha _2(n)$.
  \end{df}

  \begin{rmk}\label{rmk-5-10}
    There is an abuse of notation in the definition of $H_{Z,W}$
    because it depends on the choices  made during the resolution of
    singularities. Nevertheless for the purposes of defining the
    height this choice will be irrelevant because if we have two
    different choices of spaces $\caX_{Z,W,1}$ and $\caX_{Z,W,2}$
    satisfying the conditions needed in Theorem \ref{thm:5} giving rise
    to two different mixed Hodge structures $H_{Z,W,1}$,
    $H_{Z,W,2}$ and classes $\cl(Z)_{1}$, $\cl(Z)_{2}$ respectively,
    then  there is always a third possible choice of resolution of
    singularities $\caX_{Z,W,3}$ with maps
    \begin{displaymath}
      \xymatrix{
        &\caX_{Z,W,3}\ar[dl] \ar[dr] & \\
        \caX_{Z,W,1} && \caX_{Z,W,2}
      }
    \end{displaymath}
    that gives rise to a mixed Hodge structure $H_{Z,W,3}$ with maps
    \begin{displaymath}
      \xymatrix{
        &H_{Z,W,3}& \\
        H_{Z,W,1}\ar[ur] && H_{Z,W,2}\ar[ul]
      }
    \end{displaymath}
    sending the classes $\cl(Z)_{1}$ and $\cl(Z)_{2}$ to the
    corresponding class $\cl(Z)_{3}$.
  \end{rmk}

  The construction of the space $\caX_{Z,W}$ is symmetric in $Z$ and
  $W$ so we can write
  \begin{displaymath}
    H_{W,Z}\coloneqq \rH^{2q+n-1}
    (\caX_{Z,W}\smallsetminus \widehat A_2\cup
    \widehat B_1\cup E_2\cup G_2   , \widehat A_1\cup
    \widehat B_2\cup E_1\cup G_1\cup E_3\cup F;q +n).
  \end{displaymath}
 and apply Theorem \ref{thm:5} to the cycle $W$ to obtain a
  class
  \begin{displaymath}
    \cl(W)\in H_{W,Z}.
  \end{displaymath}

  \begin{rmk}\label{rem:9}
    The mixed Hodge structures $H_{Z,W}$ and $H_{W,Z}$ are not in
    general in duality, even after a twist, because the divisors $E_3$  and $F$ are both on
    the same side. By \cite[Lemma 1.13]{BGGP:Height} we can move
    $E_{3}$ from one side to the other in one of the mixed Hodge
    structures to solve part of the issue. But since $W$ and $Z$ are
    not in local product situation to start with, \cite[Lemma
    1.13]{BGGP:Height} does not apply to $F$.
  \end{rmk}

  In spite of Remark \ref{rem:9} we have the following result that is
  enough for our needs.
  \begin{prop}\label{prop:4} There is a map
    \begin{displaymath}
      \sigma_{W,Z}\colon H_{W,Z}(-n-1)\longrightarrow H_{Z,W}^{\vee},
    \end{displaymath}
    such that, if
    \begin{displaymath}
      \omega \in
          E^{2q+n-1}_{\caX_{Z,W}}(\si \widehat A_2\cup
    \widehat B_1\cup E_2\cup G_2   , \rd \widehat A_1\cup
    \widehat B_2\cup E_1\cup G_1\cup E_3\cup F)
    \end{displaymath}
and 
    \begin{displaymath}
      \eta \in
          E^{2p+n-1}_{\caX_{Z,W}}(\si \widehat A_1\cup
    \widehat B_2\cup E_1\cup G_1   , \rd \widehat A_2\cup
    \widehat B_1\cup E_2\cup G_2\cup E_3\cup F)
  \end{displaymath}
  then
  \begin{equation}\label{eq:50}
    \sigma _{W,Z}\big(\omega (q-1)\big)\Big(\eta(p+n)\Big)=
    \frac{1}{(2\pi i)^{d+2n}}\int_{X\times (\P^{1})^{2n}}\omega \wedge \eta
  \end{equation}
  \end{prop}
  \begin{proof}
    We first check that the degrees and the twists are compatible with
    the dimension. On the
    one hand
    \begin{displaymath}
      2q+n-1+2p+n-1=2d+4n=2\dim \caX_{Z,W}.
    \end{displaymath}
    On the other hand
    \begin{displaymath}
      q+n -n-1 +p+n=p+q+n-1=d+2n= \dim \caX_{Z,W}.
    \end{displaymath}
    In both cases we have used the condition $p+q=d+n+1$.
    For shorthand  we write
    \begin{align*}
      D_{1}&=\widehat A_1\cup
    \widehat B_2\cup E_1\cup G_1,\\
      D_{2}&=\widehat A_2\cup
             \widehat B_1\cup E_2\cup G_2,\\
      D_{3}&=E_{3}\cup F.
    \end{align*}
    So $H_{Z,W}=\rH^{2p+n-1}(\caX_{Z,W}\smallsetminus D_1  , D_2\cup
    D_3;p+n)$
    since $D_{1}\cup D_{2}\cup D_{3}$ is a normal crossing divisor,
    $\caX_{Z,W}$  is a smooth projective variety of dimension 
    $d+2n$ and  we
    have the condition 
    \begin{math}
      p+q=d+n+1.
    \end{math}
    Hence we deduce
    \begin{displaymath}
      H_{Z,W}^{\vee}=\rH^{2q+n-1}(\caX_{Z,W}\smallsetminus D_2 \cup
    D_3 , D_1;q+n-(n+1)),
  \end{displaymath}
  while
      \begin{displaymath}
      H_{W,Z}(-n-1)=\rH^{2q+n-1}(\caX_{Z,W}\smallsetminus D_2  , D_1\cup
    D_3;q+n-(n+1)).
  \end{displaymath}
  The sought map is given by the functoriality of Proposition
  \ref{prop:10} for the identity map with divisors $A=D_{2}$,
  $B=D_{1}\cup D_{3}$, $A'=D_{1}\cup D_{3}$ and $B'=D_{1}$. Since the
  map satisfies the extra condition of
  Proposition \ref{prop:4-4'} we know that this map can be represented
  using slowly increasing/rapidly decreasing 
    differential forms. Formula \eqref{eq:50} follows from the
    explicit description of the trace map in \ref{exm:12}.
  \end{proof}

  \begin{rmk}\label{rem:10}
    One has to be careful with the sign of the morphism in
    Proposition \ref{prop:4} as it involves duality and differential
    forms. In fact, the identification of a cohomology group
    with its bidual can be the identity or minus the identity
    depending on the parity. 
  \end{rmk}
  One can obtain a similar map
  \begin{displaymath}
   \sigma_{Z,W}\colon H_{Z,W}(-n-1)\longrightarrow H_{W,Z}^{\vee}.
  \end{displaymath}
  The following proposition is a concrete example of Remark \ref{rem:10}.
  \begin{prop}\label{prop:12}
  The map of Proposition \ref{prop:4} satisfies  
  \begin{displaymath}
    \sigma^{\vee}_{W,Z}\otimes \Id_{\Q(-n-1)}=(-1)^{n-1}\sigma_{Z,W}.
  \end{displaymath}
  \end{prop}
  \begin{proof}
    We use the notations in the proof of Proposition \ref{prop:4}.
    Let
    \begin{align*}
      \eta &\in E^{2p+n-1}_{\caX_{Z,W}}(\si D_{1}, \rd D_{2}\cup
              D_{3})\\
      \omega &\in E^{2q+n-1}_{\caX_{Z,W}}(\si D_{2}, \rd
  D_{1}\cup D_{3})
    \end{align*}
    be closed forms representing classes $\{\eta(p+n) \}\in H_{Z,W}$
    and $\{\omega (q+n)\}\in H_{W,Z}$ respectively. Then
    \begin{displaymath}
      \sigma _{Z,W}\big(\eta(p-1)\big)(\omega (q+n)
      )=\frac{1}{(2\pi i)^{d+2n}}\int_{\caX_{Z,W}}\eta \wedge \omega.
    \end{displaymath}
    While
    \begin{align*}
      (\sigma^{\vee} _{W,Z}\otimes \Id_{\Q(-n-1)})\big(\eta(p-1)\big)(\omega (q+n)
      )  &=
      \sigma^{\vee} _{W,Z}(\eta(p+n) )(\omega(q-1))\\
      &=
      \sigma _{W,Z}\big(\omega (q-1)\big)(\eta(p+n) )\\
        &=\frac{1}{(2\pi i)^{d+2n}}
      \int_{\caX_{Z,W}}\omega \wedge \eta. 
    \end{align*}
 The proposition follows from the identity
 \begin{displaymath}
   (-1)^{(2p+n-1)(2q+n-1)}=(-1)^{n-1}.
 \end{displaymath}
\end{proof}

\begin{df}\label{def:13} Following the notation of Remark \ref{rem:11}
  we define a $(0,-n-1)$-framed mixed Hodge structure that we still
  denote by $H_{Z,W}$ as
  \begin{displaymath}
    H_{Z,W}=\big(H_{Z,W},\cl(Z), \sigma _{W,Z}(\cl(W)\otimes \bfone(-n-1)_{\Betti})\big).
  \end{displaymath}
In the notation of Definition \ref{framed-mhs}, 
the class $\cl(Z)$ defines a morphism 
  \begin{equation}\label{eq:47}
   \phi _{Z,W}\colon \Q(0)\to \Gr^{W}_{0}H_{Z,W}
  \end{equation}
  that sends $\bfone(0)_{\Betti}$ to $\cl(Z)$. Similarly, the class
$\sigma _{W,Z}(\cl(W)\otimes \bfone(-n-1)_{\Betti})$ gives a morphism
  \begin{equation}
    \label{eq:49}
    \psi _{Z,W}\colon \Gr^{W}_{-2n-2}H_{Z,W}\longrightarrow \Q(n+1).
  \end{equation}
  Then
  \begin{displaymath}
    H_{Z,W}=(H_{Z,W},\phi _{Z,W},\psi _{Z,W}).
  \end{displaymath}
\end{df}

    \begin{df}\label{def:12} 
    The height pairing $\Ht_1$ and $\Ht_{2}$ between $Z$ and $W$ are
    defined as the height of the framed mixed Hodge structure
    $H_{Z,W}$:
    \begin{displaymath}
      \Ht_{i}(Z,W)\coloneqq \Ht_{i}(H_{Z,W}),\qquad i=1,2.
    \end{displaymath}
  \end{df}
  By Remark \ref{rmk-5-10}, and Proposition \ref{height-ind-morphism}
  both heights are independent of the chosen resolution of singularities.
 \begin{rmk}\label{rmk:5-11}
 Following Remark \ref{rem:14}, we can similarly define $\Ht_{1,\dR}(Z,W)$ and $\Ht_{2,\dR}$ associated to the mixed Hodge structure $H_{Z,W}$.
 \end{rmk}

Although  $H_{Z,W}$ and
$H_{W,Z}$ are not in duality, we can still
compare and obtain the following symmetry property.
\begin{prop}\label{5-5-dual}
For $i=1,2$, we have $\Ht_{i}(W,Z)=(-1)^{n}\Ht_{i}(Z,W)$.
\end{prop}
\begin{proof}
  The equality we want to prove follows from the following identities
  \begin{align*}
    \Ht_{i}(Z,W)
    &=\Ht_{i}(H_{Z,W},\cl(Z),\sigma _{Z,W}(\cl(W)\otimes \bfone(-n-1)_{\Betti}))\\
    &=-\Ht_{i}(H^{\vee}_{Z,W},\sigma _{Z,W}(\cl(W)\otimes
      \bfone(-n-1)_{\Betti}),\cl(Z))\\
        &=-\Ht_{i}(H_{W,Z}(-n-1),\cl(W)\otimes
          \bfone(-n-1)_{\Betti},\sigma _{W,Z}^{\vee}(\cl(Z)))\\
        &=-\Ht_{i}(H_{W,Z},\cl(W),\sigma _{W,Z}^{\vee}(\cl(Z))\otimes
          \bfone(-n-1)_{\Betti})\\
        &=(-1)^{n}\Ht_{i}(H_{W,Z},\cl(W),\sigma _{Z,W}(\cl(Z)\otimes
          \bfone(-n-1)_{\Betti}))\\
    &=(-1)^{n}\Ht_i(W,Z),
  \end{align*}
  where the first and last equalities are the definition of the height
  pairing, the second follows from  Proposition \ref{prop-dual-ht-2},
  the third from Proposition \ref{height-ind-morphism}, the fourth one
  from Lemma \ref{lem-p} and the fifth one from Proposition \ref{prop:12}.
\end{proof}
Another interesting property is the relation of these heights with
respect to complex conjugation.
\begin{prop}\label{conj-ht-cycle}
Let $X$ be a smooth complex projective variety of dimension
$d$. Suppose $Z\in Z^{p}(X,n)_{00}$ and $W\in
Z^{q}(X,n)_{00}$ are two refined cycles intersecting properly,
such that $p+q=d+n+1$. Then $\overline Z$ and $\overline W$ are
refined cycles in the variety $\overline X$ and the equation 
\begin{equation}\label{eq:16}
\Ht_{i}(\overline{Z},\overline{W})=(-1)^{n}\Ht_{i}(Z,W)
\end{equation}
holds for $i=1,2$.
In consequence, if $X$, $Z$ and $W$ are defined over $\R$ and $n$ is
odd, then $\Ht_{i}(Z,W)=0$. 
\end{prop}
\begin{proof}
  The fact that $\overline{Z}$ and $\overline{W}$ are refined cycles
  over $\overline X$ is immediate from the fact that the complex
  conjugation commutes with the pullback morphisms $\delta^{j}_{i}$
  defining the higher cycles. By Example \ref{exm:6} and retracing the
  methods that went into proving Theorem \ref{thm:5} for $(\overline
  X, \overline Z, \overline W)$, we see that there is an isomorphism
  of MHS $H_{\overline
  Z,\overline W}\simeq \overline{H_{Z,W}}$. Thanks to the isomorphism
of Example \ref{exm:8},
this isomorphism can be upgraded to an isomorphism of framed MHS.
Equation \eqref{eq:16} follows now from Proposition
\ref{height-ind-morphism} and Proposition \ref{conj-ht} (with $a=0$
and $b=-n-1$). The last statement follows from the fact that if $X$,
$Z$ and $W$ 
are all defined over $\R$, then $H_{Z,W}=H_{\overline Z,\overline W}$. 
\end{proof}

Our next task is to write down these heights explicitly in terms of
  the differential forms of Proposition \ref{prop-difformZ}, and prove a vanishing result. While $\Ht_{2}$ involves the knowledge of the splitting $\delta_{H_{Z,W}}$, the computation of $\Ht_{1}$ only involves the complex
  conjugation that can be written in terms of differential
  forms. Thus, on the remaining of this paper,  we will focus only on $\Ht_{1}$.

  \begin{prop}\label{prop:5-13}
    The height pairing is given explicitly by
    \begin{displaymath}
      \Ht_{1}(Z,W)=\im  \left( \frac{(-1)^{p+n}}{(2\pi i)^{p+q+2n}}\int_{X\times
          (\P^{1})^{2n}} p_{2}^{\ast}\eta_{W}\wedge \alpha _{1}\wedge
        p_{1}^{\ast}\bar \eta_{Z} \wedge \bar \alpha _{2}\right).
    \end{displaymath}
 \end{prop}
  \begin{proof}
    Let $D_{1}$, $D_{2}$ and $D_{3}$ be the normal crossing divisors
    of $\caX_{Z,W}$ defined in the proof of Proposition \ref{prop:4}. 
    The class $e_{H_{Z,W}}\in H_{Z,W}$ is the unique lifting of
    $\cl(Z)\in \Gr_{0}^{W}H_{Z,W}$ to a class in $I^{0,0}$. By Theorem
    \ref{thm:5}, 
    the differential form
    \begin{displaymath}
      p_{1}^{\ast}\eta_{Z} \wedge \alpha _{2}\otimes \bfone(p+n)_{\dR}\in
      E_{\caX_{Z,W}}(\si D_{1},\rd D_{2}\cup D_{3};p+n)
    \end{displaymath}
    represents the class $e_{H_{Z,W}}$.
    Thus $\overline {e_{H_{Z,W}}}$ is represented by
    $(-1)^{p+n}p_{1}^{\ast}\bar \eta_{Z} \wedge \bar \alpha_{2}\otimes
    \bfone(p+n)_{\dR}$.
    Similarly $e_{H_{W,Z}}$ is represented by the form
    \begin{displaymath}
      p_{2}^{\ast}\eta_{W} \wedge \alpha_{1}\otimes
    \bfone(q+n)_{\dR}\in
      E_{\caX_{Z,W}}(\si D_{2},\rd D_{1}\cup D_{3};p+n)
    \end{displaymath}
    The class $e_{H_{W,Z}(-n-1)}$ is represented by the form
    $p_{2}^{\ast}\eta_{W} \wedge \alpha _{1}\otimes
    \bfone(q+n)_{\dR}\otimes \bfone(-n-1)_{\Betti}$. This class is
    sent to $e_{H_{Z,W}^{\vee}}$ under the map of Proposition
    \ref{prop:4}. For shorthand we write $\omega =p_{2}^{\ast}\eta_{W}
    \wedge \alpha _{1}\wedge
      p_{1}^{\ast}\bar \eta_{Z} \wedge \bar \alpha _{2}$, and compute  
    \begin{align*}
      \Ht_{1}(H_{Z,W})
      &=\im\left(\langle e_{H_{Z,W}^{\vee}},\overline{e_{H_{Z,W}}}\rangle\right)\\
      &=\im\left(\Tr((-1)^{p+n}\omega \otimes
      \bfone(p+q+2n)_{\dR}\otimes \bfone(-n-1)_{\Betti})\right)\\
      &=\im\left(\frac{(-1)^{p+n}}{(2\pi i)^{p+q+2n}}\int_{\caX_{Z,W}}
        \omega \right)\\
      &=\im\left(\frac{(-1)^{p+n}}{(2\pi i)^{p+q+2n}}\int_{X\times (\P^{1})^{2n}}
        \omega\right),
    \end{align*}
where $\Tr\colon \rH^{2d+4n}(\caX_{Z,W};\R(p+2n))\to \R(0)$ is the
integration morphism of Example \ref{exm:12}. In the last equality we use that $\caX_{Z,W}$
and $X\times (\P^{1})^{2n}$ are birational.
\end{proof}
\begin{rmk}\label{rmk:5-12}
For ease of notation and wherever it is clearly understood, we will omit the pullback notations from the differential forms. For example, with this notational liberty the height formula above reads
\begin{displaymath}
      \Ht_{1}(Z,W)=\im  \left( \frac{(-1)^{p+n}}{(2\pi i)^{p+q+2n}}\int_{X\times
          (\P^{1})^{2n}}\eta_{W}\wedge \alpha _{1}\wedge
        \bar \eta_{Z} \wedge \bar \alpha _{2}\right).
    \end{displaymath}
\end{rmk}

We end this section by proving a vanishing result.
\begin{prop}\label{prop:5-14}
Let $X$ be a complex smooth and projective variety of dimension $d$, and $Z\in Z^{p}(X,n)_{00}$, $W\in Z^{q}(X,n)_{00}$ be higher cycles intersecting properly and satisfying $p+q=d+n+1$. If $p,q>d$, then $\Ht_{1}(Z,W)=0$.
\end{prop}
\begin{proof}
  Recall that $(t^{1}_{i})_{i=1,\dots,n}$ are coordinates of the first set of
  $\P^{1}$'s, while $(t^{2}_{i})_{i=1,\dots,n}$ are coordinates of the
  second set. Note that only $\alpha _{1}$ and $\eta_{Z}$ may contain
  the coordinates $t_{i}^{1}$, while only $\alpha _{2}$ and $\eta_{W}$
  may contain the coordinates $t_{i}^{2}$. Thus we can write
  \begin{displaymath}
    \Ht_{1}(Z,W)=\im  \left( \frac{(-1)^{p+n+1}}{(2\pi i)^{p+q+2n}}
      \int_{X} \int_{ (\P^{1})^{n}}\bar \eta_{Z} \wedge \alpha _{1}\wedge
\int_{(\P^{1})^{n}}
        \eta_{W}\wedge 
         \bar \alpha _{2}\right).
     \end{displaymath}
We focus on the push-forward currents on $X$ given by 
\begin{displaymath}
  \beta_{Z}=\int_{ (\P^{1})^{n}}\eta_{Z} \wedge \bar \alpha _{1}\quad
    \text{ and }\quad  \beta_{W}=\int_{ (\P^{1})^{n}}\eta_{W} \wedge
    \bar  \alpha _{2}. 
\end{displaymath}
So that
\begin{displaymath}
   \Ht_{1}(Z,W)=\im  \left( \frac{(-1)^{p+n+1}}{(2\pi i)^{p+q+2n}}
      \int_{X} \bar \beta _{Z}\wedge \beta _{W}\right).
  \end{displaymath}
  First we observe that the condition $p>d$ implies that $\beta _{Z}$
  is closed. Indeed, by \eqref{eq:54} and the fact that $\eta_{Z}$
  vanishes in $B_{X}$ we have
  \begin{displaymath}
    d \beta _{Z}=(2\pi i)^{p}(\pi_{Z})_{\ast}(\bar \alpha _{1}),
  \end{displaymath}
  where $\pi _{Z}\colon Z\to X$ is the composition $Z\to X\times
  (\P^{1})^{n}\to X$. But $p>d$ implies $\dim Z <n$. Thus $\alpha
  _{1}|_{Z}=0$ and $d\beta _{Z}=0$.
  Since $\beta _{Z}$ and $\beta _{W}$ are both closed, we can
change them by boundaries without changing the integral.

Using again that $p>d$, each component of $\eta_{Z}$ must contain at
least one $dt^{1}_{i}$. Therefore 
\begin{displaymath}
  \int_{ (\P^{1})^{n}}\bar \eta_{Z} \wedge \bar \alpha _{1}=0.
\end{displaymath}
Hence
\begin{displaymath}
  \beta_{Z}=\int_{ (\P^{1})^{n}}(\eta_{Z}+(-1)^{p-1}\bar \eta_{Z})
  \wedge \bar \alpha _{1}=
  \int_{ (\P^{1})^{n}}(2dg_{Z}+2\theta _{Z})
  \wedge \bar \alpha _{1}.
\end{displaymath}
Since $g_{Z}$ has no residue and vanishes on the support of the
residue of $\alpha _{1}$ we deduce that
\begin{displaymath}
  \beta_{Z}=\int_{ (\P^{1})^{n}}2\theta _{Z}
  \wedge \bar \alpha _{1}+\text{boundaries}.
\end{displaymath}
Since
\begin{displaymath}
  d \log |t^{1}_{i}|^{2}=\frac{dt^{1}_{i}}{t^{1}_i}+
  \frac{d\overline{t^{1}_{i}}}{\overline{t^{1}_i}}, 
\end{displaymath}
we can find a $\gamma \in \Sigma_{A}E^{\ast}_{(\P^{1})^{n}}(\log B)$
with $d \gamma =\bar \alpha _{1}-(-1)^{n}\alpha _{1}$, such that the associated
current satisfies
\begin{displaymath}
  d [\gamma] =[\bar \alpha _{1}]-(-1)^{n}[\alpha _{1}]+\chi
\end{displaymath}
with $\supp \chi\subset B_{X}$. Since $\theta _{Z}$ is closed and its
restriction to $B_{X}$ is zero, we deduce that
\begin{displaymath}
  d[\theta _{Z}\wedge \gamma ]=[\theta _{Z}\wedge \bar \alpha
  _{1}-(-1)^{n}\theta _{Z}\wedge \alpha _{1}]. 
\end{displaymath}
This implies that 
\begin{displaymath}
  \beta_{Z}=\int_{ (\P^{1})^{n}}\theta _{Z}
  \wedge (\bar \alpha _{1}+(-1)^{n}\alpha _{1})+\text{boundaries}.
\end{displaymath}
Since $\bar \theta _{Z}=(-1)^{p-1}\theta _{Z}$, we conclude that
\begin{displaymath}
  \bar \beta _{Z}=(-1)^{p+n-1}\beta _{Z}+\text{boundaries}.
\end{displaymath}
 We deduce
\begin{displaymath}
  \overline {\frac{(-1)^{p+n+1}}{(2\pi i)^{p+q+2n}}
    \int_{X} \bar \beta _{Z}\wedge \beta _{W}}=
  \frac{(-1)^{p+n+1}}{(2\pi i)^{p+q+2n}}
    \int_{X} \bar \beta _{Z}\wedge \beta _{W},
  \end{displaymath}
  so it is a real number and $\Ht_{1}(Z,W)=0$.
\end{proof}

\begin{rmk}\label{rmk-ht-comp}
Note that the condition $p+q=d+n+1$, together with $p,q\le d+n$ rules out the cases
$p=0$ or $q=0$. Hence if $\dim(X)=0$ there is no example of non-zero
$\Ht_{1}$. If $\dim(X)=1$, by Proposition \ref{prop:5-14} the only possibility to get a non-zero height is when $p=1$ and $q=n+1$ or the other way around. When $d=p=q=2$ and $n=1$, we have an example of a non-trivial height in
\cite[\S 5.2-5.4]{BGGP:Height}.
\end{rmk}

\appendix
\section{Resolution of singularities}
\label{sec:resol-sing}

For convenience of the reader, we recall here a precise form of
Hironaka's resolution of singularities.

\begin{df}\label{def:6} Let $X$ be a regular variety of dimension $d$ over a field $K$ of
  characteristic zero.  Let $D$ be a codimension one subvariety  of $X$ and $W$ a
  reduced closed subscheme of $X$. We say that $W$ \emph{has only
    normal crossings with}
  $D$ \emph{at a point} $x$ (or that $D$ has only normal crossings
  with $W$ at $x$) if there is a regular sequence
  $(z_{1},\dots,z_{d})$ in $\caO_{X,x}$   such that the ideal in $\caO_{X,x}$ of each
  component of $D$ containing $x$ is generated by one of the $z_{i}$,
  while the ideal of $W$ in $\caO_{X,x}$ is generated by some of the
  $z_{i}$. We say that $W$ \emph{has only normal crossings with} $D$
  if it has normal crossings with $D$ at every point of $X$. When $W$
  is empty we just say that $D$ has \emph{only normal crossings}.
\end{df}

We will need the following two versions of Hironaka's resolution of
singularities. The first is a version of \cite[Theorem III(N,n)
]{Hironaka:rs}, while the second follows from Corollary 1 of
\cite[Theorem II]{Hironaka:rs}.

\begin{thm}\label{thm:1}
  Let $X$ be a regular variety over a field $K$ of characteristic
  zero. Let $D$ be a divisor of $X$ with only normal crossings and
  $W\subset X$ a reduced closed subscheme. Then there exists a sequence
  $(X_{i},W_{i},D_{i},C_{i})$, $i=0,\dots ,r$ such that
  \begin{enumerate}
  \item $X_{0}=X$, $W_{0}=W$, $D_{0}=D$, $C_{r}=\emptyset$;
  \item each $C_{i}$, $i=0,\dots,r-1$ is a smooth irreducible
    subvariety of $X_{i}$, has only normal
    crossings with $D_{i}$ and is contained in $W_{i}$;
  \item $\pi _{i+1}\colon X_{i+1}\to X_{i}$ is the blow up of $X_{i}$
    along $C_{i}$, $D_{i+1}=\pi _{i+1}^{-1}(D_{i}\cup C_{i})$ is a
    divisor with only normal crossings, and $W_{i+1}$ is the strict
    transform of $W_{i}$ along $\pi _{i+1}$;
  \item  $W_{r}$ is smooth and has only normal crossings with $D_{r}$.
  \end{enumerate}
\end{thm}

\begin{thm}\label{thm:2}
  Let $X$ be a regular variety over a field $K$ of characteristic
  zero. Let $D$ be a divisor of $X$ with only normal crossings and
  $Z, W\subset X$  subvarieties. Then there exists a sequence
  $(X_{i},Z_i,W_{i},D_{i},C_{i})$, $i=0,\dots ,r$ such that
  \begin{enumerate}
  \item $X_{0}=X$, $Z_{0}=Z$, $W_{0}=W$, $D_{0}=D$, $C_{r}=\emptyset$;
  \item each $C_{i}$, $i=0,\dots,r-1$ is a smooth irreducible
    subvariety of $X_{i}$, has only normal
    crossings with $D_{i}$ and is contained in $Z_{i}\cap W_{i}$;
  \item $\pi _{i+1}\colon X_{i+1}\to X_{i}$ is the blow up of $X_{i}$
    along $C_{i}$, $D_{i+1}=\pi _{i+1}^{-1}(D_{i}\cup C_{i})$ is a
    divisor with only normal crossings, and $Z_{i+1}$ and $W_{i+1}$ are the strict
    transforms of $Z_{i}$ and $W_{i}$ respectively;
  \item  $Z_{r}\cap W_{r}=\emptyset$.
  \end{enumerate}
\end{thm}

\bibliographystyle{amsalpha}

\end{document}